 \documentclass[11pt]{amsart}

\usepackage{amsfonts, amsmath,amscd, amssymb, latexsym, mathrsfs, slashed, stmaryrd, verbatim, wasysym }
\usepackage[all]{xy}
\usepackage{slashed}

\newtheorem{theorem}{Theorem}
\newtheorem{corollary}[theorem]{Corollary}

\newtheorem{lemma}[theorem]{Lemma}
\newtheorem{proposition}[theorem]{Proposition}

\numberwithin{equation}{section}
\numberwithin{theorem}{section}

\renewcommand{\tfrac}[2]{\textstyle \frac{#1}{#2} \displaystyle}

\newcommand{\script}[1]{\textsc{#1}}
\newcommand{\df}[1]{\mathfrak{#1}}

\newcommand{\bhs}[1]{\mathfrak B_{#1}}

\renewcommand{\bar}{\overline}

\renewcommand{\hat}[1]{\widehat{#1}}

\newcommand{\wt}[1]{\widetilde{#1}}
\newcommand{\rest}[1]{\!\big\rvert_{#1}}

\newcommand{\Rp}{\mathbb{R}^+}

\newcommand\pa{\partial}

\newcommand{\db}{\bar{\partial}}

\newcommand\cf{cf\@. }
\newcommand\ie{i\@.e\@. }

\newcommand\eps\varepsilon
\newcommand\p\phi
\renewcommand\P\Phi
\renewcommand\det{\operatorname{det}}

\newcommand\vpsi{\mathbf{v}_\psi}
\newcommand\hpsi{\mathbf{h}_\psi}

\renewcommand\v[1]{\mathbf{v}_{ #1 } }
\newcommand\h[1]{\mathbf{h}_{ #1 } }

\newcommand{\Cl}{\mathbb{C}l}

\newcommand\fC{\operatorname{\phi-c}}
\newcommand\fD{\operatorname{\phi-hc}}

\newcommand\FC{{}^{\fC}}
\newcommand\FD{{}^{\fD}}

\newcommand\CI{{\mathcal{C}}^{\infty}}

\newcommand{\lrpar}[1]{\left( #1 \right)}
\newcommand{\lrspar}[1]{\left[ #1 \right]}
\newcommand\ang[1]{\langle #1 \rangle}
\newcommand{\lrbrac}[1]{\left\lbrace #1 \right\rbrace}

\newcommand{\abs}[1]{\lvert #1 \rvert}

\DeclareMathOperator*{\btimes}{\times} 
\DeclareMathOperator*{\bplus}{\oplus} 

\newcommand\Ch{\operatorname{Ch}}

\newcommand\coker{\operatorname{coker}}

\newcommand\diag{\operatorname{diag}}
\newcommand\Diff{\operatorname{Diff}}
\newcommand\dvol{\operatorname{dvol}}

\newcommand\ev{\mathrm{even}}
\DeclareMathOperator*{\FP}{\operatorname{FP}}
\newcommand\Gr{\operatorname{Gr}\left( \Hom \right)}
\newcommand\GrP{\operatorname{Gr}_{\psi}\left( \Hom \right)}
\newcommand\Hom{\operatorname{Hom}}
\newcommand\Id{\operatorname{Id}}
\newcommand\ind{\operatorname{ind}}
\newcommand\Ind{\operatorname{Ind}}

\newcommand\Ran{\operatorname{Ran}}

\DeclareMathOperator*{\Res}{\operatorname{Res}}
\newcommand\Ren[2]{{}^R\operatorname{ #1 }\left( #2 \right)}
\newcommand\Ric{\operatorname{Ric}}
\newcommand\RTr[1]{{}^R\operatorname{Tr}\left( #1 \right)}

\newcommand\scal{\mathrm{scal}}
\newcommand\str{\operatorname{str}}
\newcommand\Str{\operatorname{Str}}

\newcommand\Tr{\operatorname{Tr}}

\newcommand\tr{\operatorname{tr}}

\newcommand\bfN{\mathbf{N}}

\newcommand\bbA{\mathbb{A}}
\newcommand\bbB{\mathbb{B}}

\newcommand\bbE{\mathbb{E}}

\newcommand\bbH{\mathbb{H}}

\newcommand\bbQ{\mathbb{Q}}
\newcommand\bbR{\mathbb{R}}
\newcommand\bbS{\mathbb{S}}

\newcommand\bbZ{\mathbb{Z}}

\newcommand\cA{\mathcal{A}}
\newcommand\cB{\mathcal{B}}

\newcommand\cD{\mathcal{D}}

\newcommand\cH{\mathcal{H}}

\newcommand\cJ{\mathcal{J}}
\newcommand\cK{\mathcal{K}}

\newcommand\cO{\mathcal{O}}
\newcommand\cP{\mathcal{P}}

\newcommand\cR{\mathcal{R}}
\newcommand\cS{\mathcal{S}}

\newcommand\cV{\mathcal{V}}

\DeclareMathAlphabet{\mathpzc}{OT1}{pzc}{m}{it}
\newcommand{\cl}[1]{\mathpzc{cl}\left( #1 \right)}
\newcommand{\cll}{\mathpzc{cl}}

\newcommand\paperintro%
        {%
         }
\newcommand\paperbody%
        {%
         }

\title[Families index for hyperbolic cusps]{Families index for manifolds with hyperbolic cusp singularities}
\author{Pierre Albin}
\author{Fr\'ed\'eric Rochon}
\address{Department of Mathematics, Massachusetts Institute of Technology}
\email{pierre@math.mit.edu} 
\address{Department of Mathematics, University of Toronto}
\email{rochon@math.toronto.edu} 
\thanks{The first author was partially supported by a NSF postdoctoral fellowship. The second author was supported by a postdoctoral fellowship of the Fonds qu\'eb\'ecois de la recherche sur la nature et les technologies.}

\newcommand\Sec[1]{\cS_{ #1 }} 
\newcommand\Curv[1]{\hat \cR_{ #1 }} 

\begin{document}


\begin{abstract}
Manifolds with fibered hyperbolic cusp metrics include hyperbolic manifolds with cusps and locally symmetric spaces of $\bbQ$-rank one. 
We extend Vaillant's treatment of Dirac-type operators associated to these metrics by weaking the hypotheses on the boundary families through the use of Fredholm perturbations as in the family index theorem of Melrose and Piazza
and by treating the index of families of such operators.
We also extend the index theorem of Moroianu and Leichtnam-Mazzeo-Piazza to families of perturbed Dirac-type operators associated to fibered cusp metrics (sometimes known as fibered boundary metrics).
\end{abstract}

\maketitle
\tableofcontents

\paperintro \section*{Introduction}

One of the main topics in geometric analysis is the study of elliptic operators on manifolds with singularities. Among the simplest singularities are hyperbolic cusps. Indeed it is well known that these are the only singularities appearing among hyperbolic manifolds with finite volume.
A related type of singularity which we will refer to as a fibered hyperbolic cusp occurs among the natural metrics on  locally symmetric spaces of $\bbQ$ rank one. Along these cusps there are directions in which the metric degenerates but also an underlying cylindrical core which does not.
We call fibered hyperbolic cusp metrics a class of metrics whose singularities have this kind of structure but that need not have constant curvature or any underlying algebraic structure.
The goal of this article is to analyze the families index for families of Dirac-type operators associated to these metrics.

The index theory of locally symmetric spaces of $\bbQ$-rank one has been studied at length by Werner M\"uller (e.g., \cite{Muller}). 
Boris Vaillant \cite{Vaillant} introduced fibered hyperbolic cusp metrics (which he referred to as $d$-metrics) and adapted the methods of Melrose \cite{APS Book}, Mazzeo-Melrose \cite{MazzeoMelrose}, and Dai-Melrose \cite{DaiMelrose} to construct the resolvent and the heat kernel of associated Dirac-type operators.
To extend his constructions to families of Dirac-type operators, we use an
adapted version of the proof of the families index theorem on asymptotically cylindrical manifolds of Melrose-Piazza \cite{MelrosePiazza}.

Our result can be applied to the index theory of the ``fibered cusp" metrics introduced by Mazzeo and Melrose in \cite{MazzeoMelrose} where they are referred to as $\phi$-metrics. These metrics are conformally related to Vaillant's fibered hyperbolic cusp metrics and Moroianu has shown \cite{Moroianu} that for a single operator one can deduce an index theorem for Dirac-type operators from Vaillant's index theorem. A different proof via adiabatic limits was given by Leichtnam, Mazzeo, and Piazza \cite{LMP}. We show that Moroianu's proof extends to families of Dirac-type operators and so deduce a families index theorem for fibered cusp metrics from the corresponding theorem for fibered hyperbolic cusp metrics.

In a companion paper \cite{AlbinRochon}, we also investigate some interesting 
implications of our results in Teichm\"uller theory.  More specifically,
 on the Teichm\"uller universal curve associated to Riemann surfaces of
genus $g$ with $n$ punctures, various families of $\db$ operators arise naturally
and can be interpreted as Dirac type hyperbolic cusp operators. Such
families were already studied by Takhtajan and Zograf \cite{TZ}.  Using
the Selberg Zeta function to define the determinant of the Laplacian, they
were able to compute explicitly the curvature of the Quillen connection
for the determinant line bundles associated to these families of $\db$ 
operators.  This curvature corresponds to the two form part of the 
family index.  In \cite{AlbinRochon}, we use our results to get a formula for
the `full' family index, generalizing in this way the formula of Takhtajan and
Zograf.

We now describe our results precisely and the contents of the paper.

\section*{Statement of results}

The simplest example of a fibered hyperbolic cusp metric is the hyperbolic metric on the quotient of a hyperbolic space by a group of isometries containing parabolic elements.
For instance, after moding out $\bbH^2$ (thought of as the half-plane $\{x\geq 0\}$) by the action of the parabolic isometry
\begin{equation*}
	\lrpar{x,y} \mapsto \lrpar{x,y+1}
\end{equation*}
we end up with the infinite half cylinder
\begin{equation*}
	\Rp\times \bbS^1 \text{ with the metric } g = \frac{dx^2}{x^2} + \frac{d\theta^2}{x^2}.
\end{equation*}
Introducing the change of variable $r=-1/x$ the hyperbolic metric along the cylinder takes the form
\begin{equation*}
	\frac{dr^2}{r^2} + r^2 d\theta^2.
\end{equation*}
This is the model of a {\em hyperbolic cusp metric}.

Similarly, for a locally symmetric space of $\bbQ$ rank one, $X$, a neighborhood of infinity has the form 
\begin{equation*}
	M \times \bbR^+_r
\end{equation*}
with $M$ the total space of a fibration $Z - M \xrightarrow{\phi} Y$ and the natural Riemannian metric along the ends is given by
\begin{equation*}
	dr^2 + \phi^*g_Y + e^{-2Cr}g_Z
\end{equation*}
with $g_Y$ a metric on Y and $g_Z$ a metric along the fibers.
Setting $C=1$ and introducing $x=e^{-Cr}$ this metric takes the form
\begin{equation}\label{Qrank1}
	\frac{dx^2}{x^2} + \phi^*g_Y + x^2 g_Z.
\end{equation}

A metric on the interior of a manifold with boundary $X$ is called an {\bf exact fibered hyperbolic cusp metric} if the boundary is the total space of a fibration $Z - \pa X \xrightarrow{\phi} Y$ and, near the boundary, the metric is asymptotically of the form \eqref{Qrank1} (see \S\ref{sec:BhvCusps}). Fibered hyperbolic cusp objects will generally be labeled $\fD$.

Here and below $x$ denotes a {\bf boundary defining function} or `bdf', i.e., $x$ is a smooth non-negative function on $X\cup \pa X$ such that
\begin{equation*}
	\pa X = \{ x=0 \}, 
	\quad |dx|\rest{\pa X}>0.
\end{equation*}
In a coordinate chart near the boundary - with $y_i$, $z_j$ denoting coordinates along the base $Y$ of the fibration and the fibre $Z$ respectively - the set of vector fields of finite length with respect to a $\fD$ metric (denoted $\cV_{\fD}$) is locally spanned by
\begin{equation}\label{LocalBasis}
	\left\{ x\pa_x, \pa_{y_i}, \frac1x \pa_{z_j} \right\}.
\end{equation}
As explained by Melrose \cite{APS Book}, there is a vector bundle $\FD TX$ whose sections are precisely the elements of $\cV_{\fD}$.
While a vector field in the span of \eqref{LocalBasis} is singular at the boundary as a section of $TX$, it is non-singular as a section of $\FD TX$. 
Intuitively, one can think of replacing $TX$ by $\FD TX$ as part of a wider `change of category' and then Vaillant's index theorem is the usual Atiyah-Singer index theorem in this category. This viewpoint is espoused in \cite{APS Book} where it is shown that many constructions of Riemannian geometry have natural analogues for asymptotically cylindrical manifolds.

Thus a $\fD$ Clifford module is a vector bundle $E$ with an action of the bundle of Clifford algebras, $\Cl\lrpar{\FD T^*X}$ of the dual bundle to $\FD TX$, as well as a compatible Hermitian metric and connection.  
The corresponding Dirac-type operator is referred to as a $\fD$ Dirac-type operator and is denoted $\eth^E_{\fD}$ or simply $\eth_{\fD}$.

A particular case of $\fD$ metrics are the asymptotically cylindrical or `$b$' metrics. The index theorem for asymptotically cylindrical metrics was proven by Atiyah, Patodi, and Singer \cite{APS} and extended to more general $b$-metrics by Melrose \cite{APS Book}. The extension to families of Dirac operators associated to asymptotically cylindrical metrics was carried out by Bismut and Cheeger \cite{BC1, BC2} and then by Melrose and Piazza \cite{MelrosePiazza, MelrosePiazza2}. To discuss our results it is convenient to start by reviewing these.

If $X - M \xrightarrow{\psi} B$ is a fibration of manifolds with boundary then a family $\eth_b$ of Dirac-type operators associated to a family of $b$-metrics on the fibers is always elliptic, but not necessarily Fredholm. This is determined by the family of operators on the boundary fibration $\pa X - \pa M \xrightarrow{\pa \psi} B$, say $\eth_{\pa}$, induced by restricting $\eth_b$ to the boundary. Bismut and Cheeger proved, for twisted Dirac operators, that if $\eth_{\pa}$ is an invertible family then $\eth_b$ is a continuous family of Fredholm operators and hence defines an index `bundle' as an element of the K-theory of $B$. The Chern character of the index bundle, as an element in $H^{\ev}(B)$, was computed to be
\begin{equation*}
	\Ch \Ind\lrpar{\eth_b} = \int_{M/B} \hat A\lrpar{M/B}\Ch\lrpar E - \hat\eta\lrpar{ \eth_{\pa} }
\end{equation*}
where $\hat A$ is the A-roof genus, $\Ch\lrpar E$ is the Chern character of the twisting bundle, and $\hat \eta$ is the $\eta$-form invariant introduced in \cite{BC0}. (With the convention that $[\hat\eta]_{[0]}$ is equal to one-half of the numerical $\eta$-invariant.)

Melrose and Piazza were able to prove an index theorem in this context with no assumption on the boundary family by perturbing the original family to make it Fredholm. The existence of a Fredholm perturbation among families of smoothing operators was shown to be a topological condition which, by cobordism invariance of the index, always holds in this context.
A perturbed family of $b$-Dirac operators defines an index bundle in $K(B)$ (parametrized by a `spectral section' $P$)
whose Chern character is given by essentially the formula above
\begin{equation*}
	\Ch \Ind_P\lrpar{ \eth_b } = \int_{M/B} \hat A\lrpar{M/B}\Ch'\lrpar E - \hat\eta_P\lrpar{ \eth_{\pa} }
\end{equation*}
with an appropriately modified $\eta$-form. 
A particular case of interest is when $\eth_{\pa}$ is not invertible but has constant rank kernel. Then the modified $\eta$-form is equal to the Bismut-Cheeger $\eta$ form plus the Chern character of $\ker \eth_{\pa}$, a formula that was first stated in \cite{BC3}.

Consider now a family of manifolds $X - M \xrightarrow{\psi} B$ where the boundaries of the fibers are themselves the total space of a fibration $Z - \pa X \xrightarrow{\phi} Y$, so that altogether we have
\begin{equation}\label{fullfib}
	\xymatrix { & & X \ar@{-}[r] & M \ar[d]^{\psi} \\
	Z \ar@{-}[r] & \pa X \ar[d]^{\phi}  \ar@{^{(}->}[ur] \ar@{-}[r] 
		& \pa M \ar[d]^{\wt \phi}  \ar@{^{(}->}[ur]  & B \\
	& Y \ar@{-}[r] & D \ar[ur]_{\psi_D} & }
\end{equation}
As mentioned above a family of $\fD$ Dirac-type operators $\eth_{\fD}$ is necessarily singular at the boundary. 
So to understand the behavior at the boundary we first multiply by $x$ and then restrict $x\eth_{\fD}$ to the boundary, thus obtaining the boundary family of Dirac operators in this context.
This is not a family of operators on $\pa X$ but instead is a $\phi$-vertical family of operators which we denote $\eth^V_{\fD}$.
If this family of operators is invertible then the original family $\eth_{\fD}$ is Fredholm. More generally one can often find a invertible perturbation of $\eth^V_{\fD}$ that extends to a smoothing Fredholm perturbation of $\eth_{\fD}$. As discussed in \S\ref{sec:Perturbations} these are again parametrized by spectral sections, though it is necessary to restrict to spectral sections equivariant with respect to the action of $\Cl(T^*D/B)$.
We denote the resulting index bundle by $\Ind_{P^V}(\eth_{\fD})$.

One instance where this is possible is when $\eth^V_{\fD}$ is not invertible but has kernel of constant rank. 
In this case, the {\bf constant rank case}, a subtler analysis is possible. 
These kernels fit together to form a vector bundle over $D$,
\begin{equation*}
	\cK := \ker\eth^V \to D,
\end{equation*}
and once we have restricted $\eth_{\fD}$ to sections of $\cK$, the resulting operator has non-singular restriction to the boundary. In this way, we obtain a second boundary family, a $\psi_D$-vertical family of operators, which we denote $\eth^{H,\cK}_{\fD}$. Vaillant has shown that $\eth_{\fD}$ is Fredholm (as an operator on $L^2(g_{\fD})$) if and only if $\eth^{H,\cK}_{\fD}$ is an invertible family of operators. If $\eth^{H,\cK}_{\fD}$ is not invertible, but can be perturbed to be so, then the perturbations are again parametrized by spectral sections and we denote the resulting index bundle by $\Ind_{P^H}(\eth_{\fD})$.

\begin{theorem}\label{DFamiliesIndex}
Let $M$ be as in \eqref{fullfib}, $g$ a $\psi$-vertical family of exact $\fD$ metrics, $E$ an associated Clifford module over $M$, and $\eth_{\fD}$ the corresponding family of $\fD$ Dirac-type operators. Denote $\dim M/B$ by $n$ and $\dim D/B$ by $h$.

If $\eth^V$ admits an appropriate invertible perturbation, and $\Ind_{P^V}(\eth_{\fD})$ is the resulting index bundle, then in $H^{\ev}(B)$ we have
\begin{multline}\label{LargePerturbIndex}
	\Ch ( \Ind_{P^V}(\eth_{\fD}) )
	= \frac1{(2\pi i)^{n/2}} \int_{M/B} \hat A(M/B) \Ch'(E)
	\\ - \frac1{(2\pi i)^{\lfloor\frac{h+1}{2}\rfloor}} \int_{D/B} \hat A(D/B) \hat \eta_P\lrpar{ \eth^V_{\fD} }.
\end{multline}

If the rank of $\ker \eth^V$ is constant independent of the base point in $D$ and $P$ is a spectral section for 
$\eth^H := \eth_{\fD}\rest{\pa M, \; \cK}$, then in $H^{\ev}(B)$ we have
\begin{multline}\label{SmallPerturbIndex}
	\Ch(\Ind_{P^H}(\eth_{\fD}) )
	= \frac1{(2\pi i)^{n/2}} \int_{M/B} \hat A(M/B) \Ch'(E)
	\\ - \frac1{(2\pi i)^{\lfloor\frac{h+1}{2}\rfloor}} \int_{D/B} \hat A(D/B) \hat \eta\lrpar{\eth^V_{\fD}}
	- \hat\eta_P \lrpar{ \eth^{H,\cK}_{\fD} }.
\end{multline}

If the rank of $\ker \eth^V$ is constant independent of the base point in $D$ and $\eth^{H,\cK}_{\fD}$ is invertible then we have the following equality of forms on $B$, 
\begin{multline}\label{IntroUnperturbedIndex}
\Ch\lrpar{ \Ind\lrpar{\eth_{\fD}}, \nabla^{\Ind} } \\
	= \frac1{(2\pi i)^{n/2}}  \int_{M/B} \hat A\lrpar{M/B}  \Ch'\lrpar E 
	- \frac1{(2\pi i)^{\lfloor\frac{h+1}{2}\rfloor}} \int_{D/B} \hat A\lrpar{D/B}\hat\eta^+\lrpar{\eth^V,E} \\
	- \hat\eta\lrpar{\eth^{H,\cK}_{\fD}} 
	- d \int_0^{\infty} \Str\lrpar{  \frac{\pa \bbA_{\fD}^t }{\pa t}  e^{-\lrpar{\bbA_{\fD}^t}^2} }.
\end{multline}
\end{theorem}
Since we use the standard sign convention for eta forms instead of the sign 
convention of Vaillant, notice that we get an extra minus when we compare the
degree zero part of the formula with the formula of Vaillant.  Moreover, when $h$ is even,
the second term on the right hand side of \eqref{LargePerturbIndex},\eqref{SmallPerturbIndex} and \eqref{IntroUnperturbedIndex} differs from the corresponding term in the formula
of Vaillant \cite{Vaillant} by a factor of $-\sqrt{2\pi i}$.  This is because for $h$ even,
the eta form considered by Vaillant is $-\sqrt{2\pi i}$ times the standard eta form of
Bismut and Cheeger \cite{BC0}.  This extra factor in the definition of Vaillant is due in part
to the his unusual definition of the relative supertrace (see equation (111) in \cite{Vaillant}). 

Our proof follows the usual strategy, after \cite{Bismut}, used for families of Dirac-type operators.
That is, 
introducing the rescaled Bismut superconnection associated to the family of operators and the Levi-Civita connection of the metric, say $\bbA_t$, and comparing the behavior of its Chern character
\begin{equation*}
	\Ch\lrpar{\bbA_t} = \Str\lrpar{ e^{-\bbA_t^2} }
\end{equation*}
as the parameter $t$ goes to infinity and zero.
However, for a family of $\fD$ Dirac operators, the construction of the heat kernel is complicated by the singularity of the relevant superconnection $\bbA^2_t$ at the boundary. Furthermore, after the heat kernel is constructed it turns out not to be of trace class.
We deal with the first difficulty by using Vaillant's heat calculus adapted to these operators and with the second by introducing a renormalized Chern character.
Indeed our proof of Theorem \ref{DFamiliesIndex} follows the proof in \cite{MelrosePiazza} of the index of families of $b$-Dirac operators replacing the use of the $b$-calculus with Vaillant's treatment of $\fD$ operators (note that for technical reasons there is no $\fD$ calculus).

The paper is organized as follows.
In \S 2 we describe the geometry of families of $\fD$ metrics as a generalization of Vaillant's thesis which for this purpose is reviewed in \S\ref{sec:Vaillant}.
Section 3 introduces the $\fD$ Bismut superconnection and explains how, by a suitable rescaling, it contains the geometric data of the boundary fibrations. The proof of the $\fD$ families index theorem is in \S 4, followed by the derivation of the $\fC$ families index theorem in \S 5.

\begin{equation*}
	\text{ {\bf Acknowledgements} }
\end{equation*}
The authors are happy to acknowledge helpful conversations on these topics with Rafe Mazzeo, Richard Melrose, Sergiu Moroianu, Leon Takhtajan, and Peter Zograf.

\paperbody 
\section{Review of Vaillant's thesis} \label{sec:Vaillant}

Our paper relies heavily on the thesis of Boris Vaillant \cite{Vaillant}, who
proves among other things an index theorem for Dirac-type operators associated 
to fibred cusp hyperbolic metrics.  For this reason, we intend in this section
to review the main steps and constructions involved in the proof
of this theorem.

\subsection{Behavior near the cusps}\label{sec:BhvCusps} $ $ \newline
We start out with a manifold $X$ whose boundary is the total space of a fibration
\begin{equation*}
	Z - \pa X \xrightarrow{\phi} Y.
\end{equation*}
Let $x$ be a boundary defining function, i.e., a smooth non-negative function on $\bar X$ such that $\pa X = \{x=0\}$ and $dx$ is never equal to zero at the boundary.
A complication inherent in dealing with $\FD TX$ (which is defined below) is that its sections do not form a Lie algebra. For this reason, it is convenient to work with the closely related Lie algebra
\begin{equation*}
	\cV_{\fC}(X) := 
	\lrbrac{ V \in \Gamma\lrpar{TX} : V\rest{\pa X} \in 
      \Gamma\lrpar{T(\pa X/Y)}, Vx = \cO\lrpar{x^2} }
\end{equation*}
introduced by Mazzeo and Melrose \cite{MazzeoMelrose}, and define
\begin{equation*}
	\cV_{\fD}(X) :=
	\lrbrac{ \frac1x V : V \in \cV_{\fC} }.
\end{equation*}
Then the bundles $\FC TX$ and $\FD TX$ are defined by specifying their sections (cf. \cite[Chapter 8]{APS Book})
\begin{equation}\label{BdlesDef}
	\Gamma\lrpar{ \FC TX } = \cV_{\fC}(X), \quad
	\Gamma\lrpar{ \FD TX } = \cV_{\fD}(X).
\end{equation}
This is characteristic of the methods in \cite{Vaillant} and below, we study $\fD$ objects by relating them to the generally better behaved $\fC$ objects.

A fibre metric on the bundle $\FD TX$ is known as a {\bf $\fD$ metric}.
However, in practice, it will be necessary to restrict the behavior of the metric near the boundary.
The most restrictive condition is to fix a collar neighborhood of the boundary, 
$C\lrpar{\pa X}\cong [0,\eps)_x \times \pa X$, extend the fibration to $C\lrpar{\pa X}$
\begin{equation*}
	[0,\eps)_x \times \pa X \xrightarrow{\phi} [0,\eps)_x \times Y,
\end{equation*}
choose a connection for $\phi$, and demand that in this neighborhood the metric takes the form
\begin{equation}\label{ProductType}
	\frac{dx^2}{x^2} + \phi^*g_Y + x^2 g_{\pa X/Y}
\end{equation}
where $g_Y$ (a metric on $Y$) and $g_{\pa X/Y}$ (a symmetric two-tensor that restricts to a Riemannian metric on each fiber of $\phi$) are such that $\phi^*g_Y + g_{\pa X/Y}$ is a Riemannian submersion metric for $\pa X$. Such a metric is known as a {\bf product-type $\fD$ metric}. While these metrics are easy to deal with computationally, they are not preserved by all changes of coordinates in the interior.

Instead we work with a class of metrics that are asymptotically like \eqref{ProductType}.
An {\bf exact $\fD$ metric} is a fibre metric on $\FD TX$ that near the boundary takes the form
\begin{equation}\label{fDMetric}
	g_{\fD}= \frac{dx^2}{x^2} + h + x^2 k = g_{pt} + x\wt g
\end{equation}
where $g_{pt}$ is a product-type metric and $\wt g \in \CI(M, S^2\; \FD T^*M/B)$. 

{\bf Remark.} This definition of exact $\fD$ metric, consistent with Moroianu \cite{Moroianu}, is different from that used by Vaillant. However the differences are minor (Vaillant has $g_{\fD}(x\pa_x, x\pa_x) = 1 + \cO(x^2)$ and $g_{\fD}(x\pa_x, \frac1x\pa_z) = \cO(x^2)$ while Moroianu has respectively $1 + \cO(x)$ and $\cO(x)$), and do not affect the analyis below.

If the boundary fibration has $Z=\{pt\}$ and $\phi=\mathrm{id}$, then the $\fD$ objects (vector fields, metrics, etc.) are known as $b$ objects. The geometry and index theory of $b$-metrics was worked out by Melrose in \cite{APS Book} and the family index theorem by Melrose and Piazza \cite{MelrosePiazza, MelrosePiazza2}. We will occasionally make use of $b$-vector fields and metrics.

Recall that the geometry of a submersion metric, such as 
\begin{equation*}
	g_{\pa X}:=(h + k)\rest{\pa X},
\end{equation*}
is described by two tensors, the second fundamental form of the fibers $\Sec{\phi}$ and the curvature of the fibration $\Curv{\phi}$. Indeed, the metric determines a complement to the space $V\pa X$ of vertical vector fields, which we denote $H\pa X$, and these two tensors describe the failure of the Levi-Civita connection to preserve the splitting
\begin{equation}\label{BdySplitting}
	T\pa X = V\pa X \oplus H\pa X.
\end{equation}
In terms of the Levi-Civita connection $\nabla^{\pa X}$ of $g_{\pa X}$ the second fundamental form of the fibers is given by
\begin{equation}\label{DefSec}
\begin{gathered}
	\Sec{\phi} \in \CI\lrpar{\pa X; \lrpar{V^*\pa X}^2 \otimes H^*\pa X} \\
	\Sec{\phi}\lrpar{V_1,V_2}\lrpar H
	:= g_{\pa X}\lrpar{ \nabla^{\pa X}_{V_1}V_2, H}, 
	\quad V_i \in V\pa X, H\in H\pa X,
\end{gathered}
\end{equation}
and the curvature of the fibration by
\begin{equation}\label{DefCurv}
\begin{gathered}
	\Curv{\phi} \in \CI\lrpar{\pa X; \lrpar{H^*\pa X}^2 \otimes V^*\pa X} \\
	\Curv{\phi}\lrpar{H_1, H_2}\lrpar V
	:= g_{\pa X}\lrpar{ \lrspar{H_1, H_2}, V},
	\quad H_i \in H\pa X, V \in V\pa X.
\end{gathered}
\end{equation}
It is clear from their definitions that these tensors are respectively symmetric and anti-symmetric in their  first two arguments.
Conversely we can describe the Levi-Civita connection of $g_{\pa X}$ in terms of these two tensors, (see for instance \cite{HHM})
\begin{equation*}
\begin{tabular}{|c||c|c|} \hline 
$g_{\pa X}\lrpar{\nabla^{\pa X}_{X_1} X_2, Y}$ & $V_0$ & $H_0$ \\ \hline\hline
$\nabla^{\pa X}_{V_1}V_2$ & 
	$k\lrpar{\nabla^Z_{ V_1} V_2, V_0}$ & 
	$\Sec{\phi}\lrpar{ V_1, V_2}\lrpar{H_0}$ \\ \hline
$\nabla^{\pa X}_{H} V$ & 
	$k\lrpar{\lrspar{H, V}, V_0} - \Sec{\phi}\lrpar{ V,  V_0} \lrpar{H}$ &
	$-\frac12\Curv{\phi}\lrpar{H, H_0}\lrpar{V}$ \\ \hline
$\nabla^{\pa X}_{ V} H$ &
	$-\Sec{\phi}\lrpar{V, V_0}\lrpar{H}$ &
	$\frac12\Curv{\phi}\lrpar{H,H_0}\lrpar{V}$ \\ \hline
$\nabla^{\pa X}_{H_1}H_2$ &
	$\frac12\Curv{\phi}\lrpar{H_1,H_2}\lrpar{V_0}$ &
	$h\lrpar{\nabla^Y_{H_1}H_2, H_0}$ \\ \hline
\end{tabular}
\end{equation*}
where $V_i$ and $H_i$ are respectively vertical and horizontal vector fields.
We will often denote the ortohogonal projection onto the first summand of \eqref{BdySplitting} by $\v\phi$ and the complementary projection onto $H\pa X$ by $\h\phi$.

Understanding the asymptotics of a fibered hyperbolic cusp metric means understanding how the geometry of the fibration at the boundary, as encoded in $\Sec\phi$ and $\Curv\phi$, relates to the geometry of the metric $g_{\fD}$.
To study this problem we need to define horizontal and vertical vector fields in $\FD TX$. 
Consider the inclusion map $\FC TX \xrightarrow{j} TX$ restricted to the boundary
\begin{equation*}
	0 \to \FC N\pa X \to \FC TX\rest{\pa X} \xrightarrow{j} \FC V\pa X \to 0
\end{equation*}
and choose a splitting of this exact sequence, $\FC TX\rest{\pa X} = \FC N\pa X \oplus \FC V\pa X$.
For a product-type metric these are easily identified as $\FC N\pa X = \ang{ x^{2}\pa_x, x\pa_y }$, $\FC V\pa X = V\pa X$ and extended into $C(\pa X)$.
Note that although $\FC V\pa X$ can be identified with $V\pa X$, it is {\em a priori} a quotient and not a sub-bundle of $\FC TX\rest{\pa X}$.
Define the sub-bundles of asymptotically horizontal and asymptotically vertical vector fields, respectively by
\begin{gather*}
	\Gamma\lrpar{ \lrspar{ \FC N\pa X } }
	:= \lrbrac{ W \in \Gamma\lrpar{ \FC TX } : W\rest{\pa X} \in \Gamma\lrpar{ \FC N\pa X } } \\
	\Gamma\lrpar{ \lrspar{ \FC V\pa X } }
	:= \lrbrac{ W \in \Gamma\lrpar{ \FC TX } : W\rest{\pa X} \in \Gamma\lrpar{ \FC V\pa X } }
\end{gather*}
where the restriction to the boundary is as an element of $\FC TX$, and finally define the corresponding $\fD$ bundles
\begin{equation*}
	\Gamma\lrpar{ \lrspar{ \FD N\pa X } }
	= \frac1x \Gamma\lrpar{ \lrspar{ \FC N\pa X } }, \quad
	\Gamma\lrpar{ \lrspar{ \FD V\pa X } }
	= \frac1x \Gamma\lrpar{ \lrspar{ \FC V\pa X } }.
\end{equation*}

Given a $\fC$ vector field, $W$, the Kozul formula applied to an arbitrary $\fC$ metric defines a map
\begin{equation}\label{fCconnection}
	\nabla^{\fC}_W : \Gamma\lrpar{ \FC TX} \to \Gamma\lrpar{ \FC TX }
\end{equation}
which satisfies the Leibnitz rule. However this is not quite a connection since the vector field $W$ must itself be a $\fC$ vector field. 
One of the benefits of restricting to exact $\fC$ metrics is that then the map \eqref{fCconnection} is defined for any vector field $W$ and so defines an actual connection on $\FC TX$. The analogous statement is true for exact $\fD$ metrics.

\begin{lemma}[Vaillant \cite{Vaillant}, \S1]
The Levi-Civita connection of an exact $\fD$ metric is a connection on the $\fD$ tangent bundle
\begin{equation*}
	\nabla^{\fD} : \Gamma\lrpar{ \FD TX } \to 
	\Gamma\lrpar{ T^*X \otimes \FD TX }.
\end{equation*}
In contrast to $\nabla^{\pa X}$, this connection preserves asymptotically vertical and horizontal vector fields
\begin{gather*}
	\Gamma\lrpar{ \lrspar{ \FD N\pa X } }
	\xrightarrow{ \nabla^{\fD}_W }
	\Gamma\lrpar{ \lrspar{ \FD N\pa X } }, \label{PreserveHor}\\
	\Gamma\lrpar{ \lrspar{ \FD V\pa X } }
	\xrightarrow{ \nabla^{\fD}_W }
	\Gamma\lrpar{ \lrspar{ \FD V\pa X } }
\end{gather*}
whenever $W \in \Gamma( TX )$ is tangent to the boundary.
\end{lemma}

We can be more precise about the asymptotic behavior of $\nabla^{\fD}$. 
Fix a collar neighborhood of the boundary, $C(\pa X)$, in which the metric has the form \eqref{fDMetric}.
Corresponding to the choice of product structure there is a natural extension of any vector field on the boundary to a vector field on $C(\pa X)$.
Define the $\fD$ second fundamental form of the fibers by
\begin{gather}
	\Sec{\fD} \in \CI\lrpar{ \pa X; \lrpar{V^*\pa X}^2 \otimes (\FD N\pa X)^* } 
	,\quad V_i \in V\pa X, H\in \FD N\pa X, \notag \\
	\Sec{\fD}\lrpar{V_1,V_2}\lrpar H
	:= g_{\fD}\lrpar{ \nabla^{\fD}_{\tfrac 1x V_1} \tfrac1x V_2, H}\rest{\pa X},
\label{DefFDSec}
\end{gather}
and the $\fD$ curvature of the fibration by
\begin{gather}
	\Curv{\fD} \in \CI\lrpar{\pa X; \lrpar{(N\pa X)^*}^2 \otimes V^*\pa X},
	\quad H_i \in \FD N\pa X, V \in V\pa X \notag \\
	\Curv{\fD}\lrpar{H_1, H_2}\lrpar V
	:= g_{\fD}\lrpar{ \lrspar{H_1, H_2}, \tfrac1x V}\rest{\pa X}.
\label{DefFDCurv}
\end{gather}
It is easy to relate these tensors with \eqref{DefSec} and \eqref{DefCurv}. Indeed, we have
\begin{equation}\label{TwoSecs}
\begin{gathered}
	\Sec{\fD}\lrpar{V_1, V_2}\lrpar H
	= \Sec{\phi}\lrpar{V_1, V_2}\lrpar H - \tfrac{dx}x  (H)  k(V_1, V_2) \\
	+
	\frac12 \left[
	V_1 \wt g(H, \tfrac1x  V_2)
	+V_2 \wt g(H, \tfrac1x  V_1)
	+ \wt g(H, \tfrac1x  [V_1,V_2]) \right]\rest{\pa X}
\end{gathered}
\end{equation}
and
\begin{equation}\label{CurvfD=x}
	\Curv{\fD}\lrpar{H_1, H_2}\lrpar V
	= x\Curv{\phi}\lrpar{H_1, H_2}\lrpar V.
\end{equation}
Notice that the last three summands in \eqref{TwoSecs} vanish if the metric is a product-type metric up to second order at the boundary.

Another point of view is adopted in \cite[\S 1.4]{Vaillant}. 
Instead of using a fixed product structure on a collar neighborhood of the boundary, pick a vector field $N$ transverse to the boundary and use it to extend $\fC$ vector fields off of the boundary as follows: given $A \in \FC TX$ define $\bar A \in \FC TX$ so that 
\begin{equation*}
\begin{cases}
\nabla^{\fD}_N \frac1x\bar A=0 \\
\bar A\rest{\pa X} = A\rest{\pa X} \text{ as elements of } \FC TX.
\end{cases}
\end{equation*}
For the asymptotics of the connection this is essentially what we did in the previous paragraph, and it allows us to identify the last three summands in \eqref{TwoSecs} with the tensor
\begin{equation}\label{DefFDB}
	\cB_{\fD}\lrpar{V_1, V_2}\lrpar H
	:=\frac12 g_{\fD}\lrpar{ \tfrac1{x^2}\lrpar{ [\bar V_1, \bar V_2] - \bar{ [ V_1,  V_2] } }, \tfrac1x \bar{xH} },
\end{equation}
i.e., we have \cite[Lemma 1.13]{Vaillant}
\begin{multline}\label{TwoSecs'}
	\Sec{\fD}\lrpar{V_1, V_2}\lrpar H \\
	= \Sec{\phi}\lrpar{V_1, V_2}\lrpar H - \tfrac{dx}x(H)  k(V_1, V_2)
	+ \cB_{\fD}\lrpar{ V_1, V_2}(H).
\end{multline}
The trick of passing from $A$ to $\bar A$ is particularly useful for computing the asymptotics of the curvature $R^{\fD}$ of $g_{\fD}$.

\begin{lemma}[Vaillant \cite{Vaillant}, Proposition 1.14] \label{ConnectionAsymp}
Let $N \in \Gamma(TX)$ be transverse to the boundary and let $W \in \Gamma(\FC TX)$.
Then, at the boundary, if $A \in \Gamma( [\FD V\pa X] )$ and $B \in \Gamma( [\FD N\pa X] )$
\begin{equation}\label{Prop1.14a}
	g_{\fD}\lrpar{ R^{\fD}\lrpar{N,W}A, B}\rest{\pa X}
	= dx(N) \Sec{\fD}\lrpar{ \v{\fC} W, xA}(B).
\end{equation}
If $A, B \in \Gamma( [\FD N\pa X] )$ then
$g_{\fD}\lrpar{ R^{\fD}\lrpar{N,W}A, B} = \cO(x)$, and 
\begin{multline}\label{Prop1.14b}
	g_{\fD}\lrpar{ (\nabla_N^{\fD} R^{\fD}) \lrpar{N,W}A, B}\rest{\pa X} \\
	=\lrpar{ dx(N) }^2 \Curv{\phi}\lrpar{A, B}( \v{\fC}W ).
\end{multline}
\end{lemma}

We do note give expressions for the other asymptotics of the curvature, as these are the ones we will need.

\begin{proof}
Notice how passing from $A$ to $\bar A$ (and so on) reduces the problem to the asymptotics of the connection: 
\begin{gather*}
	g_{\fD}\lrpar{ R^{\fD}\lrpar{N,W}A, B}\rest{\pa X}
	=
	g_{\fD}\lrpar{ R^{\fD}\lrpar{N,\bar W} \frac1x \bar{xA}, \frac1x \bar{xB}}\rest{\pa X} \\
	=
	N g_{\fD}\lrpar{ \nabla^{\fD}_{\bar W} \frac1x \bar{xA}, \frac1x\bar{xB}}\rest{\pa X}
	=
	N(x) g_{\fD}\lrpar{ \nabla^{\fD}_{\frac1x \bar W} \frac1x \bar{xA}, \frac1x\bar{xB}}\rest{\pa X}\\
	= 
	dx(N)\Sec{\fD} \lrpar{\v{\fD} W, xA}(B)
\end{gather*}
which proves both \eqref{Prop1.14a} and that $g_{\fD}\lrpar{ R^{\fD}\lrpar{N,W}A, B} = \cO(x)$ whenever $A, B \in \Gamma( [\FD N\pa X] )$. Vaillant's proof of \eqref{Prop1.14b} is through a similar though more involved computation.
\end{proof}

\subsection{Fredholm criteria for Dirac-type operators} \label{sec:DiracTypeOps} $ $ \newline
A {\bf $\fC$ differential operator} of order $k$ is an element of the enveloping algebra of $\cV_{\fC}$, that is, a differential operator that in local coordinates near the boundary has an expression of the form
\begin{equation*}
	\sum_{j+|\alpha|+|\beta| \leq k} a_{j,\alpha,\beta}(x,y,z) (x^2\pa_x)^j (x\pa_y)^\alpha (\pa_z)^\beta.
\end{equation*}
A differential operator acting between sections of two vector bundles $E$ and $F$ has locally the same form but with $a_{j,\alpha,\beta}$ valued in the homomorphism bundle between $E$ and $F$.
A differential operator, $P$, is said to be a {\bf $\fD$ differential operator} of order $k$ if $x^kP$ is a $\fC$ differential operator of order $k$.

A {\bf $\fD$ Clifford module}, $E \to X$, is a Hermitian vector bundle over $X$ with a non-trivial smooth orthogonal fibre action of 
\begin{equation*}
	\Cl_{\fD}\lrpar X := \Cl\lrpar{\FD T^*X},
\end{equation*}
which we will denote $\cll$, and a unitary Hermitian connection, $\nabla^E$, which distributes over the Clifford action compatibly with the connection $\nabla^{\fD}$ on $\FD TX$.
We also assume that $X$ is even dimensional so that $\Cl_{\fD}(X)$ is $\bbZ/2$-graded and that $E$ has a compatible $\bbZ/2$-grading
\begin{equation*}
	E = E^+ \oplus E^- .
\end{equation*}

This information determines an associated {\bf $\fD$ Dirac-type operator} by
\begin{equation*}
	D_{\fD}:\CI\lrpar{X;E} \xrightarrow{\nabla^E} \CI\lrpar{X;\FD T^*X\otimes E}
	\xrightarrow{\cll} \CI\lrpar{X;E}.
\end{equation*}
This is an element of $\Diff^1_{\fD}\lrpar{M,E}$, i.e., $x\eth_{\fD}$ is a $\fC$ differential operator of order one.
The operator $D_{\fD}$ acting on $L^2_{\fD}(X,E)$ is unitarily equivalent to the operator
\begin{equation*}
	\eth_{\fD} = D_{\fD} - \tfrac v2 \cl{\tfrac{dx}x}
\end{equation*} 
acting on the corresponding $b$ space, $L^2_b(X,E) = x^{-v/2}L^2_{\fD}(X,E)$ (in particular, they have the same index). 
It is convenient to work on this space in order to make use of the usual constructions.

The operators $\eth_{\fD}$ are always elliptic, whether or not they are Fredholm is determined by two model operators. The first of these is a family of Dirac operators on the fibers of the boundary fibration $\phi$. This operator, denoted $\eth^V_{\fD}$, is obtained by restricting $x\eth_{\fD}$ to the boundary,
\begin{equation*}
	\eth^V_{\fD} := x\eth_{\fD}\rest{\pa X} 
	\in \Diff^1\lrpar{ \pa X/Y; E}.
\end{equation*}
The dimension of the null spaces of these operators will generally vary with the base point in $Y$. If it does not, we say that $\eth_{\fD}$ satisfies the {\bf constant rank assumption}. Vaillant makes this assumption and denotes by $\cK$ the vector bundle over $Y$ whose fibers are these null spaces.

It is easy to see - using the characterization of Fredholm operators in the $\fC$ calculus from \cite{MazzeoMelrose} - that $x\eth_{\fD}$ is Fredholm as a $\fC$ operator if and only if $\eth^V_{\fD}$ is invertible and we will refer to this possibility as the {\bf invertible assumption}. However, it is possible for $\eth_{\fD}$ to be Fredholm even though $x\eth_{\fD}$ is not. Under the constant rank assumption there is a second model operator, which we will denote $\eth^{H,\cK}_{\fD}$, that determines whether or not $\eth_{\fD}$ is Fredholm. This operator is obtained by restricting $\eth_{\fD}$ to sections of $\cK$ over the boundary, i.e,
\begin{equation*}
	\eth^{H,\cK}_{\fD} := \eth_{\fD}\rest{\pa X, \cK}
	\in \Diff^1\lrpar{ Y; \cK }.
\end{equation*}

\begin{theorem}[Vaillant \cite{Vaillant}, \S3]
Under the constant rank assumption, the operator $\eth_{\fD}$ is Fredholm as an operator on $L^2_{\fD}(X;E)$ if and only if $\eth^{H,\cK}_{\fD}$ is invertible.
\end{theorem}

Vaillant proves this theorem (and a lot more besides) by constructing a parametrix for $\eth_{\fD}$ and showing that it can be meromorphically continued to a Riemann surface determined by the spectrum of $\eth^{H,\cK}_{\fD}$. The construction uses the constant rank assumption and allows Vaillant to describe the spectral measure at zero in the non-Fredholm case.

Before passing to the next section we recall Vaillant's description of $\eth^{H,\cK}_{\fD}$ and its relation to a natural Dirac operator on $\cK$ (see \cite[$\S$3.5]{Vaillant}).
Since the following computation takes place at the boundary, we can assume without loss of generality that the Dirac operator is represented by
\begin{equation*}
	\eth_{\fD} = \cl{\bar\theta^i} \nabla^E_{\bar X_i}
\end{equation*}
where the $\fD$ vector fields $\bar X_i$ satisfy $\nabla_{dx^\flat}^E \bar X_i = 0$ with $dx^{\flat}$ the vector field dual to $dx$. 
Then we have
\begin{equation}\label{HorizontalComp}
\begin{split}
	\eth^{H,\cK}_{\fD}
	&= \cP_\cK \eth_{\fD}
	= \cP_\cK \nabla_{dx^\flat}^E (x\eth_{\fD}) 
	= \cP_\cK [\nabla_{dx^\flat}^E, x\eth_{\fD}] \\
	&= \cP_\cK \cl{\nabla_{dx^\flat} \bar \theta^i} \nabla_{x\bar X_i}^E
	+ \cP_\cK \cl{\bar\theta^i} [\nabla_{dx^\flat}^E, \nabla_{x\bar X_i}^E] \\
	&= \cP_\cK \cl{\bar\theta^i} \lrspar{ K_E(dx^\flat, x\bar X_i) + \nabla^E_{[dx^\flat, x\bar X_i]} } \\
	&= \cP_\cK \cl{\bar\theta^i} 
	\lrspar{ K_E(dx^\flat, x\bar X_i) + \nabla^E_{\bar X_i - x\nabla_{\bar X_i}^{\fD} dx^\flat} },
\end{split}
\end{equation}
and we will denote $\bar X_i - x\nabla_{\bar X_i}^{\fD} dx^\flat$ by $- x\nabla_{\bar X_i}^{\fC} dx^\flat$ (Vaillant points out that this equals $\bar X_i$ if $X_i \in [\FD N\pa X]$ and $dx(X_i)=0$). It is convenient to reinterpret this result after introducing notation to organize the geometry.

First recall that, since $E$ is a Clifford module, its curvature $K_E$ decomposes with respect to the action of Clifford multiplication into
\begin{equation}\label{CurvSplitting}
	K_E(\cdot, \cdot)
	= \frac12 \cl{R^{\fD}(\cdot,\cdot)}
	+ K_E'(\cdot, \cdot)
\end{equation}
where $K_E'$, known as the twisting curvature of $E$, takes values in the homomorphisms of $E$ that commute with Clifford multiplication.

The $\fD$ metric $g_X$ induces a $b$-metric, $g_{Y_+}$ on $Y_+ := Y \times \bbR^+_x$ by $\frac{dx^2}{x^2} + \phi^*g_Y$ in the notation of \eqref{ProductType}. The bundle $\cK$ pulls-back from $Y$ to a bundle over $Y^+$, it inherits a Hermitian metric from the one on $E$ by
\begin{equation}\label{MetricK}
	\ang{\xi,\eta}_\cK := \int_{\pa X/Y} \ang{\xi,\eta}_E \; \dvol_{\pa X/Y}
\end{equation}
where the integral stands for push-forward along $\phi$, and $\dvol_{\pa X/Y}$ is induced by the $\fC$ metric on the fibers $x^{-2}g_X$, and it inherits a Clifford action by projecting the Clifford action on $E$, $\cll_{\cK}:=\cP_\cK \h\phi \cll$.
To define a connection $\nabla^{\cK}$ compatible with $\ang{\cdot,\cdot}_{\cK}$, notice that for a vector field $U$ on $Y^+$ we have
\begin{equation*}
\begin{split}
	U\lrpar{\dvol_{\pa X/Y}}
	&=U \lrpar{\sqrt{\det g_{\pa X/Y}}} \theta^1\wedge\cdots\wedge\theta^v 
	\\ &= \frac12 \tr\lrpar{ g_{\pa X/Y}^{-1} (Ug_{\pa X/Y}) } \dvol_{\pa X/Y}
	\\ &= \sum_{V_i \in V\lrpar{\pa X/Y}} \ang{\nabla^{\pa X/Y}_UV_i,V_i} \dvol_{\pa X/Y} 
	\\ &= \tr\lrpar{ S_{\wt\phi}\lrpar{\cdot,\cdot}\lrpar U } \dvol_{\pa X/Y}
\end{split}
\end{equation*}
whence we define
\begin{equation}\label{NablaK}
	\nabla^{\cK}_U:=\cP_\cK\lrpar{\nabla^E_{\wt\phi^*U} 
		+ \frac12 \tr\lrpar{ S_{\wt\phi}\lrpar{\cdot,\cdot}\lrpar{\wt\phi^*U} }},
\end{equation}
and obtain a unitary connection that makes $\cK$ a Clifford module (cf. \cite[(4.24)]{BC0}).
If both the family of metrics $g_X$ and the family of metrics on $E$ are of product-type near the boundary then the computation above for $\eth^{H,\cK}_{\fD}$ collapses to $\eth^{\cK}:=\cll_{\cK}(\bar\theta^i) \nabla^{\cK}_{\bar X_i}$ with $X_i$ $\phi$-horizontal, i.e., to the Dirac-type operator on $\cK$. However in general there is an extra endomorphism on $\cK$, $\cJ$, such that \cite[Prop. 3.15]{Vaillant}
\begin{equation}\label{HorizontalNorm}
	\eth^{H,\cK}_{\fD} = \cll_\cK \circ \nabla^{\cK} + \cJ
\end{equation}
where $\cJ$ is given by 
\begin{multline*}
	\cJ := 
	\cP_\cK\left(
	\v\phi\cl{K_E'\lrpar{dx^\flat,\cdot}} 
	+\frac12\v\phi \cll_{\fD}\lrpar{\Ric^V\lrpar{dx^\flat,\cdot}}
	\right. \\ \left.
	+\frac14\cll_{\fD}\lrpar{ \cB_{\fD}\lrpar{x\cdot,x\cdot}\lrpar{\cdot}} 
	-{\v\phi\cl{\bar\theta^i}}\nabla^E_{x\nabla^{\fC}_{\bar X_i} dx^\flat}
	\right)
\end{multline*}
with $\Ric^V$ the `vertical' Ricci curvature (i.e., contraction of the curvature by vertical vector fields) and $\cB_{\fD}$ defined above \eqref{DefFDB}. (This follows from \eqref{HorizontalComp} and \eqref{CurvSplitting} by applying Lemma \ref{ConnectionAsymp}.) Hence $\eth^{H,\cK}_{\fD}$ is a zero${}^{\text{th}}$ order perturbation of the Dirac-type $b$-operator, $\eth^{\cK}$.

\subsection{The heat kernel of a $\fD$ Dirac-type operator} \label{sec:HeatKerOne}

Let $\eth_{\fD}$ be a $\fD$ Dirac-type operator acting on the bundle $E$ and satisfying the constant rank assumption.
As the first step in the proof of the index theorem for $\eth_{\fD}$, Vaillant constructs its heat kernel, $e^{-t\eth_{\fD}^2}$, by which we mean the solution to the equation
\begin{equation}\label{HeatEq}
	\begin{cases} ( \pa_t + \eth^2_{\fD} )H_t = 0 \\ H_{t=0} =\Id \end{cases}.
\end{equation}

The integral kernel of the heat kernel naturally lives as a distribution on the {\bf $\fD$ heat space} which we denote $HX_{\fD}$. This is a space obtained from $X^2 \times \bbR^+$ by a sequence of blow-ups. Let us motivate the structure of this space by considering how one would construct the heat kernel.

First, away from the boundary of the manifold, the heat kernel should resemble the heat kernel of the Euclidean space, 
\begin{equation*}
	\frac1{ (2\pi i)^{n/2} } e^{- \frac{|\zeta-\zeta'|^2}{2t} },
\end{equation*}
so clearly the interesting behavior occurs when $t\to0$ along the diagonal. To capture this we introduce polar coordinates along $\{\zeta = \zeta', t=0\}$ parabolic with respect to $t$ (i.e., $t$ should scale like $\zeta^2$), for instance we could take
\begin{equation*}
	\rho := \lrspar{ |\zeta - \zeta'|^4 + t^2}^{1/4}, \quad
	\hat\zeta := \frac\zeta\rho, \quad
	\hat\zeta' := \frac{\zeta'}\rho, \quad
	\hat t := \frac t{\rho^2}.
\end{equation*}
Geometrically this corresponds to starting with the manifold $X^2 \times \bbR^+$ and replacing the submanifold $\{\zeta = \zeta', t=0\}$ with a new boundary face (which we will denote $\bhs{00,2}$)
 on which $\rho, \hat\zeta, \hat\zeta', \hat t$ are local coordinates. 
The resulting manifold is denoted
\begin{equation*}
	\lrspar{ X^2 \times \bbR^+; \{\zeta=\zeta', t=0 \}, dt}
\end{equation*}
where the last `$dt$' records the parabolic direction (see \cite{DaiMelrose}, \cite{APS Book} for details) and comes with a natural blow-down map back to the original manifold,
\begin{equation*}
	\lrspar{ X^2 \times \bbR^+; \{\zeta=\zeta', t=0 \}, dt}
	\xrightarrow{\beta} X^2 \times \bbR^+ .
\end{equation*}
The operator $t\pa_t + t\eth^2_{\fD}$ lifts to the blown-up space where it is tangent to the new face (this is why we have multiplied the operator in \eqref{HeatEq} by $t$). Because of this, it determines a `model problem' at this face whose solution is precisely the Euclidean heat kernel (frozen at $t=1$ as in \cite[Chapter 7]{APS Book}).

Knowing the solution away from the boundary at time $t=0$, consider next the behavior of the operator $t\pa_t + t\eth^2_{\fD}$ as we approach the boundary. Because of the degeneracy of the metric at the boundary in the fiber directions,
the operator $\eth^2_{\fD}$ has a singularity of order two at the boundary. To understand what is happening near the submanifold
\begin{equation*}
	\diag_\phi \times \{0\} = \{ x=x'=0, y=y', t=0 \}
\end{equation*}
we again introduce polar coordinates around this submanifold, e.g., a rescaled time variable $\tau = t/x^2$. Geometrically, we blow-up $\diag_\phi \times \{0\}$ and lift $t\pa_t + t\eth^2_{\fD}$ to an operator on the resulting space. 
This space has a new boundary face $\bhs{\phi\phi,2}$ on which there is a new model problem. 
Near this face the lift of $t$ is $\tau x^2$ so the restriction of $t\eth^2_{\fD}$ to the new boundary face will single out the part of $\eth^2_{\fD}$ that is most singular at the boundary, namely $(\eth^V_{\fD})^2$.
The solution to the new model problem is the product of the Euclidean heat kernel on the base of the boundary fibration with the heat kernel of $(\eth^V_{\fD})^2$, both evaluated in the rescaled time variable $\tau$. As $\tau \to 0$ the solution to the model problem goes to the identity (which is consistent with the model at $\bhs{00,2}$) while as $\tau \to \infty$ it approaches the projection onto the null space of $\eth^V_{\fD}$ (here the constant rank assumption is 
essential).

There is one more submanifold that should be blown-up, the full corner at positive time
\begin{equation*}
	\pa X^2 \times \bbR^+,
\end{equation*}
and this blow-up should be radial (i.e., without any parabolic directions).
The resulting face is denoted $\bhs{11,0}$ and it too has a model problem. We need the solution to the heat operator $t\pa_t + t\eth_{\fD}^2$ restricted to this face with initial condition imposed by matching with the solution of the model problem on $\bhs{\phi\phi,2}$. Thus as $t\to0$ we need to match with the limit as $\tau\to\infty$ of the solution of the previous model, i.e., the initial value should be the projection onto $\cK$. The solution to the model problem is thus the heat kernel of $\eth^{H,\cK}_{\fD}$ conjugated by the projection onto $\cK$.

The heat space $HX_{\fD}$ is the result of blowing-up the three submanifolds described above (though for technical reasons the blow-ups should be made in the opposite order),
\begin{equation*}
	HX_{\fD}
	= \lrspar{ 
	X^2 \times \bbR^+
	; (\pa X)^2 \times \bbR^+ 
	; \diag_\phi \times \{ 0\}, dt 
	; \diag \times \{0\}, dt }.
\end{equation*}
We have described the solution to the heat operator $t\pa_t + t\eth_{\fD}^2$ when restricted to the blown-up boundary faces, $\bhs{00,2}$, $\bhs{\phi\phi,2}$, and $\bhs{11,2}$. Vaillant constructed a heat calculus by singling out distributions on $HX_{\fD}$ that have asymptotic expansions at all of the boundary faces, that act as smoothing operators on functions on $X \times\bbR^+$, and that are (conditionally) closed under composition. He then showed that the solution to the equation \eqref{HeatEq} is an element of this calculus by extending the solutions described above off of the  boundary faces and solving away the error.
If we denote by $\Psi^{2,2,0}_{Heat, \fD}(X,E)$ the elements of this calculus that (in an appropriate sense) vanish to second order at $\bhs{00,2}$ and $\bhs{\phi\phi,2}$, to $\text{zero}^{\text{th}}$ order at $\bhs{11,0}$ and to infinite order at all other boundary faces, then we can state Vaillant's result.

\begin{theorem}[Vaillant \cite{Vaillant}, \S 4]
The solution to \eqref{HeatEq} is an element of \newline $\Psi^{2,2,0}_{Heat, \fD}(X,E)$ with restriction to each of the three blown-up boundary faces given by the solution to the model problem described above.
\end{theorem}

A little more precisely, if $\rho_S$ denotes a boundary defining function for $\bhs{S}$, then the elements of $\Psi^{2,2,0}_{H\fD}(X,E)$ are smooth sections of the density bundle
\begin{equation}\label{KernelSpace}
	\rho_{00,2}^2 \rho_{\phi\phi,2}^2 \rho_{11,0}^{-\dim Z} 
	\lrspar{ \frac{(x')^n}{t^{n/2}} \frac{dt}t \beta^*_{H,R}\FC \Omega(X) \otimes \Hom(E) }
\end{equation}
that vanish to infinite order at all boundary faces other than the three blown-up ones.
Here $\beta_{H,R}$ is the projection to the right factor of $X$, $\FC \Omega(X)$ is the $\fC$ density bundle on $X$ and $\Hom(E)$ is the bundle over $X^2$ whose fiber over $(\zeta,\zeta')$ is the space of homomorphisms from $E_\zeta$ to $E_{\zeta'}$.
A little unwinding shows that the restriction of the heat kernel to the diagonal (away from the boundary) is a density on $X$ with a singularity of $t^{-n/2}$ as $t\to0$.
For the index theorem it is necessary (as recalled in the next section) to compute the supertrace of the heat kernel as time goes to zero. The existence of this limit is due to a "fantastic cancellation" of all of the potentialy singular terms in the supertrace.  One of the best ways to see this is via the use
of Getzler rescaling.

Getzler rescaling is carried out by Vaillant (following \cite[Chapter 8]{APS Book}) by refining the calculus to one with the rescaling built in and by showing that this refined calculus contains the heat kernel (essentially by constructing it again).
Let us describe this refined calculus.

Consider the restriction of the bundle $\Hom(E)$ to $\bhs{00,2}$. Since this face lies over the diagonal in $X^2$ (and $\Hom(E)$ is pulled-back to $HX_{\fD}$ from $X^2$), we can identify $\Hom(E)$ at this face with $\hom(E)$, the bundle over $X$ whose fiber over $\zeta$ consists of all homomorphisms from $E_\zeta$ to itself. A fundamental property of Clifford bundles is that the homomorphisms have a decomposition
\begin{equation*}
	\hom(E) \cong \Cl(T^*X) \otimes \Hom'_X(E)
\end{equation*}
where $\Hom'_X(E)$ denotes the homomorphisms that commute with Clifford multiplication. Furthermore the natural filtration of $\Cl(T^*X)$ by Clifford degree induces a filtration of $\hom(E)$ by
\begin{equation*}
	\Cl_{00,2}^{(i)} := \Cl^{(i)}(T^*X) \otimes \Hom'_X(E).
\end{equation*}
Something similar is true at the boundary face $\bhs{\phi\phi,2}$. This face lies over $\{y=y', x=x'=0\}$, so if we denote by $T^*Y^+ = T^*Y \oplus \bbR$ the $b$ cotangent bundle of $Y^+:=Y \times \bbR^+_x$ restricted to the boundary, we have
\begin{equation*}
	\Hom(E)\rest{ \bhs{\phi\phi,2} } \cong \Cl(T^*Y^+) \otimes \Hom'_{Y^+}(E)
\end{equation*}
where $\Hom'_{Y^+}(E)$ denotes the homomorphisms that commute with Clifford multiplication by vectors in $T^*Y^+$.  This suggests to define
\begin{equation*}
	\Cl_{\phi\phi,2}^{(i)} := \Cl^{(i)}(T^*Y^+) \otimes \Hom'_{Y^+}(E).
\end{equation*}

We want to use these filtrations to single out sections of $\Hom(E)$ (over $HM_{\fD}$), however we should first extend the filtrations off of the boundary faces.
One way of doing this is by using the connection
\begin{equation*}
	\nabla^{\Hom(E)} = \pa_t \;dt \otimes \nabla^E \otimes \nabla^{E^*},
\end{equation*}
a vector field $N$ transverse to $\bhs{00,2}$ but tangent to all other boundary faces, and another vector field $N'$ transverse to $\bhs{\phi\phi,2}$ but tangent to all other boundary faces. Then we declare that a section $s$ of $\Hom(E)$ is {\bf rescaled} if
\begin{gather*}
	\lrpar{\nabla^{\Hom(E)}_N}^j s \rest{00,2} \in \Cl_{00,2}^{(j)}(E),
	\text{ for } j \in \{0,1,\ldots, \dim X\} \\
	\lrpar{\nabla^{\Hom(E)}_{N'}}^k s \rest{\phi\phi,2} \in \Cl_{\phi\phi,2}^{(k)}(E),
	\text{ for } k \in \{0,1,\ldots, \dim Y+1\}.
\end{gather*}
Melrose showed that there is a bundle, $\Gr(E)$, whose sections are precisely the rescaled sections of $\Hom(E)$. Vaillant defines a rescaled heat calculus by replacing $\Hom(E)$ in \eqref{KernelSpace} with $\Gr(E)$. 

Showing that $t\pa_t + t\eth_{\fD}^2$ acts on elements of the rescaled heat calculus reduces, by means of the Lichnerowicz formula, essentially to showing that the connection $\nabla^E$ does. Melrose gave a criterion for this and Vaillant uses the asymptotics of the curvature (Lemma \ref{ConnectionAsymp}) to show that the criterion holds. The construction of the heat kernel in this rescaled calculus then proceeds by solving the model problems at each of the boundary faces.

At $\bhs{00,2}$ the model problem is the heat equation for the harmonic oscillator and, as shown by Getzler, Mehler's formula extends to this context to give an explicit solution. The model problem at $\bhs{\phi\phi,2}$ essentially splits into a problem in the base directions (again a harmonic oscillator) and a problem on the fibers. The latter is the heat equation for a superconnection $\bbB_{\fD}$, defined just like the Bismut superconnection except that the second fundamental form of the fibers $\Sec{\phi}$ is replaced by $\Sec{\fD}$. It is easy to see that these solutions patch together and then Vaillant again solves away the error to construct the solution to the heat equation.
To state Vaillant's construction as a theorem, we denote by 
$\Psi^{2,2,0}_{RHeat,\fD}(X,E)$ the subset of 
$\Psi^{2,2,0}_{Heat,\fD}(X,E)$ that takes values in $\Gr(E)$ instead of $\Hom(E)$.

\begin{theorem}[Vaillant \cite{Vaillant}, \S 5]
The solution to \eqref{HeatEq} is an element of \newline $\Psi^{2,2,0}_{RHeat,\fD}(X,E)$.
\end{theorem}

\subsection{Vaillant's index theorem for $\fD$ Dirac-type operators} \label{sec:VaillantIndex} $ $ \newline
In section \S 5 of his thesis, Vaillant proves an index theorem for $\fD$ Dirac-type operators satisfying the constant rank assumption. These operators need not be Fredholm, in which case Vaillant's formula is for the {\em extended} $L^2$-index. The proof is modeled on Melrose's proof of the Atiyah-Patodi-Singer index theorem in \cite{APS Book} and the investigation of the adiabatic limit of analytic torsion by Dai and Melrose \cite{DaiMelrose}.

First recall the heat equation proof of the index theorem on a closed manifold, $X$. 
If $\eth$ is a Dirac-type operator acting on the $\bbZ_2$-graded Clifford bundle $E$, then $\eth$ is elliptic and hence automatically Fredholm with index
\begin{equation*}
	\ind(\eth):= \dim\ker\eth - \dim\coker\eth.
\end{equation*}
The heat kernel of $\eth^2$ is the solution to 
\begin{equation*}
	\begin{cases} (\pa_t + \eth^2)H_t = 0 \\ H_{t=0} =\Id \end{cases}
\end{equation*}
and is denoted $e^{-t\eth^2}$.
McKean and Singer noticed that its supertrace (the difference of the trace on $E^+$ and the trace on $E^-$) is in fact independent of $t$.
The index is the limit as $t\to\infty$ of this supertrace and hence we have 
\begin{equation*}
	\ind (\eth) 
	= \lim_{t \to \infty} \Str(e^{-t\eth^2})
	= \lim_{t \to 0} \Str(e^{-t\eth^2}).
\end{equation*}
Getzler showed that this last term could be computed explicitly by a method now known as Getzler rescaling and gives the Atiyah-Singer index formula
\begin{equation}\label{ASFormula}
	\ind(\eth) = \int_X \hat A(X) \Ch'(E)
\end{equation}
in terms of the $\hat A$-genus of the manifold and the `twisting' Chern character of the bundle $E$.

Melrose proved the Atiyah-Patodi-Singer index theorem by adapting this proof to the context of manifolds with asymptotically cylindrical ends or $b$-manifolds (a $b$ metric is a $\fD$ metric where the fiber of the boundary fibration is a single point).
If $X$ is a manifold with an exact $b$-metric, $\eth_b$ is a $b$ Dirac-type operator, and $e^{-t\eth^2_b}$ is the heat kernel of its square, then the proof described above can be applied to compute the index of $\eth_b$, but most steps are not as simple as on a closed manifold. 
First $\eth_b$ is not necessarily Fredholm and even when it is Fredholm its heat kernel is not trace-class. Melrose introduces a renormalized trace to replace the usual trace, but then the renormalized supertrace of the heat kernel is not independent of $t$. Nevertheless the limits as $t\to\infty$ and $t \to 0$ can be computed just as before and so we obtain
\begin{equation*}
	\ind (\eth_b) 
	= \lim_{t \to \infty} {}^R\Str(e^{-t\eth^2})
	= \lim_{t \to 0} {}^R \Str(e^{-t\eth^2}) 
	+ \int_0^\infty \pa_t {}^R\Str(e^{-t\eth^2}).
\end{equation*}
and Melrose shows that this last term is given by the $\eta$ invariant introduced by Atiyah, Patodi and Singer,
\begin{equation*}
	\ind(\eth_b) = \int_X \hat A(X) \Ch'(E) - \frac12\eta(\eth_{\pa X}).
\end{equation*}

Vaillant's proof of the $\fD$ index theorem follows the same steps as Melrose's proof of the $b$ index theorem and we briefly describe what is involved in each step. As mentioned above, the ellipticity of $\eth_{\fD}$ is not enough to guarantee that it is a Fredholm operator, and Vaillant assumes that $\eth_{\fD}$ satisfies the constant rank assumption. Given this assumption, Vaillant constructs the heat kernel as described in the previous section. As for the $b$ calculus, the heat kernel is not trace-class and Vaillant uses the renormalized trace as introduced in \cite{MelroseNistor} (equivalent to the $b$-trace used by Melrose in \cite{APS Book}, see \cite{Albin}) which we now describe.

Recall that whenever an operator acts by an integral kernel so that
\begin{equation*}
	Au(\zeta) =  \int K_A(\zeta, \zeta') u(\zeta') \; d\zeta'
\end{equation*}
and is trace-class, Lidskii's theorem says that its trace is given by
\begin{equation*}
	\Tr A = \int K_A(\zeta, \zeta) \; d\zeta.
\end{equation*}
When $X$ is the interior of a manifold with boundary it can happen that this integral does not converge - even for a smoothing operator - because the measure blows-up at the boundary.
Let $x$ be a boundary defining function and assume that $K_A$ has an asymptotic expansion in $x$ as $\zeta$ approaches the boundary, then the function
\begin{equation*}
	F(z) = \int x^z K_A(\zeta, \zeta) \; d\zeta
\end{equation*}
is holomorphic on a half-plane and extends meromorphically to all of the complex plane. We define {\bf the renormalized trace of $A$} to be the finite part of $F(z)$ at $z=0$,
\begin{equation}\label{DefRTr}
	\RTr A := \FP_{z=0} \int x^z K_A(\zeta, \zeta) \; d\zeta.
\end{equation}
It is easy to see that if $A$ is trace-class then ${}^R\Tr(A)=\Tr(A)$ so the renormalized trace of $A$ is an extension of the usual trace. However the renormalized trace is not actually a trace in that it does not always vanish on commutators.

Vaillant uses the description of the spectral measure at zero to show that
\begin{equation*}
	\lim_{t\to\infty} {}^R \Str\lrpar{ e^{-t\eth_{\fD}^2} } = \ind_{-}(\eth_{\fD}),
\end{equation*}
the extended index of $\eth_{\fD}$. The contribution as $t\to0$ has two parts corresponding to the two boundary faces in $HX_{\fD}$ over the diagonal at $t=0$, namely $\bhs{00,2}$ and $\bhs{\phi\phi,2}$. These can both be computed using the model problems solved by the heat kernel at these faces. Thus the former contributes the usual Atiyah-Singer integral occurring in \eqref{ASFormula} (convergence of the integral follows from the fact that the Levi-Civita connection of a $\fD$ metric is a true connection). The latter boundary face contributes 
\begin{equation*}
	\int_Y \hat A(Y) \hat\eta(\bbB_{\fD})
\end{equation*}
where $\hat\eta(\bbB_{\fD})$ is the Bismut-Cheeger $\eta$ form associated to the superconnection $\bbB_{\fD}$ (described in the previous section).

Finally the contribution from finite time,
\begin{equation*}
	\int_{\bbR^+} \pa_t {}^R\Str ( e^{-t\eth^2} ),
\end{equation*}
reduces to a contribution from the third of the faces blown-up to construct $HX_{\fD}$, that is, $\bhs{11,0}$. Computations very similar to those occurring in the proof of the index theorem for $b$-metrics show that this contribution is equal to the $\eta$-invariant of the horizontal operator $\eth^{H,\cK}_{\fD}$. Putting these three contributions together, Vaillant obtains the following theorem.

\begin{theorem}[Vaillant \cite{Vaillant}, Theorem 5.29]
The extended index, $\ind_{-}(\eth_{\fD})$, of a $\fD$ Dirac-type operator $\eth_{\fD}$ satisfying the constant rank assumption is given by
\begin{equation*}
	\frac1{(2\pi i)^{n/2}} \int_X \hat A(X) \Ch'(E)
	- \frac1{(2\pi i)^{\lfloor\frac{h+1}{2}\rfloor}} \int_Y \hat A(Y) \hat\eta(\bbB_{\fD})
	- \frac12 \eta(\eth^H),
\end{equation*}
where $n = \dim X$ and $h = \dim Y$.
\end{theorem}

{\bf Remark.} Our conventions for $\eta$ invariants consistently differ from those of Vaillant by a sign.  Moreover, for eta forms, we use the convention of Bismut and Cheeger \cite{BC0}
instead of the one of Vaillant.  For $h$ odd, the eta form of Vaillant differs from the one of Bismut and Cheeger only by a factor of $-1$, while for $h$ even, the eta form of
Vaillant is $-\sqrt{2\pi i}$ times the one of Bismut and Cheeger.

\section{Families of $\fD$ Dirac-type operators and Fredholm perturbations}

In the previous section we have explained how Vaillant extended Melrose's proof of the index theorem for exact $b$-metrics to an index theorem for exact $\fD$ metrics. In this section we will begin the process of extending the proof of Melrose and Piazza of the index theorem for families of Dirac-type operators associated to exact $b$-metrics to an index theorem for such families associated to exact $\fD$ metrics.

\subsection{Families of $\fD$ Dirac-type operators} \label{sec:FamiliesDirac} $ $ \newline
We start with a fibration $M \xrightarrow\psi B$ as in \eqref{fullfib},
\begin{equation*}
	\xymatrix { & & X \ar@{-}[r] & M \ar[d]^{\psi} \\
	Z \ar@{-}[r] & \pa X \ar[d]^{\phi}  \ar@{^{(}->}[ur] \ar@{-}[r] 
		& \pa M \ar[d]^{\wt \phi}  \ar@{^{(}->}[ur]  & B \\
	& Y \ar@{-}[r] & D \ar[ur]_{\psi_D} & }
\end{equation*}
Given a boundary defining function $x$ for $\pa M$ we define $\FC TM$ and $\FD TM$ just as in \eqref{BdlesDef}. 
The differential of $\psi$ extends from the interior to a map $\psi_*: \FC TM \to TB$, so we can define
\begin{equation*}
	\FC TM/B = \ker \psi_*
\end{equation*}
and define $\FD TM/B$ in terms of $\FC TM/B$.
We choose a connection for $\psi$ in the form of a splitting $\FD TM = \FD TM/B \oplus \psi^* TB$ and denote the projections onto each factor by $\v\psi$ and $\h\psi$ respectively.

A family of exact $\fD$ metrics, $g_{\fD}$, is a fibre metric on the bundle $\FD TM/B$ whose restriction to each fiber is an exact $\fD$ metric as described in \S\ref{sec:BhvCusps}. We will denote the vertical $\fD$ Clifford algebra by
\begin{equation*}
	\Cl_\psi(M) := \Cl(\FD T^*M/B).
\end{equation*}
A $\fD$ Clifford module, $E$, over the fibers of $\psi$ is a bundle over $M$ that restricts to each fiber to be a Clifford module associated to the $\fD$ metric $g_{M/B}$ as described in \S\ref{sec:DiracTypeOps}.  As before we assume that the fibers of $M \xrightarrow\psi B$ are even-dimensional so that $\Cl_\psi\lrpar M$ is $\bbZ/2$-graded and that $E$ has a compatible $\bbZ/2$=grading
\begin{equation*}
	E = E^+ \oplus E^-.
\end{equation*}
We denote the resulting $\psi$-vertical family of $\fD$ Dirac-type operators by
\begin{equation*}
	D_{\fD}:\CI\lrpar{M/B;E} \xrightarrow{\nabla^E} \CI\lrpar{M/B;\FD T^*M/B\otimes E}
	\xrightarrow{\cll} \CI\lrpar{M/B;E}.
\end{equation*}
The family of operators $D_{\fD}$ acting on the space $L^2_{\fD}(M/B;E)$ is unitarily equivalent to the operators
\begin{equation*}
	\eth_{\fD} = D_{\fD} - \tfrac v2 \cl{\tfrac{dx}x}
\end{equation*}
acting on $L^2_b(M/B;E)$ and, as for one operator, we prefer to work on this space.

As before, whether or not the family of elliptic operators $\eth_{\fD}$ is Fredholm is determined by two model operators. The first of these is a family of Dirac operators on the fibers of the fibration $Z - \pa M \xrightarrow{\wt \phi} D$ which we again denote $\eth^V_{\fD}$ and is obtained by restricting $x\eth_{\fD}$ to the boundary, i.e., 
\begin{equation*}
	\eth^V_{\fD} := x\eth_{\fD}\rest{\pa M} 
	\in \Psi^1\lrpar{ \pa M/D; E}.
\end{equation*}
We say that $\eth^V_{\fD}$ satisfies the {\bf constant rank assumption} if the dimension of the null space is independent of the point in $D$. We use the same term for this as for the constant rank assumption for a single operator from \S \ref{sec:DiracTypeOps} and hope that, as we shall only consider families of operators from now on, this will not result in confusion. We point out that the constant rank assumption for families is stronger than imposing the constant rank assumption on each of the individual operators $\eth_{\fD}$.

It is easy to see that the operator $x\eth_{\fD}$ is itself Fredholm as an element of the $\fC$ calculus precisely when $\eth^V_{\fD}$ is invertible (e.g., from \eqref{NormalFCOp} below). In which case, by results of Mazzeo-Melrose, $x\eth_{\fD}$ has a parametrix with compact error inside $\Psi^{-1}_{\fC}\lrpar{M/B;E}$, say $B$. Thus $xB$ is both a parametrix for $\eth_{\fD}$ and a compact operator so $\eth_{\fD}$ necessarily has purely discrete spectrum and a trace-class heat kernel, properties we do not usually expect of a geometric $\fD$ Dirac-type operator.
It is therefore interesting that we are often able to perturb $x\eth_{\fD}$ by an operator that is smoothing in the interior and obtain a Fredholm $\fC$ operator.

Recall that Mazzeo and Melrose constructed in \cite{MazzeoMelrose} a pseudo-differential calculus, $\Psi_{\fC}^*$, on manifolds with a fibered boundary that includes differential $\fC$ operators and their parametrices (when these exist). 
Operators in this calculus have two model operators: the principal symbol (defined as on closed manifolds) and the normal operator which models the behavior at the boundary (and takes values in a suspended calculus of pseudo-differential operators).
An operator in the $\fC$ calculus is Fredholm if both model operators are invertible and it is compact if they both vanish. Elements of $\Psi^{-\infty}_{\fC}$ are smoothing in the interior but are not compact operators (when acting on $L^2_{\fC}$ for instance).

An operator $Q \in \Psi^{-\infty}_{\fC}(M/B; E)$ is a {\bf large smoothing perturbation} if $x\eth_{\fD} + Q$ is a Fredholm $\fC$ operator.
As already mentioned, in this case the operator $\eth_{\fD} + \frac1x Q$ has compact resolvent and discrete spectrum so these perturbations will generally change the structure of $\eth_{\fD}$ a great deal.

An operator $\eth_{\fD}$ satisfying the constant rank assumption does not necessarily admit a large smoothing perturbation.  However, this assumption will sometimes allow us to make a subtler correction and obtain a Fredholm operator corresponding to the second of the model operators of $\eth_{\fD}$. Denote by $\cK$ the bundle over $D$ with fibers the null spaces of the operators $\eth^V_{\fD}$.
If $\dim \cK >0$ there is a second induced family, $\eth^{H,\cK}_{\fD}$, obtained by restricting $\eth_{\fD}$ to $\cK$ at the boundary. This family acts on the fibers of the fibration $ Y - D \xrightarrow{\psi_D} B$.
As mentioned in \S\ref{sec:DiracTypeOps}, Vaillant showed that in this case $\eth_{\fD}$ is Fredholm if and only if $\eth^{H,\cK}_{\fD}$ is invertible. 
Let $\cP_\cK$ denote an element of $\Psi^0_b(M/B;E)$ such that $I_b(\cP_{\cK})$ is the projection onto $\cK$.
An operator $Q' \in \Psi^{-\infty}_b\lrpar{M/B;\cK}$ is a {\bf small smoothing perturbation} if $\eth_{\fD} + Q'\cP_\cK$ is Fredholm (its index bundle is easily seen to be independent of the choice of $\cP_{\cK}$).

\subsection{Existence of Fredholm perturbations} \label{sec:Perturbations} $ $ 
\newline
We now address the existence of large and small smoothing perturbations. 
Given an elliptic $\fC$ operator, $A \in \Psi^*_{\fC}(M/B;E)$, Melrose and the second author investigated the existence of Fredholm perturbations by elements of $\Psi^{-\infty}_{\fC}(M/B;E)$ in \cite{MelroseRochon}. Such a perturbation exists if and only if a topological obstruction on the principal symbol of $A$ vanishes. Precisely, the K-theory class determined by $\sigma(A)$ must be in the null space of the families index map
\begin{equation*}
	K_c(T^*M/B) \xrightarrow{AS} K_c^1(T^*D/B).
\end{equation*}
For Dirac-type operators, we will be interested in more specific perturbations.

If $\eth_{\fD}$ is a family of $\fD$ Dirac-type operators,  
a simple computation \cite[Proposition 3.6]{Vaillant} (or an appeal to naturality) shows that the normal operator of $x\eth_{\fD}$ is
\begin{equation}\label{NormalFCOp}
	N(x\eth_{\fD}) = \eth^V_{\fD} + \eth_{\FC N\pa X}.
\end{equation}
The \emph{normal family} of a $\fC$ operator (also called the indicial family) is obtained by taking the Fourier transform of the normal operator in the $x$ and $y$ directions. For every point $(q,\xi,\eta)\in T^*Y^+$ we have an operator (see \cite[Lemma 9]{MazzeoMelrose})
\begin{equation*}
	\hat N(x\eth_{\fD})(q, \xi,\eta) 
	\in \Psi^1(\phi^{-1}(q);E) 
\end{equation*}
and clearly $\hat N(x\eth_{\fD})(q, 0, 0)= \eth^V_{\fD}$.

Suppose now that $A\in \Psi^{-\infty}(\pa M/D;E)$ is a self-adjoint 
perturbation such that the family
$\eth^{V}_{\fD}+A$ is invertible.  Suppose also that A anti-commutes with Clifford multiplication
by odd elements of $\Cl(T^*D^+/B)$.  We say in that case that $A$ is a 
{\bf $\Cl(T^{*}D^+/B)$-odd self-adjoint perturbation}.
To get a corresponding large smoothing perturbation, let 
$\hat{\rho}\in \dot{\mathcal{C}}^{\infty}(T^{*}(D/B))$ be a smooth function symmetric with 
respect to the zero section of $T^{*}(D/B)\to D$.  Assume also that $\rho$ is identically one
on the zero section.  Consider the function $\hat{\rho}_{\eps}(t):=\rho(t\eps)$ for
$\eps>0$ acting by rescaling in each fiber of $T^{*}(D/B)\to D$.  It is such that 
$\hat{\rho}_{\eps}$ approximates the identity function on $T^{*}(D/B)$ as $\eps\to 0$.  
Alternately and to compare with \cite[Lemma 9]{MelrosePiazza}, notice that this means
its fibrewise inverse Fourier transform, which is a fibrewise volume form
$\rho_{\eps}\in\dot{\mathcal{C}}^{\infty}(T(D/B);\Lambda^{h}(T(D/B)))$, approximates
the Dirac delta function of the zero section of $T(D/B)\to D$ as $\eps\to 0$.  Because of the
decay conditions on $\hat{\rho}$ and $\hat{\rho}_{\eps}$, the operator 
$\hat{A}^{\eps}:= \hat{\rho}_{\eps}A$ can be interpreted as the normal family of some
operator $A^{\eps}\in \Psi_{\fC}^{-\infty}(M/B;E)$,
\[
     \hat{N}(A^{\eps})= \hat{A}^{\eps}.
\]
This is the large Fredholm perturbation we are looking for.
\begin{lemma}
Let $A\in \Psi^{-\infty}(M/D;E)$ be a $\Cl(T^{*}D^+/B)$-odd self-adjoint perturbation such
that $\eth^{V}_{\fD}+A$ is invertible.  Then for each $\eps>0$, any operator
$A^{\eps}\in \Psi^{-\infty}_{\fC}(M/B;E)$ with normal family
family $\hat{N}(A)\hat{A}^{\eps}= \hat{\rho}_{\eps} A$ is a Fredholm perturbation
of $x\eth_{\fD}$.
\label{fp.1}\end{lemma}
\begin{proof}
It suffices to check that $\hat{N}(x\eth_{\fD}+ A^{\eps})$ is an invertible family.  But
the normal family of $x\eth_{\fD}$ is given by
\begin{equation}
  \hat{N}(x\eth_{\fD})= \eth^{V}_{\fD}+ \sigma(\eth_{\FC N\pa X})
\label{fp.2}\end{equation}
where $\sigma(\eth_{\FC N\pa X})$ is the symbol of $\eth_{\FC N\pa X}$, that is,
at $\gamma\in T^{*}D^{+}/B$, it is given by $i\cl{\gamma}$ where $\cl{\gamma}$ is the 
corresponding Clifford multiplication.  Since $A$ is $\Cl(T^{*}D^{+}/B)$-odd, we compute 
directly that
\begin{equation}
   \lrpar{\hat{N}(x\eth_{\fD}+A^{\eps})(d,\xi)}^{2}= 
     \lrpar{\eth^{V}_{\fD}+\hat{\rho}_{\eps}A}^{2} + \|\xi\|^{2} \Id_{E},       \label{fp.3}\end{equation}  
for $d\in D$ and $\xi \in T_{d}^{*}D^{+}/B$.
This family is clearly invertible, which implies a fortiori that
$\hat{N}(\eth_{\fD}+A^{\eps})$ is invertible as well. 
\end{proof}
The converse of Lemma~\ref{fp.1} is also true, namely, the existence of a Fredholm perturbation
for $x\eth_{\fD}$ implies the existence of a $\Cl(T^{*}D^{+}/B)$-odd self-ajoint
perturbation $A\in \Psi^{-\infty}(\pa M/D;E)$ such that $\eth^{V}_{\fD}+A$ is invertible.
The proof is however much harder.  Since we do not need it for our purposes, we will omit it. 

The type of large Fredholm perturbations arising in Lemma~\ref{fp.1} is intimately related with the notion of spectral section introduced by Melrose and Piazza in \cite{MelrosePiazza}.  Recall
that a spectral section for $\eth^V_{\fD}$ is a family of self-adjoint projections
$P\in \Psi^0\lrpar{\pa M/D;E}$ with the property that, for $d\in D$,
\begin{equation}\label{SpecSec}
	\lrpar{\eth^V_{\fD}}_d u = \lambda u \implies 
	\begin{cases}
		P_d u=u & \text{ if } \lambda > R\lrpar d \\
		P_d u=0 & \text{ if } \lambda < -R\lrpar d
	\end{cases}
\end{equation}
for some function $R:D \to \bbR^+$.
Melrose and Piazza showed that the existence of a spectral section is equivalent to the vanishing of the families index of $\eth^V_{\fD}$.
Also, in \cite{MelrosePiazza2}, they showed that a more refined statement is true taking into account the action of $\cl{\frac{dx}x}$. We generalize this to allow for the action of $\Cl(T^*D^+/B)$.

Let $P$ be a spectral section and denote the associated involution $\Id-2P$ by $\Pi$. We say that $P$ is a {\bf $\Cl(T^*D^+/B)$ spectral section} if 
\begin{equation}\label{RequiredPi}
	\textstyle
	\Pi \cl{\frac{dx}x} + \cl{\frac{dx}x} \Pi = 0, \quad
	\Pi \cl{\theta} + \cl{\theta} \Pi = 0, \phantom{xx}\theta \in \Gamma(T^*D/B).
\end{equation}
When $A\in \Psi^{-\infty}(\pa M/D:E)$ is a $\Cl(T^{*}D^{+}/B)$-odd self-adjoint perturbation
such that $\eth^{V}_{\fD}+A$ is an invertible family,  then one naturally gets a 
$\Cl(T^{*}D^{+}/B)$ spectral section $P$ defined by
\begin{equation}
      \lrpar{\eth^{V}_{\fD}+A}_{d}u =\lambda u \quad \Longrightarrow
       \left\{ \begin{array}{ll}
         Pu=u, & \text{if}\, \lambda>0, \\
         Pu=0, & \text{if}\, \lambda<0.
       \end{array}\right.  
\label{fp.5}\end{equation}
It is the \textbf{Atiyah-Patodi-Singer spectral section} associated to the invertible
family $\eth^{V}_{\fD}+A$.  Conversely, given a $\Cl(T^{*}D^{+}/B)$ spectral section $P$,
it is possible to construct a $\Cl(T^{*}D^{+}/B)$-odd self-adjoint perturbation
$A_{P}\in \Psi^{-\infty}(\pa M/D;E)$ such that $P$ is the Atiyah-Patodi-Singer spectral section
of $\eth^{V}_{\fD}+A_{P}$. Given $P$, the choice of $A_{P}$ is not unique, but it is at least
unique up to homotopy.  

In \cite{MelrosePiazza}, Melrose and Piazza obtained a family generalization of the 
Atiyah-Patodi-Singer index theorem where the boundary condition was specified by a spectral
section $P$.  Their strategy was to translate this boundary condition into an index problem
for $b$-operators by adding a cylindrical end to the manifold with boundary and by replacing
the spectral section $P$ by a Fredholm perturbation.  In our case, we do not need spectral 
sections since we start directly with a Fredholm perturbation.  Still, to ease the use of results from \cite{MelrosePiazza}, we will use the notation $A_{P}$ for a 
$\Cl(T^{*}D^{+}/B)$-odd self-adjoint perturbation such that $\eth^{V}_{\fD}+A_P$ is invertible,
the subscript $P$ denoting the Atiyah-Patodi-Singer spectral section associated to 
$\eth_{\fD}^{V}+A_{P}$.  We also denote by $A^{\eps}_{P}\in \Psi_{\fC}^{-\infty}(M/B;E)$ a
choice of Fredholm perturbation with normal family given by 
$\hat{N}(A^{\eps}_{P})=\hat{\rho}_{\eps}A_P$.

Given any such $A_P^{\eps}$, we know from \cite[Lemma 1.1]{MelroseRochon} that there is a family of smoothing operators $A'$ with compact support away from the boundary of $M/B$ such that $\dim \ker \lrpar{x\eth_{\fD} + A_P^{\eps} + A'}$ is constant independent of the point in $B$. We will incorporate $A'$ into $A_P^{\eps}$ without reflecting this in the notation. We say that $A_P^\eps$ is a {\bf large smoothing perturbation associated to $P$}. In this context $P$ will always indicate a $\Cl(T^{*}D^{+}/B)$ spectral section.

The existence of a small smoothing perturbation is similar but more direct. 
Again from \cite{MelrosePiazza} there exists a family $Q \in \Psi^{-\infty}(D/B;\cK)$ such that $\eth^{H,\cK}_{\fD} + Q$ is invertible if and only if the families index of $\eth^{H,\cK}_{\fD}$ vanishes, if and only if $\eth^{H,\cK}_{\fD}$ admits a spectral section.
Given a spectral section $R$ for $\eth^{H,\cK}_{\fD}$ there is a family of self-adjoint smoothing operators $Q_R \in \Psi^{-\infty}(D/B;\cK)$ with range in a finite sum of eigenspaces of $\eth^{H,\cK}_{\fD}$ such that $\eth^{H,\cK}_{\fD} + Q_R$ is invertible. These can be extended as above to a family $Q_R^\eps \in \Psi^{-\infty}_b(M/B;\cK)$ with $I_b(Q_R^0)=Q_R$, and such that $\eth_{\fD} + Q_R^\eps \cP_{\cK}$ is Fredholm for small $\eps>0$, and $\dim \ker (\eth_{\fD} + Q_R^\eps)$ is independent of $b \in B$.
We say that $Q_R^\eps$ is a {\bf small smoothing perturbation associated to $R$}.

Finally in the non-perturbed Fredholm case, when $\eth^V_{\fD}$ has constant rank kernel and $\eth^{H,\cK}_{\fD}$ is invertible, it is possible to add a smoothing operator whose kernel is compactly supported away from the boundary and arrange for $\ker \eth_{\fD}$ and $\coker \eth_{\fD}$ to be vector bundles \cite[Proposition 4]{MelrosePiazza}. In any of the three cases, we will assume whenever convenient that such a perturbation has been added so that the index bundle is simply the virtual difference bundle.


\subsection{Conformally related metrics} \label{sec:ConfRelMet} $ $ \newline
It was pointed out by Hitchin that Dirac operators are conformally covariant. What this means is that if a Riemannian manifold $\lrpar{N,g_N}$ has a 
Hermitian Clifford bundle, $E$, with Clifford connection, and $\wt g_N=e^{2\omega}g_N$ is a conformally related metric, then $E$ can be endowed with a $\wt g_N$-Clifford structure and the generalized Dirac operators are related by
\begin{equation*}
	\eth_{\wt g_N} = e^{-\frac{n-1}2\omega}\lrpar{e^{-\omega}\eth_{g_N}}e^{\frac{n-1}2\omega}
\end{equation*}
as operators on, for instance, $C_c^\infty\lrpar N$ (see \cite[Appendix 2]{Vaillant}).

So associated to the family of $\fD$ Dirac operators $\eth_{\fD}$ is a family of $\fC$ Dirac operators, $\eth_{\fC}$ and moreover \cite[Proposition 3.6]{Vaillant}
\begin{equation*}
	N\lrpar{\eth_{\fC}} = N\lrpar{x\eth_{\fD}},
\end{equation*}
though note that $x\eth_{\fD}-\eth_{\fC} = \frac{n-1}2 x\cll_{\fC}(\frac{dx}{x^2}) \neq0$.

Moroianu \cite{Moroianu} applies these considerations to a $\fC$ metric, $g_0$, on a manifold $X$, a conformally related metric $g_p:=x^{2p}g_0$, and the corresponding generalized Dirac operators $\eth_0$, $\eth_p$. By Hitchin's formula 
\begin{equation*}
	\eth_p = x^{-p(n+1)/2}\eth_0x^{p(n-1)/2}.
\end{equation*}
Then it is always true that
\begin{equation*}
	\ker\lrpar{\eth_p} \cong \ker\lrpar{x^{-p} \eth_0 x^{-p}}
\end{equation*}
and if $\eth_0$ is fully elliptic then also
\begin{equation*}
	\ker\lrpar{x^{-p} \eth_0 x^{-p} } \cong \ker{\eth_0},
\end{equation*}
since $\ker(\eth_0) \subseteq \dot \CI(X)$.
For families of operators Moroianu's result extends if we demand that the dimensions of the kernels of the operators be constant.

\begin{lemma}\label{GenMoroianu}
Assume that $g_{\fC}$ is a family of exact $\fC$ metrics on the fibers of $\psi:M\to B$, that $E$ is a compatible Clifford module, $\eth_{\fC}$ is a vertical family of $\fC$ Dirac operators acting on $E$, and $A \in \Psi^{-\infty}_{\fC}\lrpar{M/B;E}$ is such that $\eth_{\fC}+A$ is a Fredholm family whose kernel has the same dimension at every point of $B$.

For every $p \in \lrspar{0,1}$, let $\eth_p$ be the vertical family of Dirac operators acting on $E$ with the Clifford structure compatible with $x^{2p}g_{\fC}$, and 
\begin{equation*}
	A_p = x^{-p(n+1)/2}A x^{p(n-1)/2},
\end{equation*}
then 
$\eth_p + A_p$ is a smooth family of Fredholm operators acting on the space $L^2\lrpar{M/B;x^{2p}g_{\fC}}$
and $\ker \lrpar{\eth_p+A_p}$ is isomorphic to $\ker \lrpar{\eth_{\fC}+A}$ as $\bbZ/2$-graded bundles over $B$. In particular, their families index coincide in K-theory.
\end{lemma}

\begin{proof}
As pointed out in \cite{Moroianu},
\begin{equation*}
	L^2\lrpar{M/B;x^{2p}g_{\fC}} \xrightarrow{x^{pn/2}} L^2\lrpar{M/B;g_{\fC}}
\end{equation*}
is an isometry and so any operator
\begin{equation*}
	T:L^2\lrpar{M/B;x^{2p}g_{\fC}} \to L^2\lrpar{M/B;x^{2p}g_{\fC}}
\end{equation*}
is unitarily equivalent to the operator $x^{pn/2} T x^{-pn/2}$ acting on $L^2\lrpar{M/B;g_{\fC}}$.
In particular, $\eth_p+A_p$ acting on $L^2_p$ is unitarily equivalent to 
\begin{gather*}
	x^{pn/2}\lrpar{\eth_p + A_p}x^{-pn/2} 
	= 
	x^{pn/2}\lrpar{ x^{-p(n+1)/2}\lrpar{\eth_{\fC} + A} x^{p(n-1)/2} } x^{-pn/2} \\
	= x^{-p/2}\lrpar{\eth_{\fC} + A} x^{-p/2}
\end{gather*}
acting on $L^2_{\fC}$
and hence these have isomorphic kernels. Sections in the kernel of $\eth_{\fC}+A$ vanish to infinite order at the boundary since this is a fully elliptic family, and so the kernel of 
$x^{-p/2}\lrpar{\eth_{\fC} + A} x^{-p/2}$ is the same as the kernel of $\eth_{\fC}+A$. As this isomorphism preserves the $\bbZ/2$-grading it restricts to isomorphisms of the kernel and cokernel of the chiral Dirac operators.
\end{proof}

{\bf Remark.}
For a family of $\fD$ Dirac-type operators as in \S\ref{sec:Perturbations}, $\eth_{\fD}$, and a large smoothing perturbation $A_P^\eps$, this lemma says that 
\begin{equation*}
	\eth_{\fD} + x^{-(n+1)/2} A_P^\eps x^{(n-1)/2}
\end{equation*}
is a smooth family of Fredholm operators whose index coincides with that of $\eth_{\fC} + A_P^\eps$.
In fact, since commutators of $\fC$ operators with powers of $x$ vanish at the boundary, we can replace this family with
\begin{equation}\label{fDPerturbation}
	\eth_{\fD} + \frac1x A_P^\eps.
\end{equation}

\section{The $\fD$ Bismut superconnection and its rescaling}

In this section we describe more of the geometry of \eqref{fullfib}, in particular the boundary fibrations. We use this to construct the Bismut superconnection adapted to the vertical Dirac-type operators $\eth_{\fD}$.
We recall that in \S\ref{sec:FamiliesDirac} we have chosen a connection on $\psi$ in the form of a splitting
\begin{equation*}
	\FD TM = \FD TM/B \oplus T_HM
\end{equation*}
wherein $T_HM$ can be identified with $\psi^*TB$.

\subsection{The $\fD$ Bismut superconnection} \label{sec:Bismut} $ $\newline
The family of metrics $g_{\fD}$ extends trivially to a degenerate metric on $TM$ by
\begin{equation}\label{DegMetric}
	g_0(V,W) := g_{\fD}(\vpsi V, \vpsi W).
\end{equation}
Alternately we can choose a metric on $B$, $g_B$, and a parameter $\eps>0$ and define
\begin{equation}\label{geps}
	g_{M}(\eps) := g_{\fD} + \frac1\eps \psi^*g_B.
\end{equation}
A useful property of $\fD$ metrics is that $g_M(\eps)$ is itself a $\fD$ metric on $M$. Notice that the dual metric on $T^*M$,
$g^M(\eps) = g^{\fD} + \eps g^B$, tends as $\eps\to0$ to the dual metric to $g_0$.

As in \S\ref{sec:BhvCusps} the behavior of the Levi-Civita connection of $g_M(\eps)$ is described 
by the two tensors
\begin{gather*}
	\Sec{\psi} \in \CI\lrpar{M; \FD T^*\lrpar{M/B} \otimes \FD T^*\lrpar{M/B} \otimes T^*_HM} \\
	\Curv{\psi} \in \CI\lrpar{M; T_H^*M \otimes T_H^*M \otimes \FD T^*\lrpar{M/B}} \\
\end{gather*}
defined as in \eqref{DefSec} and \eqref{DefCurv} respectively.
It is well-known that these tensors only depend on $g_0$, and indeed we have
\begin{gather}\label{SecPsi}
	2\Sec{\psi}\lrpar{V_1,V_2}\lrpar H 
	= H\ang{V_1,V_2}_0 - \ang{\lrspar{H,V_1},V_2}_0 
		- \ang{\lrspar{H,V_2},V_1}_0 
	\\ \label{CurvPsi}
	\Curv\psi\lrpar{H_1,H_2}\lrpar V
	= \ang{\lrspar{H_1,H_2},V}_0.
\end{gather}
We extend these trivially to $\Gamma( (\FD T^*M)^3 )$ by
\begin{gather*}
	\Sec\psi(W_1, W_2)(W_3) := \Sec\psi(\vpsi W_1, \vpsi W_2)(\hpsi W_3),\\
	\Curv\psi(W_1, W_2)(W_3) := \Curv\psi(\hpsi W_1, \hpsi W_2)(\vpsi W_3),
\end{gather*}
and then put them together to form $\omega_\psi\in \CI( M; \FD T^*M \otimes \Lambda^2(\FD T^*M) )$
\begin{multline}\label{OmegaPsi}
	\omega_\psi\lrpar X \lrpar{Y,Z}
	= S_\psi\lrpar{X,Z}\lrpar Y - S_\psi\lrpar{X,Y}\lrpar Z \\
	+\frac12\Omega_\psi\lrpar{X,Z}\lrpar Y
	-\frac12\Omega_\psi\lrpar{X,Y}\lrpar Z
	+\frac12\Omega_\psi\lrpar{Y,Z}\lrpar X.
\end{multline}

Let $\nabla^{\oplus}:= \nabla^{\fD} \oplus \psi^*\nabla^B$ (with respect to the connection on $\FD TM$ above) and let $\nabla^\eps$ be the Levi-Civita connection associated to $g^\eps$. 
The limit as $\eps \to 0$ is also a connection, $\nabla^0$, and is related to $\nabla^\oplus$ by
\begin{equation*}
	\nabla^0 = \nabla^{\oplus} + \frac12 \tau(\omega_\psi)
\end{equation*}
where
\begin{equation*}
	\tau: \Lambda^2 (\FD T^*M) \to \hom(\FD T^*M),
	\quad
	\tau(\xi\wedge\eta)\mu = 2[ g_0(\xi,\mu)\eta - g_0(\eta, \mu)\xi ].
\end{equation*}
The connection $\nabla^0$ induces a connection on the Clifford algebra associated to $g_0$,
\begin{equation}\label{tensor}
	\Cl_0\lrpar M = \psi^*\Lambda^*B \otimes\Cl_\psi\lrpar M.
\end{equation}

Let $E$ be a $\fD$ Clifford module over the fibers of $\psi$. We can extend the original fibre connection on $E$ to a full connection (again denoted $\nabla^E$) that is unitary and consistent with the Clifford action by $\Cl_\psi(M)$ (see \cite[\S 9]{MelrosePiazza}).
We can extend $E$ to the superbundle
\begin{equation*}
	\bbE = \psi^*\Lambda^*B\otimes E,
\end{equation*}
which is a Clifford module with respect to \eqref{DegMetric} and \eqref{tensor} acting via
\begin{equation}\label{DegClif}
	m_0\lrpar \alpha
	= \df e\lrpar{ \hpsi \alpha} + \cl{\vpsi \alpha}
\end{equation}
with the connection 
\begin{equation}\label{NablaE0}
	\nabla^{\bbE,0}
	=\psi^*\lrpar{\nabla^B} \oplus \nabla^E+\frac12m_0\lrpar{\omega_\psi}.
\end{equation}

The connection $\nabla^{\bbE,0}$ is a Clifford connection with respect to $\nabla^{M,0}$, so we will need to know about the curvature of the latter.
Denote by $R^{M/B}$ the curvature of $\nabla^{\fD}$.  From the choice of metric \eqref{geps}, it is easy to see that, for any $X$, $Y$ in $\Gamma(TM)$,
\begin{equation*}
	\vpsi R^{M,\eps}\lrpar{X,Y} \vpsi = R^{M/B}\lrpar{X,Y}, \quad
	\hpsi R^{M,\eps}\lrpar{X,Y} \hpsi = \psi^*R^B\lrpar{X,Y}
\end{equation*}
and both of these tensors are independent of $\eps$. 
It is convenient to extend these tensors trivially to sections of 
 $\Lambda^2T^*M \otimes \Lambda^2T^*M$ by defining
\begin{equation*}
	R^{M/B}\lrpar{W,X}\lrpar{Y,Z} := 
	R^{M/B}\lrpar{W,X}\lrpar{\vpsi Y, \vpsi Z},
\end{equation*}
then Proposition 10.9 of \cite{BGV} shows that the tensor
\begin{equation*}
	\lim_{\eps\to0} \lrpar{ R^{M,\eps}\lrpar{W,X}\lrpar{Y,Z}
		-\frac1\eps \psi^*R^B\lrpar{W,X}\lrpar{Y,Z} }
\end{equation*}
is a well-defined element of $\Lambda^2T^*M\otimes\Lambda^2T^*M$ with all of the symmetries of a curvature tensor and equal to $R^{M/B}$ when $Y$ and $Z$ are $\psi$-vertical.
Thus
\begin{equation*}
	R^{M,0}\lrpar{T_1,T_2}\lrpar{T_3,T_4}
	= \lim_{\eps\to0}
	R^{M,\eps}\lrpar{T_1,T_2}\lrpar{T_3,T_4}
\end{equation*}
is well-defined whenever $\h\psi T_i=0$ for some $i$ (as it will in all the cases we consider).
Finally we point out that Lemma \ref{ConnectionAsymp} holds for $R^{M,\eps}$ and hence for $R^{M,0}$.

The Dirac operator on $\bbE$ associated to the degenerate Clifford action $m_0$ and the Clifford connection $\nabla^{\bbE,0}$
is known as the {\bf Bismut superconnection} and denoted $\bbA$. Its action on $\CI\lrpar{M;E}$ is given by
\begin{equation*}
	\bbA=\bbA_{[0]}+\bbA_{[1]}+\bbA_{[2]}
\end{equation*}
with $\bbA_{[j]}:\CI\lrpar{M;E}\to\CI\lrpar{M;\psi^*\Lambda^jB\otimes E}$ and $\bbA_{[0]}=\eth^E_{\fD}$ the family of Dirac operators associated to $\lrpar{E,\nabla^E}$.
More explicitly,
\begin{equation*}\begin{split}
	\bbA 
	&=\cll^i\nabla_i^{\bbE,\oplus}+\df e^\alpha\nabla_\alpha^{\bbE,\oplus} 
		+ \frac14\lrpar{\omega_\psi}_{abc}m_0^am_0^bm_0^c\\
	&=\cll^i\nabla^{E}_i+\df e^\alpha\lrpar{\nabla_\alpha^E+\frac12k_\psi\lrpar{f_\alpha}}
	-\frac14\sum_{\alpha<\beta}\Curv\psi\lrpar{f_\alpha,f_\beta}\lrpar{e_i}\df e^\alpha \df e^\beta \cll^i
\end{split}\end{equation*}
where $\{ f_\alpha\}$, $\{ e_i \}$ are orthonormal frames for $TB$ and $\FD T\lrpar{M/B}$ respectively, $\df e$ denotes exterior multiplication, and $k_\psi$ is the trace of the second fundamental form tensor $\Sec\psi$ in \eqref{SecPsi}.

More generally,
a {\bf $\fD$-superconnection} on $\bbE$ is a $\fD$ differential operator
\begin{equation*}
	\CI\lrpar{M;\bbE} \xrightarrow{A} \CI\lrpar{M;\bbE}
\end{equation*}
of odd parity such that, for any $k$-form $\alpha$ on $B$ and any section $u$ of $\bbE$,
\begin{equation*}
	A\lrpar{\lrpar{\phi^*\alpha}u}
	= \lrpar{\phi^*\lrpar{d\alpha}}u + \lrpar{-1}^k\lrpar{\phi^*\alpha}A u.
\end{equation*}
$A$ is said to extend $A_{[1]}$ and to be adapted to $A_{[0]}$.

The square of the $\fD$ Bismut superconnection satisfies a Lichnerowicz formula,
\begin{equation}\label{Lich}
	\bbA^2 = \Delta_{\mathrm{Bochner}}^{M/B}+\frac14\mathrm{scal}_{M/B} -\frac12K_E'\lrpar{e_a,e_b}m_0^am_0^b.
\end{equation}
Here $\Delta_{\mathrm{Bochner}}^{M/B}$ is the Bochner or connection Laplacian associated to the connection $\nabla^{\bbE,0}$, and as such is a family of second order elliptic $\fD$ differential operators; $\mathrm{scal}_{M/B}$ is the scalar curvature of the fibers (i.e., of the family of exact $\fD$ metrics $g_{M/B}$) and $K_E'$ is the twisting curvature of the Clifford module $E$. 

We now consider the restriction to the boundary of $\bbA$. At the boundary we have three related fibrations 
\begin{equation*}
	\xymatrix{ 
	Z \ar@{-}[r] & \pa M \ar[d]_{\wt \phi} \ar@/^2pc/[dd]^{\pa \psi}\\
	Y \ar@{-}[r] & D \ar[d]_{\psi_D} \\
		& B }
\end{equation*}
and we have chosen connections that identify $\FD TM\rest{\pa M}$ with
\begin{equation*}
	 T\pa M/D \oplus TD^+/B \oplus TB.
\end{equation*}
The tensor $\Sec\psi$ is given, for any $\eps>0$, by
\begin{equation*}
	\Sec\psi(\v\psi A, \v\psi B)(\h\psi C)
	= g_M(\eps)( \nabla_{\v\psi A}^{\fD,\eps} \h\psi C, \v\psi B),
\end{equation*}
and since $\nabla^{\fD,\eps}$ satisfies \eqref{PreserveHor} this vanishes at the boundary whenever $B$ is $\wt \phi$-vertical (hence whenever $A$ is $\wt\phi$-vertical). It follows that
\begin{equation*}
	\Sec\psi \rest{\pa M} = \Sec{\psi_D}.
\end{equation*}
Similarly from \eqref{CurvfD=x} we see that
\begin{equation*}
	\Curv\psi \rest{\pa M} = \Curv{\psi_D}.
\end{equation*}
Thus, when we restrict to the boundary,  all  the information about $\wt\phi$
is lost. In the next section we explain how a rescaling of the bundles (as in \S\ref{sec:HeatKerOne}) allows us to recover this information for the `rescaled normal operator'.

Finally, we point out that Vaillant's description of $\eth^{H,\cK}_{\fD}$ (\S\ref{sec:DiracTypeOps}) extends to the families context.
Indeed, Vaillant obtains the geometrical description of $\eth^{H,\cK}_{\fD}$ through the identity \eqref{HorizontalComp} 
which is still valid with $\nabla^{\bbE,0}$ and $\nabla^{M,0}$.
To be more explicit, from above we know that $\omega_\psi \rest{\pa M} = \omega_{\psi_D}$ and hence $\nabla^{M,0} \rest{\pa M} = \nabla^{D,0}$.
Also, using $g_M(\eps)$ we can define $\Sec{\wt\phi-hc}$, $\Sec{\wt\phi}$, $\cB_{\wt\phi-hc}$, and $\Curv{\wt\phi}$ as above, and these satisfy \eqref{TwoSecs'}, \eqref{SecPsi}, \eqref{CurvPsi} and so are independent of $\eps$. 
Then, denoting the trace of $\Sec{\wt\phi}$ by $k_{\wt\phi}$, we define a superconnection by
\begin{equation}\label{HorizontalSupCon}
	\bbB_{\psi_D^+}
	=\textstyle \cl{\frac{dx}x} \displaystyle x\pa_x + \bbB_{\psi_D}
	=m_0^{\psi_D} \cP_{\cK}
	\lrpar{\nabla^{\bbE,0}  + \frac12 k_{\wt\phi} },
\end{equation}
and from \eqref{HorizontalComp} we find
\begin{equation*}
	\bbA_{\fD}\rest{\cK, \pa M} = \bbB_{\psi_D} + \cJ_{\wt\phi}.
\end{equation*}
Here the endomorphism $\cJ_{\wt\phi}$ is given by
\begin{multline*}
	\cP_\cK\left(
	\v\phi\cl{K^{E'}\lrpar{dx^\flat,\cdot}} 
	+\frac12\v\phi \cll_{\fD}\lrpar{\Ric^V_{M/B}\lrpar{dx^\flat,\cdot}}
	\right. \\ \left.
	+\frac14\cll_{\fD}\lrpar{ \cB_{\wt\phi-hc}\lrpar{x\cdot,x\cdot}\lrpar{\cdot}} 
	-{\v\phi\cl{\bar\theta^i}}\nabla^{\bbE,0}_{x\nabla^{\fC}_{\bar X_i} dx^\flat}
	\right)
\end{multline*}
with $\Ric^V_{M/B}$ the $\phi$-vertical contraction of the $\psi$-vertical curvature $R^{M/B}$.
As before, the last three terms in $\cJ_{\wt\phi}$ vanish for product-type metrics and the first term vanishes if $M$ is spin and the twisting bundle has product-type metric at the boundary.
We point out that $\cJ_{\wt\phi}$ does not involve any exterior products (essentially because $\cB_{\wt\phi-hc}$ vanishes on elements of $TB$ for any $\eps>0$).

\subsection{The rescaled normal operator} \label{sec:ResDouble} $ $ \newline
As we have mentioned, differential $\fC$ operators (say on $X$) were defined in \cite{MazzeoMelrose} as well as a larger $\fC$ pseudo-differential calculus.
Pseudo-differential $\fC$ operators are defined by describing their integral kernels. This can be done on the space $X^2$, however there is a more complicated space $X^2_{\fC}$ with a natural map $X^2_{\fC}\to X^2$ such that the integral kernels pulled-back to $X^2_{\fC}$ have a much simpler description than on $X^2$.

The space $X^2_{\fC}$ is obtained from $X^2$ by blowing-up the submanifolds where the kernels are complicated. Here blow-up (radial blow-up in \S\ref{sec:HeatKerOne}) refers to the process of replacing a submanifold of the boundary of $X^2$ by its inward-pointing spherical normal bundle - as described for instance in \cite{APS Book}. The simplest example is to introduce polar coordinates around the origin in $\bbR_x\times\bbR_{x'}$. Polar coordinates take values in the space $\bbR^+\times\bbS^1$ obtained from $\bbR^2$ by replacing the origin with $\{0\}\times \bbS^1$. A function like $\arctan\lrpar{x'/x}$, which is singular on $\bbR^2$ lifts to a smooth function on $\bbR^+\times\bbS^1$. The notation used for this is
\begin{equation*}
	\lrspar{\bbR^2; \lrpar{0,0} }
\end{equation*}
where we indicate the original manifold and the submanifold that is blown-up.

For a vertical family of $\fC$ pseudo-differential operators the space we need is
\begin{equation}\label{M2Phi}
	M^2_{\fC,\psi}:=\lrspar{ M\btimes_B M; \pa M \btimes_B \pa M; \pa M \btimes_D \pa M},
\end{equation}
where the notation indicates that we start with the fibered product of $M$ with itself over $B$,
\begin{equation*}
	M \btimes_B M = \lrbrac{ \lrpar{m,m'}\in M^2: \psi\lrpar m = \psi\lrpar{ m' } },
\end{equation*}
then we blow-up (introduce polar-coordinates around) the submanifold \newline $\pa M\btimes_B\pa M$, and then we blow-up the lift of the sub-manifold $\pa M \btimes_D \pa M$.
A theorem of Melrose guarantees that the result is still a manifold with corners (because we have only blown-up `p-submanifolds', see \cite{Melrose:Corners}).
The blow-up construction comes with a natural blow-down map which we denote 
\begin{equation*}
	M^2_{\fC,\psi}\xrightarrow{\beta_\psi} M\btimes_B M.
\end{equation*}

Integral kernels of pseudo-differential operators on closed manifolds have a conormal singularity at the diagonal (i.e., their transverse Fourier transform is a symbol). For $\fC$ pseudo-differential operators this is true at $\diag_{\fC,\psi}$, the closure of the lift of the diagonal from the interior of $M\btimes_\psi M$ to the interior of $M^2_{\fC}$. An advantage of the blown-up space is that $\diag_{\fC,\psi}$ does not hit the boundary at a corner, but rather at the interior of the `front face', the boundary face produced by the second blow-up in \eqref{M2Phi}. We will denote this face by $\bhs{\phi\phi}$.

We can now define {\bf vertical $\fC$ pseudo-differential operators} of order $k$ acting between sections of two bundles $F^1$ and $F^2$, 
\begin{equation}\label{FCPseudo}
	\Psi^k_{\fC}\lrpar{M/B;F^1,F^2}
\end{equation}
by describing their integral kernels.
These are distributions on $M^2_{\fC,\psi}$, valued in $\beta_\psi^*\Hom\lrpar{F^1,F^2}$,
that are smooth away from the boundary and the diagonal, have a conormal singularity at $\diag_{\fC,\psi}$ of order $k$ and have an asymptotic expansion at $\bhs{\phi\phi}$ in powers of a boundary defining function for that face.
Recall that the bundle $\Hom\lrpar{F^1,F^2}$ has fiber over the point $\lrpar{\zeta,\zeta'}$ given by 
$\Hom\lrpar{ F^1_\zeta, F^2_{\zeta'}}$.

{\bf Remark.}  It is convenient for the integral kernels to take values in half-density bundles, but for simplicity we do not discuss this here.

We are interested in the case where $F^1=F^2=\bbE$, so the distributions will be valued in $\beta_\psi^*\Hom\lrpar E \otimes \psi^*\Lambda^*B$, which we denote $\Hom\lrpar{\bbE}$.
In this case, we can refine the pseudo-differential operators by demanding that the $k^{\text{th}}$ coefficient in the expansion at the front face $\bhs{\phi\phi}$ has Clifford degree at most $k$ with respect to the Clifford algebra of
\begin{equation*}
	 T^*D^+/B : = {}^b T^*(D\times\bbR^+)/B \rest D.
\end{equation*}
This is a Getzler rescaling argument, but applied to the whole calculus of pseudo-differential operators. This technique was introduced in \cite[Chapter 8]{APS Book} which we follow closely.

Since $\bhs{\phi\phi}$ lies over the fibered diagonal of the boundary, and $E$ is a Clifford module, we have 
\begin{equation*}
	\Hom\lrpar E \rest{\phi\phi}
	\cong \Cl^*\lrpar{T^*D^+/B} \otimes \Hom'\lrpar{E},
\end{equation*}
where $\Hom'\lrpar{E}$ denotes homomorphisms that (anti)commute with Clifford multiplication by elements of $\Cl^*\lrpar{T^*D^+/B}$,
thus the restriction of $\Hom\lrpar{\bbE}$ to $\bhs{\phi\phi}$ admits a filtration by Clifford degree.

We define for $k\leq b+h+1$ (where $b=\dim B$ and $h=\dim D/B$)
\begin{equation}\label{FiltPhi}
	\Cl^k_{\phi\phi} := 
	\underset{i+j\leq k}\oplus \psi^*\Lambda^iB \otimes \Cl^j\lrpar{T^*D^+/B} \otimes \Hom'\lrpar E 
\end{equation}
Pick a vector field $\nu$ on $M^2_{\fC,\psi}$, transverse to $\bhs{\phi\phi}$ such that $\nu x=1$ and use the connection on $\Hom\lrpar{\bbE}$ induced by $\nabla^{\bbE,0}$,
\begin{equation*}
	\nabla^{\Hom\lrpar{\bbE}}
	:= \beta^*_L\nabla^{\bbE,0} \otimes \beta_R^*\nabla^{\bbE',0},
\end{equation*}
(so that $\nabla^{\Hom\lrpar{\bbE}} A = \lrspar{\nabla^{\bbE,0}, A}$)
to define
\begin{equation*}
	\cD = \lrbrac{u\in \CI\lrpar{X^2_{\fC}, \psi^*\Lambda^*B \otimes \beta^*\Hom\lrpar{E}} 
		: \lrpar{\nabla^{\Hom\lrpar{\bbE}}_\nu}^ku\rest{\bhs{\phi\phi}} \in \Cl^k_{\phi\phi} }.
\end{equation*}
As in \cite{APS Book} we can identify $\cD$ with the sections of a bundle $\Gr$ over $M^2_{\fC,\psi}$. The connection $\nabla^{\Hom\lrpar{\bbE}}$ has the following properties with respect to the filtration \eqref{FiltPhi}.

\begin{lemma}\label{ConnectionAction} $ $
\begin{itemize}
	\item [i)] The connection induced on the boundary preserves the filtration.
	\item [ii)] If $W$ is a vector field on $M^2_{\fC,\psi}$ tangent to all of the boundary hypersurfaces then
	\begin{equation*}
		\lrpar{\nabla^{\Hom\lrpar{\bbE}}_\nu}^p K^{\Hom\lrpar{\bbE}}\lrpar{\nu,W}\rest{\bhs{\phi\phi}} 
			: \Cl^k_{\phi\phi} \to \Cl^{k+2}_{\phi\phi}
	\end{equation*}
	\item [iii)] If $U \in \beta^*_{\psi,L} \Gamma(\FC TM/B)$ then
	\begin{equation*}
		K^{\Hom\lrpar{\bbE}}\lrpar{\nu,U}\rest{\bhs{\phi\phi}} 
			: \Cl^k_{\phi\phi} \to \Cl^{k+1}_{\phi\phi}
	\end{equation*}
\end{itemize}
\end{lemma}

\begin{proof}
The first item is true because $\nabla^{\bbE,0}$ is a Clifford connection with respect to $\nabla^{M,0}$ and the Clifford algebra $\Cl_0\lrpar M$, so it preserves the degree in $\Cl_0\lrpar{D}:=\Lambda^*B \otimes \Cl\lrpar{TD^+/B}$.

The second item is true because at $\bhs{\phi\phi}$ the splitting of the $\Hom$ bundle induces a splitting of the curvature into 
\begin{equation*}
	K^{\Cl_0\lrpar D} \otimes \Id + \Id \otimes K^{\Hom'\lrpar{\bbE}},
\end{equation*}
the second summand does not change the filtration nor do its normal covariant derivatives, while the first summand acts by Clifford multiplication by two-forms
\begin{equation}\label{KCl}
	K^{\Cl_0\lrpar D}\lrpar{\nu,W}
	=\frac14 \lrspar{ \beta_L^*\cl{R^{D^+,0}\lrpar{\nu,W}} -
	 \beta^*_R\cl{R^{D^+,0}\lrpar{\nu,W}}}
\end{equation}
and hence so do its normal covariant derivatives.

Finally from \eqref{KCl} we see that, although $K^{\Hom\lrpar{\bbE}}\lrpar{\nu,U}$ could move the filtration by two, the extra degree of vanishing of $U$ at $\bhs{\phi\phi}$  (because $U \in \Gamma(\FC TM/B)$ instead of $\Gamma(\FD TM/B)$) implies that it only moves it by one.
\end{proof}

An element of $\cD$ has an expansion at $\bhs{\phi\phi}$
\begin{equation*}
	u \sim u_0 + xu_1 + \ldots + x^{h+b+1}u_{h+b+1} + \cO\lrpar{x^{h+b+2}}
\end{equation*}
with $u_j$ satisfying
\begin{equation}\label{Taylor}
	u_j = \frac1{j!}\lrpar{\nabla_\nu^{\Hom\lrpar{\bbE}}}^p u\rest{\bhs{\phi\phi}} 
		\in \CI\lrpar{\bhs{\phi\phi},\Cl^j_{\phi\phi}}.
\end{equation}
The restriction map to $\bhs{\phi\phi}$ as elements of $\cD$ can be identified with the projection to $\cD / x\cD$, so it is equal to
\begin{equation}\label{Restriction}
	u \mapsto \lrpar{ u_0, \lrspar{u_1}, \ldots, \lrspar{u_{h+b+1}}}
\end{equation}
with $\lrspar{u_j}$ representing the class of $u_j$ in $\Cl_0^k(D)/\Cl_0^{k-1}(D)$. The map \eqref{Restriction} is the difference between the {\bf rescaled normal operator} and the usual normal operator (which would just take $u_0$). The exterior algebra is the graded algebra associated to the Clifford filtration, i.e., for any vector space $V$, we have $\Cl^k\lrpar V/\Cl^{k-1}\lrpar V \cong \Lambda^kV$. Thus the $k^{\text{th}}$ term in the restriction map can be naturally identified with an element of $\Lambda^kT^*D^+$ and the restriction itself as taking values in $\Lambda^*T^*D^+\rest D$.

According to \cite[Lemma 8.10]{APS Book}, Lemma \ref{ConnectionAction} implies that vector fields $U \in \FC TM/B$ act on $\cD$,
\begin{equation*}
	\nabla^{\Hom\lrpar{\bbE}}_{\beta^*_{L,\psi}U} : \cD \to \cD.
\end{equation*}
One can think of this as defining a partial connection on $\Gr$. What is more, \cite[Proposition 8.12]{APS Book} gives a formula for the action of $\nabla^{\Hom\lrpar{\bbE}}_{\beta^*_{L,\psi}U}$:
\begin{multline}\label{RescaledAction}
	\lrspar{\nabla^{\Hom\lrpar{\bbE}}_{\beta^*_LU}u}_j
	= \lrspar{\nabla^{\Hom\lrpar{\bbE}}_{\beta^*_LU}\rest{\bhs{\phi\phi}}u_j 
		+ \lrpar{h+b+1-j}n_Uu_j}_j \\
	+\lrspar{K^{\Hom\lrpar{\bbE}}\lrpar{\nu,U}\rest{\bhs{\phi\phi}}u_j}_{j-1} \\
	+\lrspar{\frac12\nabla^{\Hom\lrpar{\bbE}}_\nu K^{\Hom\lrpar{\bbE}}\lrpar{\nu,U}\rest{\bhs{\phi\phi}}u_j}_{j-2},
\end{multline}
where $n_U$ is given by $\nu Ux\rest{\bhs{\phi\phi}}$. Note that if $U \in \FC TM/B$ then $Ux$ vanishes to second order at the boundary and hence $n_U=0$.

We define the {\bf rescaled $\fC$ pseudo-differential operators}
\begin{equation}\label{GFCPseudo}
	\Psi^k_{\fC,G}\lrpar{M/B;\bbE}
\end{equation}
by replacing $\beta_\psi^*\Hom\lrpar{\bbE}$ in the description of \eqref{FCPseudo} with $\Gr$. The rescaled normal operator is a surjective map
\begin{equation}\label{RescaledNormal}
	\Psi^k_{\fC,G}\lrpar{M/B;\Lambda^*B \otimes E} \xrightarrow{N^G_{\phi\phi}}
	\Psi^k_{sus}\lrpar{\FC N\pa M/B; \Lambda^*_xD^+ \otimes E}
\end{equation}
with null space $x\Psi^k_{\fC,G}\lrpar{M/B;\Lambda^*B \otimes E}$.

Before computing the rescaled normal operator of $\nabla^{\bbE,0}$ we establish some notation. 
Given a vector field $V \in \Gamma(\FD TM/B)$ we denote by $V^\sharp$ the one form dual to $V$ with respect to $g_{\fD}$. 
We denote by $m_0^{\wt\phi}$ the degenerate Clifford action, defined on $\alpha\in\Gamma(T^*M)$ by
\begin{equation*}
	m_0^{\wt\phi}(\alpha) = \df e(\h{\wt\phi}\alpha) + \cl{\v{\wt\phi}\alpha}.
\end{equation*}
Recall (from the end of \S\ref{sec:Bismut}) the tensors $\Sec{\wt\phi-hc}$, $\Curv{\wt\phi}$, and $\cB_{\wt\phi-hc}$. We define $\omega_{\wt\phi-hc}$ by substituting $\Sec{\wt\phi-hc}$ and $\Curv{\wt\phi}$ into formula \eqref{OmegaPsi}.

\begin{proposition}\label{RescaledNormalOp} $ $
{\bf a)} For a vector field $U\in \FC TM/B$,
\begin{multline*}
	N^G_{\phi\phi}\lrpar{\nabla^{\bbE,0}_U} \\
	=\begin{cases}
	 \sigma\lrpar U +  
		\frac14 \df e\lrpar{ R^{D^+/B}\lrpar{\cR,U}} + \frac12\df e\lrpar{Q_{\phi\phi}},
	& \text{ if } U\in\lrspar{\FC N\pa M/B} \\
	\nabla^{\bbE,\pa M/D}_U +\df e\lrpar{\frac{dx}x} \cl{\frac1xU^\sharp}
	& \text{ if } U\in\lrspar{\FC V\pa M/B} 
	 \end{cases}
\end{multline*}
for a form $Q_{\phi\phi}$ defined below \eqref{DefQ112} and
where $\nabla^{\bbE,\pa M/D}$ is given by 
\begin{equation}\label{RescaledVerticalConn}
	\nabla^{\bbE,\pa M/D,\oplus} + m_0^{\wt\phi}\lrpar{\omega_{\wt\phi-hc}} .
\end{equation}

{\bf b)} The rescaled normal operator of the $\fD$ Bismut superconnection $\bbA$ is
\begin{equation}\label{NPhiPhi}
\begin{split}
	N^G_{\phi\phi}\lrpar{x^2\bbA^2_{\fD}} 
	&=\cH_{D^+/B} + \bbB_{\pa M/D,\wt\phi-hc}^2 \\
	&\phantom{xxx}
	+\df e\lrpar{\tfrac{dx}x}
	\lrspar{\eth^V_{\fD} - m_0^{\wt\phi}\lrpar{\cB_{\wt\phi-hc}} 
	- \frac14 m_0^{\wt\phi}\lrpar{\Curv{\wt\phi}} } \\
	&=: \cH_{D^+/B} + \bbB_{\pa M/D,+}^2.
\end{split}
\end{equation}
Here $\bbB_{\pa M/D,\wt\phi-hc}$ is equal to the Bismut superconnection for $\wt\phi$ where $\Sec{\wt\phi}$ has been replaced by $\Sec{\wt\phi-hc}$, $\cH_{D/B}$ is the harmonic oscillator on the fibers of $\FC N\pa M/B$, given by
\begin{equation*}
	-\sum \lrpar{V_i - \frac14 \df e\lrpar{R^{D^+/B}\lrpar{\cR,V_i}
		+2 Q_{\phi\phi}\lrpar{V_i} }}^2
\end{equation*}
with $V_i$ ranging over a orthonormal frame for $TD/B$, and $\cR$ the radial vector field on the fibers.
\end{proposition}
 
{\bf Remark.} We point out that the first two terms in \eqref{NPhiPhi} commute with $\df e\lrpar{\frac{dx}x}$. This will be used in the computation of the small-time asymptotics of the heat kernel (Lemma \ref{lem:Smallt}).
 
\begin{proof}

{\bf (a)}
All that is left to do is compute \eqref{RescaledAction} mostly using \eqref{Prop1.14a} and \eqref{Prop1.14b} which, as pointed out in \S\ref{sec:Bismut}, hold for $R^{M,0}(\nu, U)$ since $\nu$ and $U$ are $\psi$-vertical.

The first term in \eqref{RescaledAction} is just the usual normal operator of the connection which can be easily computed (e.g., by oscillatory testing) and seen to be the vertical connection for vertical vector fields, while for a horizontal vector field, $U$, it is equal to its principal symbol,
$\sigma\lrpar U$, thought of as a constant coefficient vector field on the fibers of $\FC N\pa M$.

The second term, vanishes for $U\in\FC N\pa M/B$ since $U$ vanishes at the boundary, while if $U \in T\pa M/D$, then by \eqref{Prop1.14a} equals (recall that $\nu x=1$)
\begin{equation*}
	\frac14 m_0^{\wt\phi}\lrpar{R^{M,0}\lrpar{\nu,U}} 
	= \frac12 \sum_{i,\alpha} S^+_{\wt\phi}\lrpar{U,e_i}\lrpar{e_\alpha} 
		\cl{e^i} \df e\lrpar{e^\alpha}
	+ \frac12 \textstyle \df e\lrpar{\frac{dx}x} \cl{\frac1x U^{\sharp}} 
\end{equation*}
where $e_i$ ranges over a basis for $T\pa M/D$, $e_\alpha$ over a basis for $TD$, and we raise indices to indicate the dual one-forms.

Computing the last term in \eqref{RescaledAction} for $U \in \FC N\pa M/B$ can be done recalling that $\bhs{\phi\phi}$ is a bundle over $\pa M \btimes_D \pa M$, and comparing the value at any point with the value at the zero section. This can be computed as in \cite[Lemma 5.15]{Vaillant} and equals
\begin{gather} \label{DefQ112}
	\sum_{i<j} 
	{\bf R}_{\psi_D}\lrpar{\cR+\nu\rest 0,U}\lrpar{Y_1,Y_2} 
		\df e\lrpar{Y_1^{\sharp}\wedge Y_2^{\sharp}} \\
	=: \df e\lrpar{ {\bf R}_{\psi_D}\lrpar{\cR,U} } + \df e\lrpar{Q_{\phi\phi}\lrpar U}, \notag
\end{gather}
with $Y_i$ a frame for $TD$; just as in \cite{BGV} we can replace ${\bf R}_{\psi_D}$ in the last expression with $R^{D^+/B}$.
Finally, this same term for $U \in T\pa M/D$ is computed in \eqref{Prop1.14b} to be
\begin{equation*}
	\sum_{i<j} 
	\Curv{\wt\phi}\lrpar{Y_1,Y_2}\lrpar U \df e\lrpar{Y_1^{\sharp}\wedge Y_2^{\sharp}}
	= \df e\lrpar{\Curv{\wt\phi}\lrpar{\cdot,\cdot}\lrpar U} .
\end{equation*}

{\bf (b)}
We start with the Lichnerowicz formula \eqref{Lich},
\begin{equation*}
\bbA^2_{\fD} = \Delta_{\mathrm{Bochner}}^{M/B}+\frac14\mathrm{scal}_{M/B} -\frac12K_E'\lrpar{e_a,e_b}m_0^am_0^b.
\end{equation*}
Notice that the last two terms will contribute $\scal_{\pa M/D}$ and the twisting curvature of the fibration $\wt\phi$. The first term will split into the harmonic oscillator on the fibers of $\FC N\pa M/B$ and the Bochner Laplacian of $\nabla^{\bbE,\pa M/D}_\cdot +\df e\lrpar{\frac{dx}x} \cl{\frac1x\cdot^{\#}}$ along $\pa M/D$. Using the Lichnerowicz formula again, we can reassemble the square of the superconnection along $\pa M/D$, then finally note that 
\begin{multline*}
	-\sum_{e_i}
	\lrpar{ \nabla^{\bbE,\pa M/D}_{e_i} +\frac12 \df e\lrpar{\tfrac{dx}x} \cl{\tfrac1xe^i} }^2\\
	=\lrpar{\nabla^{\bbE,\pa M/D}_{e_i}}^2
	+\df e\lrpar{\tfrac{dx}x}
	\lrspar{\eth^V_{\fD} - m_0^{\wt\phi}\lrpar{\cB_{\wt\phi-hc}} 
	- \frac14 m_0^{\wt\phi}\lrpar{\Curv{\wt\phi}} },
\end{multline*}
in the last term the contribution from  $\Sec{\wt\phi}$ cancels out \cite[Prop. 2.10]{BGS}.
\end{proof}

The superconnection $\bbB_{\pa M/D,\fD}$ is a $\Cl(1)$ superconnection with respect to Clifford multiplication by $\cl{\frac{dx}x}$ in the sense of \cite[\S 10]{MelrosePiazza}. It can even be thought of as a $\Cl(h+1)$ connection with respect to $\Cl(T^*D^+/B)$.

{\bf Remark.}
Although $\nabla^{\bbE,0}$ and $\bbA$ are both in $\Diff^1_{\fD}$, so that $x\nabla^{\bbE,0}$ and $x\bbA_{\fD}$ are $\fC$ differential operators, the former is  also an element of the rescaled calculus, while the latter is not. It is true that $x^2\bbA_{\fD}$ is an element of the rescaled calculus, and so one might expect to have to multiply $\bbA^2_{\fD}$ by $x^4$ to get an element of the rescaled calculus. One advantage of using the Lichnerowicz formula is that we see that in fact $x^2\bbA^2$ rescales. 

Finally, we comment on the normal operators for perturbed $\fD$ Dirac-type operators.
A small perturbation of $\eth_{\fD}$ does not affect the rescaled normal operator.
However, for a large perturbation of $\eth_{\fD}$ associated to a spectral section $P$, we replace every occurrence of $\eth^V_{\fD}$ by $\eth^V_{\fD}+A_P^{\eps}$. Significantly, the resulting Bismut superconnection is still an element of the rescaled calculus and the rescaled normal operator is only changed by this replacement. The reason is that $A_P^{\eps}$ anti-commutes with $\Cl(T^*D^+/B)$ and so does not interfere with the rescaling.

\section{The $\fD$ families index theorem}

The heat kernel of the Bismut superconnection will be described as a function on a fibrewise heat space. The rescaling argument in $\S$\ref{sec:ResDouble} gives a rescaling at one boundary face in the heat space; there is another rescaling along the diagonal in the interior as time goes to zero.
We describe the heat kernel of both perturbed and unperturbed Bismut superconnections.

The interior rescaling is carried out in \cite[Chapter 8]{APS Book}. It is a geometric extension of the usual Getzler rescaling, in this case we mean e.g., Theorem 10.21 from \cite{BGV} 
which says that for a fibration of closed manifolds the heat kernel has an expansion along the diagonal as $t\to0$
\begin{equation}\label{BGVObjective}
	k_t\lrpar{\zeta,\zeta'} \sim t^{-n/2}\sum t^i k_i
\end{equation}
in which each $k_i$ has Clifford degree in $\Cl_0$ at most $2i$ and furthermore includes an expression for its `full symbol'.

\subsection{The rescaled heat calculus} $ $ \newline
Just as in \$\ref{sec:HeatKerOne} and \S\ref{sec:ResDouble}, the analysis of the heat kernel in \cite{DaiMelrose} and \cite{APS Book} proceeds by blowing-up submanifolds to resolve analytic difficulties.
The heat kernel of the Euclidean Laplacian is given, for $t>0$, as a function on 
$\bbR^n_\zeta \times \bbR^n_{\zeta'} \times \bbR^+_t$ by
\begin{equation*}
	\cK_t\lrpar{\zeta,\zeta'} = \frac1{\lrpar{4\pi t}^{n/2}} \exp\lrpar{-\frac{\abs{\zeta-\zeta'}^2}{4t}}.
\end{equation*}
So the blow-up proceeds by introducing `polar coordinates' around the submanifold $\{\zeta=\zeta', t=0\}$ that are homogeneous with respect to the scaling $\lrpar{\zeta,\zeta',t}\mapsto\lrpar{\lambda\zeta, \lambda\zeta', \lambda^2t}$, such as
\begin{equation}\label{Parabolic}
	\rho = \lrpar{\abs{\zeta-\zeta'}^4+t^2}^{1/4},
	\quad
	\omega = \lrpar{\frac{\zeta-\zeta'}\rho, \frac t{\rho^2}}, 
	\quad \zeta'.
\end{equation}
We refer to this process as {\bf parabolic blow-up} and refer the reader to \cite{APS Book}, \cite{DaiMelrose} for the details of this construction. The notation used for the resulting manifold is
\begin{equation*}
	\lrspar{ (\bbR^n)^2 \times\bbR^+_t; \{\zeta=\zeta', t=0\}, \ang{dt}}
\end{equation*}
where we indicate the original manifold, the submanifold blown-up and the parabolic directions used in the blow-up. A parabolic blow-up without any parabolic directions is the blow-up used in $\S$\ref{sec:ResDouble}.

For the fibration in \eqref{fullfib}, the appropriate heat space is obtained from 
$M \btimes_\psi M \times\bbR^+ \to B$ by blowing-up, on each fiber, the 
`spatial corner' \newline 
$\pa M \btimes_\psi \pa M \times \bbR^+$, the fiber diagonal at the boundary at time zero
\newline $\pa M \btimes_{\wt \phi} \pa M \times \{0\}$ (at time zero because here `the heat arrives in finite time'), and the spatial diagonal at time zero, $\diag_M \times \{0\}$. Thus the heat space, $HM_{\fD,\psi}$ is given by the iterated parabolic blow-up
\begin{multline}\label{HeatSpace}
	HM_{\fD,\psi} :=
	\left[ M \btimes_B M \times\bbR^+ ; \pa M \btimes_B \pa M \times \bbR^+; 
	\right. \\ \left.
	\pa M \btimes_D \pa M \times \{0\}, \ang{dt} 
	; \diag_M \times \{0\}, \ang{dt} \right] .
\end{multline}
As all of the operations are done fiber-by-fiber, we end up with a fibration
\begin{equation*}
	\xymatrix{HX_{\fD} \ar@{-}[r] & HM_{\fD,\psi} \ar[r]^(.6){\psi_D} & B }
\end{equation*}
where the fiber is the heat space from \S\ref{sec:HeatKerOne} used by Vaillant.

We denote the three new boundary hypersurfaces introduced by the blow-ups respectively by $\bhs{11,0}$, $\bhs{\phi\phi,2}$, and $\bhs{00,2}$. The notation indicates which variables vanish to define each face and to what order - thus $11,0$ indicates that we are at the boundary in both the left and right factors but not in the temporal factor, while $00,2$ indicates that we are at the temporal boundary but at neither of the spatial boundaries, and $\phi\phi,2$ indicates the fiber diagonal at the temporal boundary.

The interior of the boundary face $\bhs{\phi\phi,2}$ can be identified with 
\begin{equation*}
	\bhs{\phi\phi,2}^{\circ} \cong \FC N\pa M/B \btimes_D \pa M \times\bbR^+.
\end{equation*}
We will use as a time variable in the last factor the projective coordinate
\begin{equation}\label{Deftau}
	\tau := \frac t{x^2}.
\end{equation}

The heat space comes with a blow-down map
\begin{equation*}
	HM^2_{\fD,\psi} \xrightarrow{\beta} M \btimes_B M \times \bbR^+,
\end{equation*}
and we will denote by $\beta_L$ and $\beta_R$ the composition of $\beta$ with the projection onto the left or right factor of $M$, respectively.
Elements of the fibrewise heat calculus will be sections of 
\begin{equation*}
	\cH := \beta^*\Hom\lrpar{\bbE} : = \psi^*\Lambda^*B \otimes \beta^*\Hom\lrpar E.
\end{equation*}

Proceeding as in $\S$\ref{sec:ResDouble}, we note that
$\beta^*\Hom\lrpar{\bbE}$ admits filtrations at both $\bhs{\phi\phi,2}$ and $\bhs{00,2}$, since 
\begin{equation*}
	\beta^*\Hom\lrpar E\rest{\phi\phi,2}
	\cong \Cl\lrpar{T^*D^+/B} \otimes \Hom'_{\phi\phi,2}\lrpar E
\end{equation*}
and similarly
\begin{equation*}
	\beta^*\Hom\lrpar E\rest{00,2}
	\cong \Cl\lrpar{T^*M/B} \otimes \Hom'_{00,2}\lrpar E.
\end{equation*}
Just as in \eqref{FiltPhi}, we denote by $\Cl_0^k(D^+)$ the elements of $\Lambda^*B \otimes \Cl(T^*D^+/B)$ of total degree at most $k$ and similarly $\Cl_0^k(M)$.
We define
\begin{gather*}
	\Cl^k_{\phi\phi,2} := \Cl_0^k\lrpar{D^+} \otimes \Hom'_{\phi\phi,2}\lrpar E,
		\quad k\in\lrbrac{1,\ldots, b+h+1} \\
	\Cl^\ell_{00,2} := \Cl_0^\ell\lrpar M \otimes \Hom'_{00,2}\lrpar E,
		\quad\ell\in\lrbrac{1,\ldots, b+n} 
\end{gather*}
where $b=\dim B$, $n=\dim X$, and $h=\dim Y$, then we use the connection 
\begin{equation*}
	\nabla^\cH := \beta^*\nabla^{\Hom\lrpar{\bbE}} \otimes dt \frac{\pa}{\pa t},
\end{equation*}
and two vector fields, $\nu_{\phi\phi,2}$ and $\nu_{00,2}$, normal to $\bhs{\phi\phi,2}$ and $\bhs{00,2}$ respectively but tangent to all other boundary faces, to define the space of {\bf rescaled sections of $\cH$} by
\begin{multline*}
	\cD_H = \left\{ u \in \CI\lrpar{HM_{\fD,\psi},\cH} : \right. \\ \left.
	\lrpar{\nabla^{\cH}_{\nu_{00,2}}}^\ell u \rest{00,2} \in \Cl_{00,2}^\ell, \quad 
	\lrpar{\nabla^{\cH}_{\nu_{\phi\phi,2}}}^k u \rest{\phi\phi,2} \in \Cl_{\phi\phi,2}^k \right\}. 
\end{multline*}
The boundary hypersurface $\bhs{00,2}$ fibers over $\diag_M$, and we will require that $\nu_{\phi\phi,2}$ be tangent to the fibers of $\bhs{00,2}$; similarly, $\nu_{00,2}$ should be tangent to the fibers of $\bhs{\phi\phi,2}$ over $\pa M \btimes_D \pa M$.
As explained in \cite[Chapter 8]{APS Book}, there is a bundle $\GrP$ whose space of sections is $\cD_H$.

We denote by $\rho_{00,2}$ a boundary defining function for $\bhs{00,2}$ and use similar notations $\rho_{\phi\phi,2}$ and $\rho_{11,0}$ for the other
boundary faces.
An element of the heat calculus will be a smooth function on $HM_{\fD,\psi}$ valued in $\GrP$ (times a density bundle) with prescribed degree of vanishing at $\bhs{00,2}$, $\bhs{\phi\phi,2}$, and $\bhs{11,0}$ and vanishing to infinite order at all other boundary faces.
Thus for three integers $\ell>0$, $m\geq 0$, and $p\geq 0$ we define the {\bf fibrewise rescaled heat calculus} to be
\begin{equation}\label{RescaledHeat}
	\Psi^{\ell,m,p}_{RHeat}\lrpar{M/B,E}
	:= \rho_{00,2}^\ell \rho_{\phi\phi,2}^m \rho_{11,0}^{p-v}
	\dot{\mathcal{C}}^{\infty}\lrpar{HM_{\fD,\psi}, \GrP\otimes KD_H}.
\end{equation}
with $v=\dim \pa M/D$, $\dot{\mathcal{C}}^{\infty}$ denoting vanishing to infinite order at the other faces, and with the density bundle
\begin{equation}\label{KDH}
	KD_H = \frac{\lrpar{x'}^n}{t^{n/2}} \frac{dt}t \beta_{H,R}^*\FC\Omega(M/B),
\end{equation}
where $\FC\Omega(M/B)$ refers to the $\fC$ density bundle on the fibers of $\psi$.

{\bf Remark.}
An important aspect of this definition is that, for any fixed $t>0$, the integral kernels are $b$-densities, as pointed out in \cite[$\S$4.3, Lemma 5.26(b)]{Vaillant}.

\subsection{The heat kernel} $ $ \newline
We will construct the heat kernel for Bismut superconnections associated to Fredholm Dirac-type operators possibly with perturbations.
To avoid any problems in applying the local index theorem, and also avoid having to pass to an `extended' fibrewise heat calculus, we will multiply any perturbations by an appropriate cut-off function.
Let $\chi\lrpar t$ be a smooth function identically equal to zero near zero and equal to one for all $t\geq1$.
If $A_P^\eps$ is a large perturbation associated to a spectral section $P$ then we are interested in
\begin{equation}\label{CutOffA}
	\bbA_{\fD} + \chi\lrpar{\tfrac t{x^2}} \tfrac1x A_P^\eps,
\end{equation}
while if $Q_R^\eps$ is a small perturbation with associated spectral section $R$, we are interested in
\begin{equation}\label{CutOffB}
	\bbA_{\fD} + \chi(t) Q_R^\eps \cP_\cK.
\end{equation}
{\em We will be purposely vague and write $\wt \bbA_{\fD}$ to mean any of $\bbA_{\fD}$, \eqref{CutOffA}, or \eqref{CutOffB} and only specify when necessary.}
We now show that we can solve the heat equation
\begin{equation*}
	\lrpar{ t\pa_t + t\wt\bbA_{\fD}^2 } H = 0, \quad \lim_{t\to0} H = \Id
\end{equation*}
in the fibrewise rescaled heat calculus, by computing its rescaled normal operators.

It is easy to see that the rescaling at $\bhs{\phi\phi,2}$ is compatible with the rescaling at $\bhs{\phi\phi}$ in section $\S$\ref{sec:ResDouble} in that, for instance,
for a $\fD$ differential operator $B$ we have
\begin{equation*}
	N^G_{\phi\phi,2}\lrpar{t^{1/2}BA} 
	= \tau^{1/2}N^G_{\phi\phi}\lrpar{xB}N^G_{\phi\phi,2}\lrpar A,
\end{equation*}
thus the rescaled normal operator of $\nabla^{\bbE,0}$ at $\bhs{\phi\phi,2}$ can be read off from Proposition \ref{RescaledNormalOp}, and also that of $t \wt \bbA_{\fD}^2$.
Indeed, we get
\begin{equation*}
	N^G_{\phi\phi,2}( t \bbA_{\fD}^2 ) = \tau N_{\phi\phi,2}^G(x^2\bbA_{\fD}^2)
	=\tau \cH_{D^+/B} + \tau \bbB_{\pa M/D,+}^2
\end{equation*}s
while for a large perturbation
\begin{equation*}
\begin{split}
	N^G_{\phi\phi,2}\lrpar{t \wt \bbA_{\fD}^2}
	&= \tau \cH_{D^+/B} + \tau \bbB_{\pa M/D,\fD}^2(\eps,\tau) \\
	&
	+\tau\df e\lrpar{\frac{dx}x}
	\lrspar{\eth^V_{\fD} + \chi(\tau) A_P^\eps - m_0^{\wt\phi}\lrpar{\cB_{\wt\phi-hc}} 
	- \frac14 m_0^{\wt\phi}\lrpar{\Curv{\wt\phi}} } \\
	&=: \tau \cH_{D^+/B} + \tau \wt \bbB_{\pa M/D,+}^2(\eps,\tau),
\end{split}
\end{equation*}
where $\bbB_{\pa M/D, \fD}(\eps,\tau)$ is as in Proposition \ref{RescaledNormalOp} with $\eth^V_{\fD}$ replaced by $\eth^V_{\fD} + \chi(\tau) A_P^\eps$.
Notice that a small pertubation does not affect $N^G_{\phi\phi,2}(t\wt\bbA_{\fD}^2)$.

As pointed out in \cite{MelrosePiazza} the rescaling at $\bhs{00,2}$ behaves just like Getzler rescaling on a family of closed manifolds. Indeed, the discussion in \cite[Chapter 8]{APS Book} generalizes to families just as in $\S$\ref{sec:ResDouble} for horizontal vector fields. 
Thus, if $U \in \FD TM/B$, then
\begin{equation}\label{Act002}
	N^G_{00,2}\lrpar{t^{1/2}\nabla^{\cH}_{\beta_L^*U}} 
	= \sigma\lrpar U - \frac12 \df e \lrspar{R^{M/B}\lrpar{\nu_{00,2}\rest 0 
		+ \frac12\cR_{00,2},U}}
\end{equation}
where $\sigma\lrpar U$ is the symbol of $U$ interpreted as a constant coefficient vector field on the fibres of $\FD TM/B$, $\nu_{00,2}\rest 0$ is the restriction of $\nu_{00,2}$ to the zero section of the fibration $\bhs{00,2}\to \diag_M\times\{0\}$ and will be denoted by $Q_{00,2}$, 
$\cR_{00,2}$ is the radial vector field of this fibration, and the curvature acts as exterior product by a two-form. It follows that 
\begin{equation*}
	N^G_{00,2}\lrpar{t\wt\bbA_{\fD}^2} = \cH_{M/B} + \df e\lrpar{K_E'},
\end{equation*}
with the harmonic oscillator on the fibers of $\FC TM/B \to M$ defined as in Proposition \ref{RescaledNormalOp} with $R^{M/B}$ instead of $R^{D/B}$ and $K^{E'}$ the twisting curvature of $E$ (cf. \cite[Prop. 10.28]{BGV}).
Notice that neither a small nor a large perturbation affects the normal operator at $\bhs{00,2}$.

The normal operator at $\bhs{11,0}$ is potentially more complicated. However, as we mentioned in \S\ref{sec:HeatKerOne}, for the construction of the heat kernel we are only interested in the operators on this face acting on the bundle $\cK = \ker \eth^V_{\fD}$. In particular, if $\eth^V_{\fD}$ has no kernel (or has been perturbed to make it invertible), then we do not have a problem to solve at this face.
On the other hand, if $\dim \cK>0$ then it follows from the discussion around \eqref{HorizontalSupCon} that the normal operator at this face is
\begin{equation*}
	N_{11,0}(\bbA_{\fD})
	= \bbB_{\psi_D^+} + \cJ_{\wt\phi}
\end{equation*}
acting on $\Lambda^*B \otimes \cK$.
Finally, if $\wt\bbA_{\fD}$ is given by \eqref{CutOffB} then its restriction to $\cK$ at the boundary involves $\wt \bbB_{\psi_D^+}(t,\eps)$ obtained from $\bbB_{\psi_D^+}$ by replacing every instance of $\eth^{\cK}$ by $\eth^{\cK} + \chi(t) Q_R^\eps$.

Having described the normal operators at each of the relevant boundary faces, we next follow Vaillant and construct the heat kernel in the fibrewise rescaled heat calculus.
The heat kernels of the harmonic oscillators are given by Mehler's formula (as extended by Getzler), see \cite[Appendix 4]{Vaillant}. 

\begin{proposition}\label{prop:HeatKernels}
The heat kernel of $\wt\bbA_{\fD}^2$ is an element of the rescaled heat calculus
\begin{equation*}
	\exp (-t\wt\bbA_{\fD}^2) \in \Psi^{2,2,0}_{RHeat,\fD}\lrpar{M/B;E}.
\end{equation*}
The normal operators depend on the presence of a perturbation:
\begin{itemize}
\item [i)] 
In the invertible case, when $\cK$ forms a bundle and $\eth^{H,\cK}$ is invertible,
\begin{align*}
	&N^G_{00,2}\lrpar{e^{-t\bbA_{\fD}^2}}\lrpar \zeta 
	  = e^{-\cH_{M/B}\lrpar \zeta}e^{-\frac12 Q_{00,2}\lrpar \zeta}e^{-x^2K^{E'}}\\
	&N^G_{\phi\phi,2}\lrpar{e^{-t\bbA_{\fD}^2}}\lrpar{\hat \zeta, \tau}
	  = e^{-\tau\cH_{D^+/B}\lrpar{\hat\zeta}}
	  e^{-\frac12 Q_{\phi\phi,2}\lrpar{\hat\zeta}}
	  e^{-\tau \bbB_{\pa M/D,+}^2} \\
	&N^G_{11,0}\lrpar{e^{-t\bbA_{\fD}^2}}\lrpar t
	  =\cP_\cK e^{-t\lrpar{\bbB_{\psi_D}+\cJ_{\wt\phi}}^2}
	  \lrspar{ \frac1{\sqrt{4\pi t}} e^{-\frac{|\log s|^2}{4t} } }
\end{align*}
with $\zeta$ a linear variable in the fibres of $\FC TM/B$ and $\hat\zeta$ a linear variable in the fibers of $\FC N\pa M/B$.

\item [ii)]
If $P$ is a spectral section for $\eth^V_{\fD}$ and $\wt \bbA_{\fD}$ is given by \eqref{CutOffA}, 
\begin{align}
	&N^G_{00,2}\lrpar{e^{-t\wt \bbA_{\fD}^2}}\lrpar \zeta 
	= e^{-\cH_{M/B}\lrpar \zeta}e^{-\frac12 Q_{00,2}\lrpar \zeta}e^{-x^2K^{E'}}
	\notag \\
	&N^G_{\phi\phi,2}\lrpar{e^{-t\wt \bbA_{\fD}^2}}\lrpar{\hat \zeta, \tau}
	= e^{-\tau\cH_{D^+/B}\lrpar{\hat\zeta}}
	e^{-\frac12 Q_{\phi\phi,2}\lrpar{\hat\zeta}} 
	e^{-\tau \wt \bbB_{\pa M/D,+}^2(\tau,\eps)}
	\notag \\
	&N^G_{11,0}\lrpar{e^{-t\wt \bbA_{\fD}^2}}\lrpar t = 0. \label{N110}
\end{align}

\item [iii)]
If $R$ is a spectral section for $\eth^{H,\cK}_{\fD}$ and $\wt \bbA_{\fD}$ is given by \eqref{CutOffB}, 
\begin{align*}
	&N^G_{00,2}\lrpar{e^{-t \wt \bbA_{\fD}^2}}\lrpar \zeta 
	  = e^{-\cH_{M/B}\lrpar \zeta}e^{-\frac12 Q_{00,2}\lrpar \zeta}e^{-x^2K^{E'}}\\
	&N^G_{\phi\phi,2}\lrpar{e^{-t \wt \bbA_{\fD}^2}}\lrpar{\hat \zeta, \tau}
	  = e^{-\tau\cH_{D^+/B}\lrpar{\hat\zeta}}
	  e^{-\frac12 Q_{\phi\phi,2}\lrpar{\hat\zeta}}
	  e^{-\tau \bbB_{\pa M/D,+}^2} \\
	&N^G_{11,0}\lrpar{e^{-t \wt \bbA_{\fD}^2}}\lrpar t
	  =\cP_\cK e^{-t\lrpar{\cl{\tfrac{dx}x}x\pa_x + \wt\bbB_{\psi_D}(t,\eps)+\cJ_{\wt\phi}}^2}
\end{align*}

\end{itemize}
\end{proposition}

{\bf Remark.} The vanishing of \eqref{N110} means that $e^{-t\wt \bbA_{M/B,P}}$ is actually an element of $\Psi^{2,2,1}_{RHeat,\fD,\psi}\lrpar{M;E}$, but we are interested in its normal operators as an element of $\Psi^{2,2,0}_{RHeat,\fD,\psi}\lrpar{M;E}$.

\begin{proof}
Notice that the normal operators in every case match at the corner $\bhs{00,2}\cap \bhs{\phi\phi,2}$ by compatibility of the rescalings and solve the heat equation at these faces.
Demanding that compatibility at the corner $\bhs{\phi\phi,2}\cap \bhs{11,0}$ yields a boundary value for the heat equation for the face $\bhs{11,0}$. If $\eth^V_{\fD}$ is invertible (or has been perturbed to be invertible), then the boundary value is zero, and the unique solution to the heat equation vanishing at $t=0$ is the zero kernel.
If, on the other hand, $\eth^V_{\fD}$ is not invertible then the compatibility condition is that the kernel at time $t=0$ should be the projection onto the kernel of $\eth^V_{\fD}$.
Thus the relevant problem at $\bhs{11,0}$ is the heat equation for the indicial operator $I_b(\wt\bbA_{\fD}) = \cl{\frac{dx}x}x\pa_x + \wt\bbB_{\psi_D}+\cJ_{\wt\phi}$ acting on the bundle $\cK$ and this is solved by the normal operator prescribed above.
Thus, in every cases, the alleged normal operators are compatible and so there is an element of the rescaled heat calculus with these normal operators. Any such element is a parametrix which solves the heat equation to first order in the rescaled heat calculus (which is all we need for the index formulas below).

The composition result in \cite[Chapter 4]{Vaillant} can be applied fibrewise to get a composition result for the fibrewise heat calculus. In fact, as pointed out in \cite{MelrosePiazza}, the construction of the triple space used to prove the composition result can itself be carried out fibrewise and show that not only does the same composition result hold, but it has smooth dependence on the base of the fibration.

The composition in the heat calculus can be used, together with the solution to the normal operators, to construct a parametrix that solves the heat equation to infinite order at each boundary face. At that point all that is left is a Volterra-type operator which as shown in \cite[Thm. 7.24]{APS Book} has an inverse of the same type.
\end{proof}

{\bf Remark.} In an early version of \cite{MelrosePiazza} the perturbations were not multiplied by cut-off functions to keep them away from $t=0$ and so an `extended' heat calculus composition result was used and established in an appendix. The reason being that the composition resulted in elements with non-trivial asymptotics as $t\to0$ away from the diagonal, which are not allowed in the `standard' heat calculus.
Later the cut-off was added and, though the appendix remains, the composition result is not needed as the perturbations vanish to infinite order at $t=0$. The same is true in our case, no non-trivial asymptotics arise at any face other than $\bhs{00,2}$, $\bhs{\phi\phi,2}$, and $\bhs{11,0}$ hence we do not need to extend Vaillant's heat calculus.

\subsection{The Chern character} $ $ \newline
The Chern character of a superconnection on a closed manifold is defined to be the supertrace of its heat kernel. 
A key part of extending the index theorem to non-compact manifolds is that the heat kernel is generally not of trace class (cf. \S\ref{sec:VaillantIndex}).
This is true for $\fD$ superconnections and we will have to deal with renormalization. 
Fortunately, the behavior of the heat kernel for positive time is just like that treated in \cite{MelrosePiazza}. What is more, if the heat kernel vanishes at $\bhs{11,0}$ then it actually is of trace class and behaves much like the heat kernel on a closed manifold. 

Recall the definition of the renormalized trace from \eqref{DefRTr}.
As this involves restricting the integral kernel of an operator to the diagonal, it is important to note - from the definition of the density bundles - that restricting an element of the heat calculus to the diagonal and taking the pointwise trace yields a function on
\begin{equation*}
	\diag_H = \lrspar{ M \times \bbR^+_{\sqrt t}; \pa M \times \lrbrac 0}
\end{equation*}
given by \cite[Lemma 5.26]{Vaillant}
\begin{multline*}
	\tr \rest{\diag}: \Psi^{2,2,0}_{RHeat,\fD}\lrpar{M/B,E} \\
	\to \rho_{11,0}^{-1}\rho_{\phi\phi,2}^{-h-2}\rho_{00,2}^{-n}
	\CI\lrpar{\diag_H,\beta_{\diag_H,R}^*\Omega\lrpar X \; dt} 
\end{multline*}
\begin{equation*}
	\str \rest{\diag}: \Psi^{2,2,0}_{RHeat,\fD}\lrpar{M/B,E}
	\to \rho_{11,0}^{-1}
	\CI\lrpar{\diag_H,\beta_{\diag_H,R}^*\Omega\lrpar X \; dt}
\end{equation*}
where $\Omega\lrpar X$ denotes the density bundle on $X$, and the extra cancellation for the supertrace is the result of the Getzler rescalings at $\bhs{00,2}$ and $\bhs{\phi\phi,2}$.

Thus, we see that in the constant rank case, $e^{-t\bbA_{\fD}^2}$ is not of trace class but is a $b$-density for $t>0$. This means that we will be able to make use of the $b$-trace of \cite{APS Book} to define a renormalized Chern character and especially that the results of Melrose and Piazza \cite{MelrosePiazza} describe this renormalized Chern character very explicitly. In the presence of a large perturbation we have an extra degree of vanishing from \eqref{N110}, so $e^{-t\bbA_{\fD,P}^2}$ actually is of trace class for $t>0$. Hence we can use the traditional definition of the Chern character, which is then described very explicitly by Berline, Getzler, and Vergne \cite{BGV}.

As already remarked, the renormalized trace is not an actual trace since it does not vanish on commutators. However, it is always true that the renormalized trace of a commutator localizes to the boundary. For $b$-densities this is particularly elegant as shown by Melrose \cite[Chapter 4]{APS Book} (see also \cite[Proposition 9]{MelrosePiazza} for families) and is given by the trace-defect formula
\begin{equation*}
	\Ren{Tr}{\lrspar{A,B}}
	= \frac i{2\pi}\int_\bbR 
	\Tr\lrpar{ \pa_\lambda I\lrpar{A,\lambda} \cdot I\lrpar{B,\lambda} } \; d\lambda,
\end{equation*}
whenever at least one of $A$ and $B$ is smoothing.
Recall that, given a $b$-operator $A$, the indicial family of $A$ is, for each complex number $\lambda$, an operator on the boundary $\pa X$ defined by
\begin{equation*}
	I\lrpar{A,\lambda}\lrpar{\xi} = \lrpar{x^{-i\lambda}Ax^{i\lambda}\wt\xi}\rest{\pa X}
\end{equation*}
where $\wt\xi$ is any extension of $\xi$ from the boundary (and the result is independent of the extension used).

As the heat kernels of $\fD$ Bismut superconnections in the constant rank case are $b$-densities for $t>0$, so generally not integrable, we define
\begin{equation*}
	\Ren{\Ch}{\bbA_{\fD}} = \Ren{Str}{e^{-\bbA_{\fD}^2}}.
\end{equation*}

The heat equation proof of the index theorem proceeds by comparing the behavior at $t\to \infty$ and $t\to0$ of the rescaled Bismut superconnection, whose definition we now recall.
Given any superconnection $\bbA$, say it acts on $\CI(X;E)$ by 
\begin{equation*}
	\bbA = \bbA_{[0]} + \bbA_{[1]} + \ldots + \bbA_{[k]}
\end{equation*}
where each $\bbA_{[j]}$ maps into $j$-forms, we define the {\bf rescaled superconnection} $\bbA^t$ to be
\begin{equation*}
	\bbA^t := t^{1/2}\lrpar{\bbA_{[0]} + t^{-1/2}\bbA_{[1]} + \ldots + t^{-k/2}\bbA_{[k]} }.
\end{equation*}
This can be expressed in terms of the function $\delta_t^B$ that multiplies forms on $B$ of degree $k$ by $t^{-k/2}$ as
\begin{equation}\label{RescaledA}
	\bbA^t = t^{1/2}\delta_t^B \bbA \lrpar{\delta_t^B}^{-1}.
\end{equation}

In the following lemma we compute the derivative in $t$ of the renormalized Chern character. When $\eth^V_{\fD}$ has constant rank kernel, $\cK$, but is not invertible there is a boundary contribution in the form of a supertrace of an operator on $\cK$. We define the supertrace on $K$ by
\begin{equation*}
	\str_{\cK}(A) = \int_{\pa M/D} \tr_E \lrspar{ Q_{M/B} \cl{\tfrac{dx}x} A }
\end{equation*}
where $Q_{M/B}$ is the grading operator on $\Cl(M/B)$ (=$i^{\lfloor m/2 \rfloor} \cl{\dvol_{M/B}}$).

\begin{lemma}[\cite{BGV}, Theorem 9.17; \cite{MelrosePiazza}, Proposition 11, Corollary 3] \label{lem:dt} $ $\newline
i)
If $\eth^V_{\fD}$ is invertible or is perturbed to be invertible, then the Chern character of the $\fD$ Bismut superconnection $\wt \bbA_{\fD}$ satisfies
\begin{equation*}
\frac{\pa}{\pa t} \Ch\lrpar{ \wt \bbA_{\fD}^t } 
= -d_B\Str\lrpar{ \frac{\pa \wt \bbA_{\fD}^t}{\pa t} e^{-(\wt \bbA_{\fD}^t)^2} }.
\end{equation*}

Otherwise, the renormalized Chern character of $\wt \bbA_{\fD}$ satisfies 
\begin{equation*}
	\frac{\pa}{\pa t} \lrpar{\Ren{Ch}{ \wt \bbA_{\fD}^t } }
	= 
	-d_B\; \Ren{Str}{ \frac{\pa \wt \bbA_{\fD}^t}{\pa t} e^{-(\wt \bbA_{\fD}^t)^2 } }
	+\hat\eta\lrpar t
\end{equation*}
where $\hat\eta\lrpar t \in \Lambda^\ev B$ is defined to be
\begin{equation*}
	\hat\eta\lrpar t 
	= \sqrt t
	\int_{D/B} \str_K\lrpar{ N^G_{11,0}\lrspar{ \frac{\pa \wt \bbA_{\fD}^t}{\pa t} e^{-(\wt \bbA_{\fD}^t)^2 } } }.
\end{equation*} 

ii) If $\eth^{H,\cK}_{\fD}$ is invertible then
\begin{align*}
	-\hat\eta\lrpar{ \eth^{H,\cK}_{\fD} } &:= \int_0^\infty \hat\eta(t) \; dt \\
	&= \frac1{2\sqrt\pi} \int_0^\infty \int_{D/B} 
	\str_{\cK}\lrpar{ \frac{\pa}{\pa t} (\bbB_{\psi_D}^t + \cJ_{\wt\phi}^t)
	e^{-(\bbB_{\psi_D}^t + \cJ_{\wt\phi}^t)^2}  } \;dt,
\end{align*}
and if $\dim D/B$ is even and $\cJ_{\wt\phi}=0$ then $\hat\eta(t)$ vanishes.
If $Q$ is a spectral section for $\eth^{H,\cK}_{\fD}$ then
\begin{equation*}
	-\hat\eta_Q\lrpar{ \eth^{H,\cK}_{\fD} } := \int_0^\infty \hat\eta(t) \;dt
\end{equation*}
is given, up to an exact form, by
\begin{equation*}
	\frac1{2\sqrt\pi} \int_0^\infty \int_{D/B} 
	\str_{\cK}\lrpar{ \frac{\pa}{\pa t} (\hat\bbB_{\psi_D}^t + \cJ_{\wt\phi}^t)
	e^{-(\hat \bbB_{\psi_D}^t + \cJ_{\wt\phi}^t)^2}  } \;dt,
\end{equation*}
where $\hat\bbB_{\psi_D}$ is equal to $\wt\bbB_{\psi_D}$ with $\eps=0$.
\end{lemma} 

\begin{proof}
For the duration of the proof we lighten the notation by introducing
\begin{equation*}
	\bbA_t := \wt\bbA_{\fD}^t, \quad
	\dot \bbA_t :=  \frac{\pa \wt\bbA_{\fD}^t}{\pa t}.
\end{equation*}

First one can justify as in \cite{MelrosePiazza} using Duhamel's principle and the trace-defect formula that
\begin{equation*}
	\frac{\pa}{\pa t} \lrpar{\Ren{Ch}{ \wt \bbA_{\fD}^t } }
	= - \Ren{Str}{ \frac{\pa (\wt\bbA_{\fD}^t)^2}{\pa t} e^{-(\wt \bbA_{\fD}^t)^2} }.
\end{equation*}
Next since $\wt\bbA_{\fD}$ is odd with respect to the filtration on $\Cl_0(M/B)$, we can identify the derivative with a supercommutator
\begin{equation*}
\begin{split}
	{}^R\Str \Bigl( &\frac{\pa (\wt\bbA_{\fD}^t)^2}{\pa t} e^{-(\wt \bbA_{\fD}^t)^2}  \Bigr)
	=\Ren{Str}{ \lrspar{ \bbA_t,  \dot\bbA_t e^{-\bbA_t^2} } } \\
	&=
	\FP_{z=0} \Str
	\lrpar{ x^z \lrspar{ \bbA_t,  \dot \bbA_t e^{- \bbA_t^2 } } } \\
	&=
	\FP_{z=0} \Str \lrpar{  
	\lrspar{ \bbA_t, x^z \dot \bbA_t e^{-\bbA_t^2} } 
	- \lrspar{ \bbA_t, x^z } \dot \bbA_t e^{-\bbA_t^2} }    \\
	&=
	\FP_{z=0} \Str \lrpar{  
	\lrspar{ \bbA_{[1]}, x^z \dot \bbA_t e^{-\bbA_t^2} } 
	- zx^z t^{1/2} m_0(\tfrac{dx}x) \dot \bbA_t e^{-\bbA_t^2} }    \\
	&=
	d_B {}^R \Str
	\lrpar{ \dot \bbA_t e^{-\bbA_t^2}  }
	- \Res_{z=0} \; \Str\lrpar{x^z \sqrt t \cl{\tfrac{dx}x}  \dot \bbA_t e^{-\bbA_t^2} } \\
	&=
	d_B {}^R \Str
	\lrpar{ \dot \bbA_t e^{-\bbA_t^2}  }
	-\int_{\bhs{11,0}\cap \diag} \str
	\lrpar{ \cl{\tfrac{dx}x} \sqrt t \dot \bbA_t e^{-\bbA_t^2} \rest{\bhs{11,0}} } \\
	&=
	d_B {}^R \Str
	\lrpar{ \dot \bbA_t e^{-\bbA_t^2}  }
	- \hat\eta(t)
\end{split}
\end{equation*}
which proves ($i$). 

The convergence of the integral of $\hat\eta(t)$ follows from the existence of the limits as $t \to \infty$ and as $t\to 0$ of ${}^R\Ch( \wt \bbA_{\fD}^t )$ which will be discussed in Lemmas \ref{lem:Bigt} and \ref{lem:Smallt}.

If $\eth^{H,\cK}_{\fD}$ is invertible, then 
\begin{multline*}
	\hat\eta(t)
	=\int_{\bhs{11,0}\cap \diag} \str
	\lrpar{ \cl{\tfrac{dx}x} \sqrt t
	\frac{\pa}{\pa t} (\bbB_{\psi_D}^t + \cJ_{\wt\phi}^t)
	e^{-(\bbB_{\psi_D}^t + \cJ_{\wt\phi}^t)^2} \frac1{\sqrt{4\pi t}} e^{ -\frac{|\log s|^2}{4t} } }\\
	= 
	\frac1{ 2\sqrt{\pi} } \int_{\bhs{11,0}\cap \diag} 
	\str\lrpar{ \cl{\tfrac{dx}x} 
	\frac{\pa}{\pa t} (\bbB_{\psi_D}^t + \cJ_{\wt\phi}^t)
	e^{-(\bbB_{\psi_D}^t + \cJ_{\wt\phi}^t)^2}  }
\end{multline*}
as required.

If $\dim D/B$ is even and $\cJ_{\wt\phi}=0$ then let $Q_{D^{+}/B}$ be the involution \newline $i^{\lfloor\frac{h+1}{2}\rfloor}\cl{\dvol_{D^{+}/B}}$ and note that
\begin{gather*}
	\wt \bbB_{\psi_D}^t Q_{D^{+}/B} = Q_{D^{+}/B} \wt \bbB_{\psi_D}^t,
	\quad
	\cl{\tfrac{dx}x} Q_{D^{+}/B} = Q_{D^{+}/B} \cl{\tfrac{dx}x}, \\
	Q_{M/B} Q_{D^{+}/B} = -Q_{D^{+}/B} Q_{M/B},
	\quad 
	Q_{D^{+}/B}^2 = 1.
\end{gather*}
Hence,
\begin{equation*}
\begin{split}
	\hat\eta(t)
	&=  \frac{1}{2\sqrt{\pi}}\int_{D/B} \tr_E \lrspar{ Q_{D^{+}/B}^2 Q_{M/B} \cl{\tfrac{dx}x} \frac{\pa \bbB_{\psi_D}^t}{\pa t}
	 e^{-(\bbB_{\psi_D}^t)^2}  }
	\\&= -\frac{1}{2\sqrt{\pi}}\int_{D/B} \tr_E \lrspar{ Q_{D^{+}/B} Q_{M/B} \cl{\tfrac{dx}x} \frac{\pa \bbB_{\psi_D}^t}{\pa t}
	 e^{-(\bbB_{\psi_D}^t)^2} Q_{D^{+}/B} }
	\\&= -\frac{1}{2\sqrt{\pi}}\int_{D/B} \tr_E \lrspar{ Q_{D^{+}/B}^2 Q_{M/B} \cl{\tfrac{dx}x} \frac{\pa \bbB_{\psi_D}^t}{\pa t}
	 e^{-(\bbB_{\psi_D}^t)^2} }
	= - \hat\eta(t)
\end{split}
\end{equation*}
and so $\hat\eta(t)$ vanishes.

Finally if $\eth^{H,\cK}_{\fD}$ is being perturbed to be invertible, we appeal to the computations in \cite[Corollary 3]{MelrosePiazza} using the trace-defect formula to show that one can replace $\wt\bbB_{\psi_D}$ with $\hat\bbB_{\psi_D}$ at the cost of an exact form. The difference between our situation and theirs is the presence of the endomorphism $\cJ_{\wt\phi}$ measuring the failure of the $\fD$ metric and the Clifford bundle $E$ to be of product-type. However, one can check that this term does not affect the computations in the proof of \cite[Corollary 3]{MelrosePiazza}.
\end{proof}

Note that, even when $\dim D/B$ is even, $\hat\eta(t)$ need not vanish if $\cJ_{\wt\phi}$ is not zero.
The above argument fails because, even when the metric is of product-type at the boundary,
$\cJ_{\wt\phi} = \cP_\cK \v\phi\cl{K_E'\lrpar{dx^\flat,\cdot}} $ anti-commutes with $Q_{D^{+}/B}$ while $\wt \bbB_{\psi_D}^t$ commutes. 
More to the point, we can explicitly compute the contribution in a greatly simplified context.
Assume that $\dim D/B=0$ and $\dim M/B=2$, that $\cK$ forms a vector bundle over $B$, and the metric $g_{\fD}$ is of product-type so that
\begin{equation*}
	\bbB_{\psi_D}^t + \cJ_{\wt\phi}^t = 
	\cP_{\cK}\lrspar{
	\sqrt t K_E'(x\pa_x, \tfrac1x \pa_z) \cl{\tfrac1x\pa_z} 
	+ \df e^\alpha \nabla^E_\alpha } 
\end{equation*}
where $\alpha$ ranges over a frame for $\psi^*TB$; furthermore assume that the twisting curvature $K_E'(x\pa_x, \frac1x \pa_z)\rest{\pa M}=i C$ where $C$
is a non-zero real constant and that $\nabla^E \frac1x\pa_z$ vanishes at the boundary.
(All of these assumptions hold for a natural family of  $\bar\pa$ operators on Riemann surfaces with cusps as will be explained in a companion paper \cite{AlbinRochon}.)
Then, in this case, we have
\begin{equation}\label{h=0EtaHat}
\begin{aligned}
	\hat\eta \lrpar{ \eth^{H,\cK}_{\fD} } 
	&= -\frac{1}{2}\int_0^{\infty} \int_{\pa M/B}
	\tr_E \lrpar{ Q_{M/B} \cl{ \tfrac {dx}x } \frac{i}{2\sqrt{\pi t}} C \cl{\tfrac1x \pa_z} e^{-\bbB_{[1]}^2} e^{-tC^2} } \\
	&=  \frac{1}{4}\int_{\pa M/B}
	\tr_E \lrpar{ e^{-\bbB_{[1]}^2}}
	\int_0^{\infty} C e^{-tC^2} \frac{dt}{\sqrt{\pi t}} \\ 
	&= \frac{1}{4} \mathrm{sign}(C) \Ch(\ker \eth^V_{\fD} )=\frac{1}{4} \mathrm{sign}(C) \Ch(\ker \eth^{V,+}_{\fD}\oplus\ker \eth^{V,+}_{\fD} ) \\
	&= \frac{1}{2} \mathrm{sign}(C) \Ch(\ker \eth^{V,+}_{\fD} ),
\end{aligned}
\end{equation}
where in the line before the last, we have used the fact that, under the $\bbZ_{2}$ grading,
$\ker \eth^{V}_{\fD}$ decomposes as
\begin{equation}
     \ker\eth^{V}_{\fD}= \ker \eth^{V,+}_{\fD}\oplus \ker\eth^{V,-}_{\fD}
\label{dec}\end{equation}
with $\ker \eth^{V,-}_{\fD}$ canonically identified with $\ker\eth^{V,+}_{\fD}$ via the 
Clifford action of $\cl{ \tfrac {dx}x }$.
Note that
\begin{equation}\label{h=0Eta}
	\lrspar{ \hat\eta \lrpar{ \eth^{H,\cK}_{\fD} } }_{[0]} 
	= \frac12\mathrm{sign}(C) \dim (\ker \eth^{V,+}_{\fD} )
\end{equation}
holds without the assumption that $\nabla^E \frac1x\pa_z$ vanishes at the boundary.

Back to the general case, we next assume that $\ker \eth_{\fD}$ and $\coker \eth_{\fD}$ are vector bundles over $B$, potentially after a smoothing perturbation. If $B$ is compact and $\eth_{\fD}$ has been perturbed to be Fredholm then, as pointed out at the end of \S\ref{sec:Perturbations}, this can always be arranged. However if $B$ is not compact, then we must make this assumption to start with.
We denote the resulting $\bbZ/2$ graded bundle by $\Ind(\eth_{\fD})$ in the unperturbed case and by $\Ind_P(\eth_{\fD})$ when $\eth_{\fD}$ is being perturbed by an operator associated to the spectral section $P$. We will denote it by $\wt \Ind (\eth_{\fD})$ when we do not want to specify.
The orthogonal projection $\cP_{\Ind}$ onto this bundle is a smoothing compact operator which we use to contract the Levi-Civita connection to form
\begin{equation*}
	\nabla^{\Ind} := \cP_{\Ind} \lrpar{ \wt \bbA_{\fD}}_{[1]} \cP_{\Ind}.
\end{equation*}

The following lemma can be proven just as in \cite[Proposition 15]{MelrosePiazza} by applying the argument of Berline and Vergne \cite{BV}.

\begin{lemma}[\cite{BGV}, Theorem 9.23; \cite{MelrosePiazza}, Proposition 15] \label{lem:Bigt}
If $\eth_{\fD}$ is Fredholm or has been perturbed to be Fredholm, the Chern character (renormalized if necessary) of the $\fD$ superconnection $\wt \bbA_{\fD}$ satisfies
\begin{equation*}
\lim_{t\to\infty} \Ren{Ch}{\wt \bbA_{\fD}^t}
= \Ch\lrpar{ \wt \Ind\lrpar{\eth_{\fD}}, \nabla^{\Ind} }.
\end{equation*}
\end{lemma}

The last piece we need is the behavior of the supertrace of the heat kernel as $t \to 0$. To compute this we need to keep track of the three rescalings, one from replacing $\bbA_{\fD}$ with $\bbA_{\fD}^t$, and two from changing the homomorphism bundle at $\bhs{\phi\phi,2}$ and $\bhs{00,2}$.

Before stating the result we recall the definition of the twisted supertrace and the Bismut-Cheeger $\hat \eta$ form.
The former is denoted $\str'$ and defined in \cite{APS Book} as follows. Denote the grading operator on $E$ by $Q_E$ so that $\str_E\lrpar\cdot = \tr_E\lrpar{Q_E\cdot}$. Because $E$ is a $\Cl\lrpar{M/B}$ module we can write $Q_E = Q_{\Cl\lrpar{M/B}} \otimes Q_E'$ where $Q_{\Cl\lrpar{M/B}}$ is the grading operator on $\Cl\lrpar{M/B}$ and $Q_E'$ supercommutes with Clifford multiplication, then if $A\otimes B \in \Cl\lrpar{M/B} \otimes \hom'_{M/B}\lrpar E$
\begin{equation}
	\str_E\lrpar{A\otimes B} = \tr\lrpar{Q_{\Cl\lrpar{M/B}}A}\tr\lrpar{Q_E'B}
	=:\str_{\Cl\lrpar{M/B}}\lrpar A \str'\lrpar B.
\label{mar26.1}\end{equation} 
The twisted Chern character is defined as $\str'(e^{-t(K_E')^2})$ where $K_E'$ is the twisting curvature (the part of the curvature that anti-commutes with Clifford multiplication).

When $\left. E\right|_{\pa M}$ is seen as a $\Cl(D^{+}/B)$ Clifford module, there is also a notion of relative supertrace.  Namely, if 
$Q_{\Cl(\pa M/D)}$ is the vertical grading operator, then we define the
relative supertrace to be
\begin{equation}
   \str_{E/D^{+}}(L)= \tr (Q_{\Cl(\pa M/D)} L), \quad L\in 
\hom_{\Cl(D^{+}/B)}(E).
\label{relsupt}\end{equation}
Instead of \eqref{mar26.1}, we have in this case that 
\begin{equation}
   \str_{E}(A \hat{\otimes} L)= \left\{    
                \begin{array}{ll}
                   \str_{D^{+}/B}(A) \str_{E/D^{+}}(L), &  h \quad \mbox{odd}, \\
i \str_{D^{+}/B}(A) \str_{E/D^{+}}(L), &  h \quad \mbox{even},
\end{array}
\right.
\label{mar26.2}\end{equation}
for 
$A\hat{\otimes} L\in \Cl(D^{+}/B)\hat{\otimes} \hom_{\Cl(D^{+}/B)}(E)$.
The extra $i$ factor in the case $h$ is even comes from the fact (\cf p.105 in 
\cite{Vaillant})
\begin{equation}
  Q_{\Cl(M/B)}= \left\{
              \begin{array}{ll}
                  Q_{\Cl(D^{+}/B)}Q_{\Cl(\pa M/D)}, & h \quad\mbox{odd}, \\
-iQ_{\Cl(D^{+}/B)}Q_{\Cl(\pa M/D)}, & h \quad \mbox{even}, \\
\end{array}
\right.
\label{mar26.3}\end{equation}
together with the fact $Q_{\Cl(\pa M/D)}$ anti-commutes with $A$ whenever
$h$ is even and $\str_{E}(A\hat{\otimes}L)$ is non-zero.  Notice that our
definition \eqref{relsupt} of the relative supertrace differs from the one of
Vaillant when $h$ is even.  

If $Z' \to X' \xrightarrow{\phi'} Y'$ is a fibration, and $\bbA'_t$ is the rescaled Bismut superconnection associated to a submersion metric, then the Bismut-Cheeger $\hat\eta$ form is defined (up to powers of $2$ and $\pi$) to be
\begin{equation*}
	\int_0^\infty
	\Str\lrspar{ 
	\lrpar{\eth - \frac1{4t} m_0^{\phi'}\lrpar{\Omega_{\phi'}} }
	e^{-(\bbA'_t)^2} } \frac{dt}{2\sqrt t}
\end{equation*}
if the dimension of the fibers is even and 
\begin{equation*}
	\frac1{\sqrt\pi} 
	\int_0^\infty
	\Tr^\ev\lrspar{ 
	\lrpar{\eth - \frac1{4t} m_0^{\phi'}\lrpar{\Omega_{\phi'}} }
	e^{-(\bbA'_t)^2} } \frac{dt}{2\sqrt t}
\end{equation*}
if the dimension of the fibers is odd. As a form on $Y'$, $\hat \eta$ is even (respectively odd) when the dimension of the fiber is odd (respectively even).

Below we will denote by $\hat\eta^+$ forms that differ from the Bismut-Cheeger $\hat\eta$ in the same way that $\bbB_{\pa M/D, \fD}$ differs from the Bismut superconnection of $\wt\phi$, i.e., in the inclusion of the tensor $\cB_{\wt\phi-hc}$ which measures the failure of $g_{\fD}$ to be of product-type at the boundary. 
Explicitly, if $D/B$ is even, $\hat \eta^+(\eth^V_{\fD})$ is given by
\begin{equation*}
	\int_0^{\infty}
	\Str\lrspar{
	\lrpar{\eth^V_{\fD} - \frac1{\sqrt\tau} m_0^{\wt\phi}\lrpar{B_{\wt\phi}} 
	- \frac1{4\tau} m_0^{\wt\phi}\lrpar{\Omega_{\wt\phi}} }
		e^{-(\wt\bbB_{\pa M/D,\fD}^\tau)^2} }
	\frac{d\tau}{\sqrt\tau}
\end{equation*}
and {\em mutatis mutandis} when $D/B$ is odd.
In particular, if there is an extension of $\wt\phi$ to a collar neighborhood of the boundary for which $g_{\fD}$ is a submersion to second order, then $\hat\eta^+$ coincides with $\hat\eta$. We will also use the Melrose-Piazza invariants $\hat\eta_P^+$ defined just like $\hat\eta^+$ but with $\eth^V_{\fD}$  replaced by the addition of a large perturbation associated to $P$, $\eth^V_{\fD} + A_P^\eps$. Just as in \cite{MelrosePiazza}, different choices of $A_P^\eps$ corresponding to the same spectral section $P$ only change $\hat\eta_P^+$ by an exact form.

\begin{lemma}\label{lem:Smallt}
The Chern character of the rescaled Bismut superconnection, $\wt \bbA_{\fD}^t$, has a limit as $t\to 0$ which depending on the perturbation is given by:
\begin{itemize}
\item 
If $\eth^V_{\fD}$ has kernel of constant rank and $\eth^{H,\cK}_{\fD}$ is invertible or has been perturbed to be invertible then
\begin{multline}\label{Time0Unperturbed}
\lim_{t\to0} \Ren{Ch}{\wt \bbA_{\fD}^t} 
= \frac1{\lrpar{2\pi i}^{n/2}} \int_{M/B} \hat A\lrpar{M/B}\Ch'\lrpar E \\
+ \frac{1}{\lrpar{2\pi i}^{\lfloor\frac{h+1}{2}\rfloor}} \int_{D/B} \hat A\lrpar{D/B}\hat\eta^+\lrpar{\eth^V_{\fD},E}
\end{multline}

\item If $\eth^V_{\fD}$ is invertible or has been perturbed to be invertible by a large perturbation associated to $P$, then 
\begin{multline}\label{Time0Perturbed}
\lim_{t\to0} \Ch\lrpar{ \wt \bbA_{\fD}^t } 
= \frac1{\lrpar{2\pi i}^{n/2}} \int_{M/B} \hat A\lrpar{M/B}\Ch'\lrpar E \\
+ \frac1{\lrpar{2\pi i}^{\lfloor\frac{h+1}{2}\rfloor}} \int_{D/B} \hat A\lrpar{D/B}\hat\eta^+_P\lrpar{\eth^V_{\fD},E}.
\end{multline}
The limit (in cohomology)
only depends on the choice of $P$ and not on the choice of associated perturbation.
\end{itemize}
\end{lemma}

\begin{proof}
The small-time limit of the renormalized supertrace of an element 
$Q \in \Psi^{2,2,0}_{RHeat,\psi}\lrpar{M;E}$ can be computed using the push-forward theorem from \cite{Melrose:Corners}  applied to $\diag_H \to \bbR^+_t$, or directly as in \cite[Lemma 5.27]{Vaillant}, to be
\begin{equation*}
	\lim_{t\to0} \Ren{Str}{Q} = 
	\int_{\bhs{00,2}\cap \diag_H} \str\lrpar Q\rest{\bhs{00,2}}
	+ \int_{\bhs{\phi\phi,2}\cap \diag_H} \str\lrpar Q\rest{\bhs{\phi,\phi,2}}.
\end{equation*}
Convergence of the two integrals follows from the vanishing of the integrands at the corners $\bhs{11,0}\cap \bhs{\phi\phi,2}$ and $\bhs{\phi\phi,2}\cap \bhs{00,2}$ which in turn follows from the remark immediately after the statement of Proposition \ref{RescaledNormalOp}.

We make use of two well-known facts to reduce this computation to the rescaled normal operators from Proposition \ref{prop:HeatKernels}. The first is that \cite[Lemma 10.22]{BGV}
\begin{equation*}
	\str(e^{-(\wt \bbA_{\fD}^t)^2}) = \delta_t^B \lrpar{ \str( e^{-t\wt \bbA_{\fD}^2} ) }
\end{equation*}
while the second is the fundamental observation of Patodi that the supertrace in $\Cl(V)$ only depends on the terms with  Clifford degree equal to $\dim V$.

Consider first the contribution from $\bhs{00,2}$ (which is the same for the perturbed and unperturbed superconnections).
We know that the heat kernel has an expansion at this face of the form
\begin{equation*}
	k \sim t^{-n/2} \sum_{\ell \geq 0} t^{\ell/2} k_\ell
\end{equation*}
and since the heat kernel is in the rescaled calculus each $k_\ell$ has total degree at most $\ell$ with respect to the filtration by $\Cl_0(M) = \Cl(M/B) \otimes \Lambda B$.
By the two well-known facts above,
\begin{equation*}
	\delta_B^t \str(k) \sim \sum_{\ell \geq 0} t^{(\ell-\bfN_B)/2} \str(\wt k_\ell)
\end{equation*}
where each $k_\ell$ has been replace by the part with total degree $n$ with respect to the filtration by $\Cl(M/B)$. 
The fact that the total degree of $k_{\ell}$ with respect to $\Cl_0(M)$ is at most $\ell$ shows that this expansion is non-singular at $t=0$ and has limit equal to the supertrace of the part of $k$ with maximal degree with respect to this filtration. This is precisely the supertrace of a volume form ($(-2i)^{n/2}$) times the $n$-form part of the rescaled normal operator,
\begin{equation*}
	(-2i)^{n/2} \mathrm{Ev}_n^{M/B}
	\lrspar{ N^G_{00,2}\lrpar{ e^{-t\wt \bbA_{\fD}^2}} }.
\end{equation*}
So using the formula for $N^G_{00,2}\lrpar{\bbA}$ from Proposition \ref{prop:HeatKernels} and the fact that
$e^{-\cH_{M/B}\lrpar \zeta}e^{-\frac12 Q_{00,2}\lrpar \zeta}\rest{\zeta=0} = (4\pi)^{-n/2}\hat A\lrpar{M/B}$ we get
\begin{multline*}
	\int_{\bhs{00,2}\cap \diag_H} \delta_\tau
		\str\lrpar{e^{-t\bbA_{\fD}^2}}\rest{\bhs{00,2}} \\
	=\int_{\bhs{00,2}\cap \diag_H} \lrpar{-2i}^{n/2}
	 \mathrm{Ev}_n^{M/B}\lrspar{N^G_{00,2}\lrpar{e^{-t\bbA_{\fD}^2}}\rest{\diag_H}} \\
	= \frac1{\lrpar{2\pi i}^{n/2}}\int_{\bhs{00,2}\cap \diag_H}  
	\mathrm{Ev}_n^{M/B}\lrspar{ \hat A\lrpar{M/B} \str'\lrpar{e^{-t\lrpar{K_E'}^2}}}.
\end{multline*}

A similar analysis applies to $\bhs{\phi\phi,2}$.
Notice that at this face the rescaling by $t$, $\delta_B^t$, becomes $\delta^{x^2}_B \delta_B^\tau$.
The effect of $\delta^{x^2}_B$ at this face is just like the effect of $\delta_B^t$ at $\bhs{00,2}$: it compensates 
for the fact that the rescaling is with respect to $\Cl_0(D^+)$ and the supertrace picks out the term of highest degree with respect to $\Cl(D^+/B)$. 
Thus, from \eqref{mar26.2} and standard formulas for the supertrace of 
Clifford algebras (see for instance formulas (1.23) and (1.24) in 
\cite{BC0}), we see that  as $k$ approaches $\bhs{\phi\phi,2}$, $\delta_B^{t} \str(k)$ approaches
\begin{equation}\label{tocheck}
\begin{gathered}
	(-2i)^{\frac{h+1}{2}} \mathrm{Ev}_{h+1}^{D^+/B}
	\lrspar{ \str'N^G_{\phi\phi,2}\lrpar{ \delta_B^\tau e^{-t\wt \bbA_{\fD}^2}} },\quad \mbox{for $h$ odd},  \\
	(-2i)^{\frac{h}{2}} \mathrm{Ev}_{h+1}^{D^+/B}
	\lrspar{ \str'N^G_{\phi\phi,2}\lrpar{ \delta_B^\tau e^{-t\wt \bbA_{\fD}^2}} },\quad \mbox{for $h$ even}.	
	\end{gathered}
\end{equation}
Now however, the twisting term is more interesting.  Starting from \eqref{tocheck} and proceeding essentially as in p.98 of \cite{Vaillant}, we get
\begin{equation*}
\begin{gathered}
	\frac1{\lrpar{2\pi i}^{\frac{h+1}{2}}} 
	\mathrm{Ev}_{h+1}^{D^+/B}\lrspar{ \delta_B^\tau \hat A\lrpar{\tau R^{D^+/B}}
	 \str_{E/D^{+}}'\lrpar{e^{-\tau \wt \bbB_{\pa M/D,+}^2 } } }, \quad \mbox{for $h$ odd}, \\
	\frac1{\lrpar{2\pi i}^{\frac{h}{2}}2\sqrt{\pi}} 
	\mathrm{Ev}_{h+1}^{D^+/B}\lrspar{ \delta_B^\tau \hat A\lrpar{\tau R^{D^+/B}}
	 \str_{E/D^{+}}'\lrpar{e^{-\tau \wt \bbB_{\pa M/D,+}^2 } } }, \quad \mbox{for $h$ even},
\end{gathered}
\end{equation*}
where we recall that
\begin{equation*}
	\wt \bbB_{\pa M/D,+}^2 = \wt\bbB_{\pa M/D,\fD}^2 
	+\df e\lrpar{\frac{dx}x}
	\lrspar{\eth^V_{\fD} + \hat{A}_P^\eps - m_0^{\wt\phi}\lrpar{B_{\wt\phi}} 
	- \frac14 m_0^{\wt\phi}\lrpar{\Curv{\wt\phi-hc}} }.
\end{equation*}
Notice that $\delta_B^\tau \hat A\lrpar{\tau R^{D^+/B}} =  \hat A\lrpar{D/B}\delta_{D^+}^\tau$ and, since $\df e\lrpar{\frac{dx}x}$ commutes with $\hat A\lrpar{D/B}$ (as noted after Proposition \ref{RescaledNormalOp}), that the 
$\str'$ factor can be replaced by the $\str'$ of the part of $\delta_{D^+}^\tau e^{-\tau \cA }$ that contains $\df e\lrpar{\frac{dx}x}$ which is
\begin{equation*}
	\frac1{\sqrt\tau}
	\df e\lrpar{\frac{dx}x}
	\lrspar{\eth^V_{\fD} + A_P^\eps - \frac1{\sqrt\tau} m_0^{\wt\phi}\lrpar{B_{\wt\phi}} 
	- \frac1{4\tau} m_0^{\wt\phi}\lrpar{\Curv{\wt\phi-hc}} }
		e^{-(\wt\bbB_{\pa M/D,\fD}^\tau)^2}.
\end{equation*}
Let us briefly refer to this term as $\df e(\frac{dx}x)\cA$. The contribution from $\bhs{\phi\phi,2}$ is thus
\begin{equation*}
	\frac1{\lrpar{2\pi i}^{\frac{h+1}{2}}}\int_{D/B}  
	\hat A\lrpar{D/B} 
	\int_0^{\infty}
	\str_{E/D^{+}}'\left( \textstyle \df e(\frac{dx}x) \displaystyle \cA\right) \;d\tau
\end{equation*}
when $h$ is odd and 
\begin{equation*}
	\frac1{\lrpar{2\pi i}^{\frac{h}{2}}2\sqrt{\pi}}\int_{D/B}  
	\hat A\lrpar{D/B} 
	\int_0^{\infty}
	\str_{E/D^{+}}'\left( \textstyle \df e(\frac{dx}x) \displaystyle \cA\right) \;d\tau
\end{equation*}
when $h$ is even.  In both cases, this can be rewritten as
\[
     \frac{1}{(2\pi i)^{\lfloor \frac{h+1}{2}\rfloor}} \int_{D/B}
\hat A\lrpar{D/B} \hat\eta^+\lrpar{\eth^V_{\fD},E}
\]
where $\hat\eta^+\lrpar{\eth^V_{\fD},E}$ is the odd eta form when
$h$ is odd,
\begin{equation}
\hat\eta^+\lrpar{\eth^V_{\fD},E}=  \int_{0}^{\infty} \str_{E/D^{+}}'\left( \textstyle \df e(\frac{dx}x) \displaystyle \cA\right) \;d\tau
\label{eta_odd}\end{equation}
and is the even eta form when $h$ is even,
\begin{equation}
\hat\eta^+\lrpar{\eth^V_{\fD},E}= \frac{1}{2\sqrt{\pi}} \int_{0}^{\infty} 
\str_{E/D^{+}}'\left( \textstyle \df e(\frac{dx}x) \displaystyle \cA\right) \;d\tau.
\label{eta_even}\end{equation}

\end{proof}

{\bf Remark} When we are in the product-type situation (\ie $\cB_{\wt\phi-hc}=0)$ and there is no Fredholm perturbation, our conventions are such that, when $h$ is odd, $\hat\eta^+\lrpar{\eth^V_{\fD},E}$ corresponds to the 
odd  eta form of Bismut and Cheeger associated to the family $\eth^V_{\fD}$.  
When $h$ is even, the relation with the eta form of Bismut
and Cheeger is as follows.  On the boundary, the map
\[
                  \begin{array}{lcl}
                      T^{*}(\pa M/D) & \to & \Cl(M/D) \\
                           \xi & \mapsto & \frac{dx}{x}\cdot \xi 
                  \end{array}
\]
induces an isomorphism of algebra $\Cl(\pa M/D)\cong \Cl^{0}(M/D)$.  Thus
on $\pa M$, $E^{0}:= \left. E^{+}\right|_{\pa M}$ is naturally a $\Cl(\pa M/D)$
Clifford module.  Moreover, Clifford multiplication by $\textstyle \df e(\frac{dx}x)$ naturally gives an identification 
\[
\begin{array}{lcl}
                      \left. E^{-}\right|_{\pa M} & \to & E^{0} \\
                           \xi & \mapsto & \frac{dx}{x}\cdot \xi. 
                  \end{array}
\]
Under the identification
\[
           \left. E\right|_{\pa M}= \left. E^{+}\right|_{\pa M}\oplus \left. E^{-}\right|_{\pa M}\cong E^{0}\oplus E^{0},
\]
the vertical family $\eth^{V,+}_{\fD}$ takes the form
\[
                 \eth^{V,+}_{\fD}= \left( 
                      \begin{array}{cc}
                           0 & \eth^{0} \\
                       \eth^{0} & 0
                       \end{array}
 \right)
\]
where $\eth^{0}$ is the family of Dirac type operators associated to the
$\Cl(\pa M/D)$ Clifford module $E^{0}$ and its induced Clifford connection.
When $h$ is even, our conventions precisely insure that $\hat{\eta}^{+}(\eth^{V}_{\fD},E)$ is the odd Bismut-Cheeger eta form associated to the family 
$\eth^{0}$.

Integrating the formulas in Lemma \ref{lem:dt} from $t=0$ to $t=\infty$ and using Lemmas \ref{lem:Bigt} and \ref{lem:Smallt} to evaluate the limits, we get the $\fD$ families index theorem.

\begin{theorem}\label{FamiliesIndex} $ $
\begin{itemize}
\item [i)]
If $\eth^V_{\fD}$ has kernel of constant rank and $\eth^{H,\cK}_{\fD}$ is invertible, then 
the Chern character of the (unperturbed) Bismut superconnection satisfies the following equation at the level of forms
\begin{multline}\label{UnperturbedIndex}
\Ch\lrpar{ \Ind\lrpar{\eth_{\fD}}, \nabla^{\Ind} }
= \frac1{\lrpar{2\pi i}^{n/2}} \int_{M/B} \hat A\lrpar{M/B}\Ch'\lrpar E \\
- \frac1{\lrpar{2\pi i}^{\lfloor\frac{h+1}{2}\rfloor}} \int_{D/B} \hat A\lrpar{D/B}\hat\eta^+\lrpar{ \eth^V_{\fD} } 
- \hat\eta\lrpar{ \eth^{H,\cK}_{\fD} } \\
- d \int_0^{\infty} \Str\lrpar{ \frac{\pa}{\pa t} \bbA_{\fD}^t e^{-\lrpar{\bbA_{\fD}^t}^2} }.
\end{multline}

\item [ii)]
If $\eth^V_{\fD}$ has kernel of constant rank and $Q$ is a spectral section for $\eth^{H,\cK}_{\fD}$ then, in $H^{\ev}B$, we have
\begin{multline}\label{SmallPerturbedIndex}
\Ch\lrpar{ \Ind_Q\lrpar{\eth_{\fD}} }
= \frac1{\lrpar{2\pi i}^{n/2}} \int_{M/B} \hat A\lrpar{M/B}\Ch'\lrpar E  \\ 
- \frac1{\lrpar{2\pi i}^{\lfloor\frac{h+1}{2}\rfloor}} \int_{D/B} \hat A\lrpar{D/B}\hat\eta^+\lrpar{ \eth^V_{\fD} }
- \hat\eta_Q\lrpar{ \eth^{H,\cK}_{\fD} }.
\end{multline}

\item [iii)]
If $P$ is a spectral section for $\eth^V_{\fD}$ then the following equation in $H^\ev B$
\begin{multline}\label{PerturbedIndex}
\Ch\lrpar{ \Ind_P\lrpar{\eth_{\fD}} } 
= \frac1{\lrpar{2\pi i}^{n/2}} \int_{M/B} \hat A\lrpar{M/B}\Ch'\lrpar E \\
- \frac1{\lrpar{2\pi i}^{\lfloor\frac{h+1}{2}\rfloor}} \int_{D/B} \hat A\lrpar{D/B}\hat\eta^+_P\lrpar{ \eth^V_{\fD} }.
\end{multline}
\end{itemize}
\label{index_theorem}\end{theorem}

{\bf Remark} 
The choices made in building a Fredholm perturbation affect the $\hat\eta$ invariants by adding closed forms. In some cases it can be worthwhile to keep track of these forms \cite{Piazza}, but, as we do not, we have stated parts $(ii)$ and $(iii)$ of the theorem as equalities in cohomology.

{\bf Remark}
Notice that if $P_1$ and $P_2$ are both spectral sections for $\eth^V_{\fD}$ satisfying \eqref{SpecSec} for functions $R_1, R_2 : D \to \bbR^+$, then at any point $d \in D$ we have
\begin{gather*}
	\Ran P_1 \cup (\Ran P_2)^\perp \subseteq
	\bplus_{ - R_1(d) \leq \lambda \leq R_2(d) } E_\lambda, \\
	(\Ran P_1 + (\Ran P_2)^\perp )^\perp \subseteq
	\bplus_{ - R_2(d) \leq \lambda \leq R_1(d) } E_\lambda,
\end{gather*}
hence both of these spaces are finite dimensional. Thus the spaces $\Ran P_1$ and $(\Ran P_2)^\perp$ form a {\em Fredholm pair} of subspaces of $L^2$ at every point $d\in D$ and define an element of $K(D)$,
\begin{equation*}
	[P_1,P_2] := 
	[\Ran P_1 \cup (\Ran P_2)^\perp] -
	\lrspar{ (\Ran P_1 + (\Ran P_2)^\perp )^\perp }.
\end{equation*}
Melrose and Piazza \cite[Proposition 17]{MelrosePiazza} proved an Agranovich-Dynin type formula
\begin{equation*}
	\hat\eta_{P_1} - \hat\eta_{P_2} = \Ch( [P_1, P_2] ) \text{ in } H^\ev D,
\end{equation*}
and applying this to $(ii)$ and $(iii)$ above yields relative index theorems.

{\bf Remark} 
A family of operators with constant rank kernel has a canonical choice of spectral section, $\cP_{\geq}$, the projection onto the eigenspaces with positive eigenvalues. As explained in \cite[\S 16]{MelrosePiazza} in this case the usual formula for the $\hat\eta$ invariant converges and differs from $\hat\eta_{\cP_\geq}$ by the Chern character of the kernel. Thus when $\eth^{H,\cK}_{\fD}$ has constant rank kernel we get from $(ii)$,
\begin{multline*}
\Ch\lrpar{ \Ind_{\cP_{\geq}(\eth^{H,\cK}_{\fD})}\lrpar{\eth_{\fD}} }
= \frac1{\lrpar{2\pi i}^{n/2}} \int_{M/B} \hat A\lrpar{M/B}\Ch'\lrpar E  \\ 
- \frac1{\lrpar{2\pi i}^{\lfloor\frac{h+1}{2}\rfloor}} \int_{D/B} \hat A\lrpar{D/B}\hat\eta^+\lrpar{ \eth^V_{\fD} }
- \lrpar{ \hat\eta\lrpar{ \eth^{H,\cK}_{\fD} } + \frac12\Ch(\ker \eth^{H,\cK}_{\fD}) }.
\end{multline*}
However, when $\eth^V_{\fD}$ has constant rank kernel, $\cP_{\geq}$ is not a $\Cl(T^*D^+/B)$ spectral section.

\section{The $\fC$ families index theorem}

In \cite{Moroianu}, Moroianu deduces the $\fC$ index theorem for a single Fredholm generalized Dirac operator by reducing it to the $\fD$ index theorem for the conformally associated generalized Dirac operator.
As explained in \S\ref{sec:ConfRelMet}, we can extend Moroianu's argument to families of operators as long as the dimension of the kernel and cokernel are constants independent of the base point. Since we can always achieve this by a compact perturbation we get the following theorem:

\begin{corollary}\label{PhiFamInd}
Let $\eth_0$ be a family of $\fC$ Dirac-type operators such that $\eth^V_{\fD}$ admits a spectral projection, $P$, and let $\eth_p$ be as in \S\ref{sec:ConfRelMet}, then the following equation holds in $H^\ev B$
\begin{multline*}
\Ch\lrpar{ \Ind_P(\eth_p) } 
= \Ch\lrpar{ \Ind_P(\eth_{\fD}) }
\\= \frac1{\lrpar{2\pi i}^{n/2}} \int_{M/B} \hat A\lrpar{M/B}\Ch'\lrpar E 
- \frac1{\lrpar{2\pi i}^{\lfloor\frac{h+1}{2}\rfloor}} \int_{D/B} \hat A\lrpar{D/B}\hat\eta^+_P\lrpar{ \eth^V_{\fD} }.
\end{multline*}
\end{corollary}


\end{document}